\documentclass{amsproc}
\usepackage{amssymb, amsmath, amsthm, amsfonts}
\usepackage{mathrsfs,comment}
\usepackage{graphicx}
\usepackage{placeins}
\usepackage{scalefnt}
\usepackage{todonotes}
\usepackage{rotating}
\usepackage{tikz}
\usepackage{float}
\usepackage{subfigure}
\usepackage{float}
\usepackage{caption}
\usepackage{pdflscape}

\usepackage{hyperref}  
\usepackage{url}
\usepackage[all,arc,2cell]{xy}
\UseAllTwocells
\usepackage{enumerate}

\usepackage{chngcntr}
 \usepackage{lineno}
 \usepackage{blindtext}
\usepackage{verbatim}
\usepackage{soul}
\usepackage{sseq}
\usepackage[normalem]{ulem}
\newcommand{\stkout}[1]{\ifmmode\text{\sout{\ensuremath{#1}}}\else\sout{#1}\fi}

    \hypersetup{%
    bookmarksnumbered=true,%
    bookmarks=true,%
    colorlinks=true,%
    linkcolor=blue,%
    citecolor=blue,%
    filecolor=blue,%
    menucolor=blue,%
    pagecolor=blue,%
    urlcolor=blue,%
    pdfnewwindow=true,%
    pdfstartview=FitBH}

\def\@url#1{{\tt\def~{\lower3.5pt\hbox{\char'176}}\def\_{\char'137}#1}}

\let\fullref\autoref
\def\makeautorefname#1#2{\expandafter\def\csname#1autorefname\endcsname{#2}}
\makeautorefname{equation}{Equation}%
\makeautorefname{footnote}{footnote}%
\makeautorefname{item}{item}%
\makeautorefname{figure}{Figure}%
\makeautorefname{table}{Table}%
\makeautorefname{part}{Part}%
\makeautorefname{appendix}{Appendix}%
\makeautorefname{chapter}{Chapter}%
\makeautorefname{section}{Section}%
\makeautorefname{subsection}{Section}%
\makeautorefname{subsubsection}{Section}%
\makeautorefname{paragraph}{Paragraph}%
\makeautorefname{subparagraph}{Paragraph}%
\makeautorefname{theorem}{Theorem}%
\makeautorefname{theo}{Theorem}%
\makeautorefname{thm}{Theorem}%
\makeautorefname{addendum}{Addendum}%
\makeautorefname{addend}{Addendum}%
\makeautorefname{add}{Addendum}%
\makeautorefname{maintheorem}{Main theorem}%
\makeautorefname{mainthm}{Main theorem}%
\makeautorefname{corollary}{Corollary}%
\makeautorefname{claim}{Claim}%
\makeautorefname{corol}{Corollary}%
\makeautorefname{coro}{Corollary}%
\makeautorefname{cor}{Corollary}%
\makeautorefname{lemma}{Lemma}%
\makeautorefname{lemm}{Lemma}%
\makeautorefname{lem}{Lemma}%
\makeautorefname{sublemma}{Sublemma}%
\makeautorefname{sublem}{Sublemma}%
\makeautorefname{subl}{Sublemma}%
\makeautorefname{proposition}{Proposition}%
\makeautorefname{proposit}{Proposition}%
\makeautorefname{propos}{Proposition}%
\makeautorefname{propo}{Proposition}%
\makeautorefname{prop}{Proposition}%
\makeautorefname{property}{Property}
\makeautorefname{proper}{Property}
\makeautorefname{scholium}{Scholium}%
\makeautorefname{step}{Step}%
\makeautorefname{conjecture}{Conjecture}%
\makeautorefname{conject}{Conjecture}%
\makeautorefname{conj}{Conjecture}%
\makeautorefname{question}{Question}
\makeautorefname{questn}{Question}
\makeautorefname{quest}{Question}
\makeautorefname{ques}{Question}
\makeautorefname{qn}{Question}
\makeautorefname{definition}{Definition}%
\makeautorefname{defin}{Definition}%
\makeautorefname{defi}{Definition}%
\makeautorefname{def}{Definition}%
\makeautorefname{defn}{Definition}%
\makeautorefname{dfn}{Definition}%
\makeautorefname{note}{Note}
\makeautorefname{notation}{Notation}
\makeautorefname{nota}{Notation}
\makeautorefname{notn}{Notation}
\makeautorefname{remark}{Remark}%
\makeautorefname{rema}{Remark}%
\makeautorefname{rem}{Remark}%
\makeautorefname{rmk}{Remark}%
\makeautorefname{rk}{Remark}%
\makeautorefname{remarks}{Remarks}%
\makeautorefname{rems}{Remarks}%
\makeautorefname{rmks}{Remarks}%
\makeautorefname{rks}{Remarks}%
\makeautorefname{example}{Example}%
\makeautorefname{examp}{Example}%
\makeautorefname{exmp}{Example}%
\makeautorefname{exmps}{Examples}%
\makeautorefname{exam}{Example}%
\makeautorefname{exa}{Example}%
\makeautorefname{ex}{Example}%
\makeautorefname{algorithm}{Algorith}%
\makeautorefname{algo}{Algorith}%
\makeautorefname{alg}{Algorith}%
\makeautorefname{axiom}{Axiom}%
\makeautorefname{axi}{Axiom}%
\makeautorefname{ax}{Axiom}%
\makeautorefname{case}{Case}%
\makeautorefname{claim}{Claim}%
\makeautorefname{clm}{Claim}%
\makeautorefname{assumption}{Assumption}%
\makeautorefname{assumpt}{Assumption}%
\makeautorefname{conclusion}{Conclusion}%
\makeautorefname{concl}{Conclusion}%
\makeautorefname{conc}{Conclusion}%
\makeautorefname{condition}{Condition}%
\makeautorefname{condit}{Condition}%
\makeautorefname{cond}{Condition}%
\makeautorefname{construction}{Construction}%
\makeautorefname{construct}{Construction}%
\makeautorefname{const}{Construction}%
\makeautorefname{cons}{Construction}%
\makeautorefname{criterion}{Criterion}%
\makeautorefname{criter}{Criterion}%
\makeautorefname{crit}{Criterion}%
\makeautorefname{exercise}{Exercise}%
\makeautorefname{exer}{Exercise}%
\makeautorefname{exe}{Exercise}%
\makeautorefname{exc}{Exercise}%
\makeautorefname{problem}{Problem}%
\makeautorefname{problm}{Problem}%
\makeautorefname{probm}{Problem}%
\makeautorefname{prob}{Problem}%
\makeautorefname{solution}{Solution}%
\makeautorefname{soln}{Solution}%
\makeautorefname{sol}{Solution}%
\makeautorefname{summary}{Summary}%
\makeautorefname{summ}{Summary}%
\makeautorefname{sum}{Summary}%
\makeautorefname{operation}{Operation}%
\makeautorefname{oper}{Operation}%
\makeautorefname{observation}{Observation}%
\makeautorefname{observn}{Observation}%
\makeautorefname{obser}{Observation}%
\makeautorefname{obs}{Observation}%
\makeautorefname{ob}{Observation}%
\makeautorefname{convention}{Convention}%
\makeautorefname{convent}{Convention}%
\makeautorefname{conv}{Convention}%
\makeautorefname{cvn}{Convention}%
\makeautorefname{warning}{Warning}%
\makeautorefname{warn}{Warning}%
\makeautorefname{note}{Note}%
\makeautorefname{fact}{Fact}%
\makeautorefname{thmbig}{Theorem}%
\makeautorefname{conjbig}{Conjecture}%
\makeautorefname{analogy}{Analogy}%

  \makeatletter
                   \let\c@lemma\c@theorem
                  \makeatother

\newtheorem{thm}{Theorem}[subsection]
\newtheorem{cor}{Corollary}[subsection]
\newtheorem{prop}{Proposition}[subsection]
\newtheorem{lem}{Lemma}[subsection]

\theoremstyle{definition}
\newtheorem{exc}{Exercise}[subsection]
\newtheorem{defn}{Definition}[subsection]

\newtheorem{notation}{Notation}[subsection]

\newtheorem{rem}{Remark}[subsection]
\newtheorem{warn}{Warning}[subsection]

\newtheorem{assumption}{Assumption}[subsection]

\newtheorem{ex}{Example}[subsection]
\newtheorem{analogy}{Analogy}[subsection]

\makeatletter
\let\c@lem=\c@thm
\let\c@note=\c@thm
\let\c@analogy=\c@thm
\let\c@exc=\c@thm
\let\c@notation=\c@thm
\let\c@cor=\c@thm
\let\c@prop=\c@thm
\let\c@lem=\c@thm
\let\c@defn=\c@thm
\let\c@exmps=\c@thm
\let\c@rem=\c@thm
\let\c@warn=\c@thm
\let\c@claim=\c@thm
\let\c@quest=\c@thm
\let\c@exc=\c@thm
\let\c@ex=\c@thm
\let\c@assumption=\c@thm
\makeatother

\numberwithin{equation}{subsection}

\def\quickop#1{\expandafter\newcommand\csname #1\endcsname{\operatorname{#1}}}
\quickop{Hom}\quickop{RHom} \quickop{Maps} \quickop{End} \quickop{Aut} \quickop{Tel} \quickop{Mic} \quickop{map}
\quickop{Ext} \quickop{Tor} \quickop{Cotor} \quickop{Id} \quickop{Coker} \quickop{Ker}
\quickop{Lim} \quickop{Colim} \quickop{Holim} \quickop{Hocolim}
\quickop{id} \quickop{tel} \quickop{mic} \quickop{coker}
\quickop{colim} \quickop{holim} \quickop{hocolim} \quickop{im}

\numberwithin{equation}{subsection}

\newcommand{\A}{\mathcal{A}}
\newcommand{\cB}{\mathcal{B}}
\newcommand{\F}{\mathbb{F}}
\newcommand{\R}{\mathbb{R}}

\newcommand{\N}{\mathbb{N}}
\newcommand{\C}{\mathbb{C}}
\newcommand{\E}{\mathcal{E}}
\newcommand{\Q}{\mathbb{Q}}
\newcommand{\smsh}{\wedge}
\newcommand{\ra}{\rightarrow}
\newcommand{\xra}{\xrightarrow}

\newcommand{\ox}{\otimes}

\DeclareFontFamily{OMS}{rsfs}{\skewchar\font'60}
\DeclareFontShape{OMS}{rsfs}{m}{n}{<-5>rsfs5 <5-7>rsfs7 <7->rsfs10 }{}
\DeclareSymbolFont{rsfs}{OMS}{rsfs}{m}{n}
\DeclareSymbolFontAlphabet{\scr}{rsfs}

\newcommand{\Top}{\mathrm{Top}}
\newcommand{\HF}{{H}\mathbb{Z}/2}
\newcommand{\bHF}{\overline{{H}}\mathbb{Z}/2}
\newcommand{\mc}[1]{\mathcal{#1}}

\def\makeop#1{\expandafter\def\csname #1\endcsname{\mathop{\mathrm{#1}}\nolimits}}

\makeop{tmf}
\makeop{Ab}
\makeop{Gal}
\makeop{id}
\makeop{Mod}
\makeop{Tot}
\makeop{gr}
\makeop{Out}
\makeop{Hom}
\makeop{Ext}
\makeop{End}
\makeop{Aut}
\makeop{Tor}
\makeop{ev}
\makeop{sq}
\makeop{ko}
\makeop{ku}
\makeop{Vect}
\makeop{MSpin}
\makeop{MPin}
\makeop{Spin}
\makeop{Pin}
\makeop{Thom}
\def\Z{{{\mathbb{Z}}}}
\def\Ext{\mathrm{Ext}}

\newcommand{\U}{\overline{U}}
\newcommand{\w}{\overline{w}}

\def\RPn (#1,#2){
  \fill (#1, #2) circle (3pt);
  \fill (#1, #2+1) circle (3pt);
}

\def\sqtwoL (#1,#2,#3){
  \draw[#3] (#1,#2) .. controls (#1-1,#2+1) .. (#1,#2+2);
}

\def\sqtwoR (#1,#2,#3){
  \draw[#3] (#1,#2) .. controls (#1+1,#2+1) .. (#1,#2+2);
}

\def \sqtwoCR (#1,#2,#3){
   \draw[#3] (#1,#2) .. controls (#1+1,#2+.5) and (#1+1.5,#2+2) .. (#1+2,#2+2);
}

\def \sqtwoCL (#1,#2,#3){
   \draw[#3] (#1,#2) .. controls (#1-1,#2+.5) and (#1-1.5,#2+2)  .. (#1-2,#2+2);
}

\def \sqone (#1,#2,#3){
  \draw[#3] (#1,#2) -- (#1,#2+1);
}

\def\Aone (#1,#2){
\fill (#1, #2) circle (3pt);
\fill (#1, #2+1) circle (3pt);
\fill (#1, #2+2) circle (3pt);
\fill (#1, #2+3) circle (3pt);
\fill (#1+2, #2+3) circle (3pt);
\fill (#1+2, #2+4) circle (3pt);
\fill (#1+2, #2+5) circle (3pt);
\fill (#1+2, #2+6) circle (3pt);
\draw (#1, #2) -- (#1, #2+1);
\draw (#1, #2+2) -- (#1, #2+3);
\draw (#1+2, #2+3) -- (#1 + 2, #2+4);
\draw (#1+2, #2+5) -- (#1+2, #2+6);
\draw (#1, #2) .. controls (#1-1, #2+1) .. (#1, #2+2);
\draw (#1+2, #2+4) .. controls (#1+3, #2+5) .. (#1+2, #2+6);
\draw (#1, #2+1) .. controls (#1+1, #2+1.5) and  (#1+1.5 ,#2+3) .. (#1+2,#2+3);
\draw (#1, #2+2) .. controls (#1+1, #2+2.5) and (#1+1.5, #2+4) .. (#1+2, #2+4);
\draw (#1, #2+3) .. controls (#1+1, #2+3.5) and (#1+1.5, #2+5) .. (#1+2, #2+5);
}

\def\rectangle (#1,#2,#3){   \draw[#3] (#1-0.15,#2-0.15) rectangle (#1+0.15,#2+0.15)}

\def\Eone (#1,#2){
\fill (#1, #2) circle (3pt);
\fill (#1, #2+1) circle (3pt);
\fill (#1, #2+2) circle (3pt);
\fill (#1, #2+3) circle (3pt);
\draw (#1, #2) -- (#1, #2+1);
\draw (#1, #2+2) -- (#1, #2+3);
\draw (#1, #2) .. controls (#1-1, #2+1) .. (#1, #2+2);
\draw (#1, #2+1) .. controls (#1+1, #2+2) .. (#1, #2+3);
}

\def\joker (#1,#2){
  \foreach \y in {#2, #2+1, #2+2, #2+3, #2+4}
           {\fill (#1,\y) circle (3pt);}
           \draw (#1,#2) -- (#1, #2+1);
           \draw (#1,#2+3) -- (#1, #2+4);
           \draw (#1,#2+0) .. controls (#1-1,#2+1) .. (#1, #2+2);
           \draw (#1,#2+2) .. controls (#1-1,#2+3) .. (#1, #2+4);
           \draw (#1,#2+1) .. controls (#1+1,#2+2) .. (#1, #2+3);
}

\def\jokercolor (#1,#2, #3){
  \foreach \y in {#2, #2+1, #2+2, #2+3, #2+4}
           {\fill[#3] (#1,\y) circle (3pt);}
           \draw[#3] (#1,#2) -- (#1, #2+1);
           \draw[#3] (#1,#2+3) -- (#1, #2+4);
           \draw[#3] (#1,#2+0) .. controls (#1-1,#2+1) .. (#1, #2+2);
           \draw[#3] (#1,#2+2) .. controls (#1-1,#2+3) .. (#1, #2+4);
           \draw[#3] (#1,#2+1) .. controls (#1+1,#2+2) .. (#1, #2+3);
}

\def\msopart (#1,#2,#3){
    \fill[#3] (#1,#2) circle (3pt); 
      \fill[#3] (#1, #2+1) circle (3pt);
      \fill[#3] (#1, #2+2) circle (3pt);
      \fill[#3] (#1+2, #2+2) circle (3pt);
      \fill[#3] (#1+2, #2+3) circle (3pt);
      \fill[#3] (#1+2, #2+4) circle (3pt);
      \fill[#3] (#1+2, #2+5) circle (3pt);
    \sqtwoCR(#1,#2, #3);
    \sqtwoCR (#1, #2+1, #3);
    \sqtwoCR (#1, #2+2, #3);
    \sqone (#1+2, #2+2, #3);
    \sqone (#1+2, #2+4, #3);
    \sqtwoR(#1+2, #2+3, #3);
    \sqone (#1, #2+1, #3); }

    \def\amme (#1,#2,#3){
       \fill[#3] (#1,#2) circle (3pt) ;
   \sqtwoR(#1,#2,#3);
      \fill[#3] (#1,#2+2) circle (3pt) ;
         \sqone(#1,#2+2,#3);
      \fill[#3] (#1,#2+3) circle (3pt) ;
   \sqtwoR(#1,#2+3,#3);
   \fill[#3] (#1,#2+5) circle (3pt) ;}
   
    \def\questionupsidedon (#1,#2,#3){
       \fill[#3] (#1,#2) circle (3pt) ;
          \sqtwoR(#1,#2,#3);
                 \fill[#3] (#1,#2+2) circle (3pt) ;
             \sqone(#1,#2+2,#3);
                \fill[#3] (#1,#2+3) circle (3pt) ;
}

\definecolor{darkspringgreen}{rgb}{0.09, 0.45, 0.27}
\begin{document}

\title{A Guide for Computing Stable Homotopy Groups}


%    author one information

\author[A. Beaudry]{Agn\`es Beaudry}
\address{Department of Mathematics\\ University of Colorado at Boulder \\ \newline Campus Box 395 \\ Boulder \\ Colorado \\ 80309}
%\curraddr{}
%\email{}
%\thanks{}

%    author two information
\author[J. Campbell]{Jonathan A. Campbell}
\address{Department of Mathematics \\ Vanderbilt University \\ \newline 1326 Stevenson Center \\ Nashville \\ Tennessee \\ 37240}
%\curraddr{}
%\email{j.campbell at vanderbilt.edu}
%\thanks{}

%\subjclass[2000]{Primary}
%    The 2010 edition of the Mathematics Subject Classification is
%    now available.  If you are citing a classification from the
%    new scheme, use the following input coding instead.
%\subjclass[2010]{Primary }

\date{\today}

\begin{abstract}
This paper contains an overview of background from stable homotopy theory used by Freed--Hopkins in their work on invertible extended topological field theories. We provide a working guide to the stable homotopy category, to the Steenrod algebra and to computations using the Adams spectral sequence. Many examples are worked out in detail to illustrate the techniques.
\end{abstract}

\maketitle

\setcounter{tocdepth}{1}
\tableofcontents

% !TEX root = cbms-master.tex

\section{Introduction and organization}

\subsection{Introduction} 
The main theorem of \cite{FH} states that deformation classes of reflection positive invertible $n$-dimensional extended topological field theories with symmetry group $H_n$ are classified by the torsion in
\[ [MTH, \Sigma^{n+1}I_{\Z}]  . \]
Here, $MTH$ is the Madsen--Tillmann spectrum associated to a group $H$ which is a stabilization of $H_n$, $I_{\Z}$ is the Anderson dual of the sphere spectrum, and $[-, -]$ denotes the \emph{stable} homotopy classes of maps. These concepts will be discussed in \fullref{sec:spectra}. 

In order to complete the classification problem, it is necessary to be able to compute stable homotopy classes of maps from a spectrum $X$ to $I_{\Z}$. This problem can be reduced to the computation of the stable homotopy groups of $X$ itself as will be described in \fullref{sec:andersondual}. 

In general, it is notoriously difficult, if not impossible, to completely compute the homotopy groups of a spectrum $X$. However, homotopy theorists are very good at doing these computations in small ranges and the problems motivated by physics only require information in small dimensions, making us a perfect match.

The main tool used to compute low-dimensional homotopy groups of spectra is the Adams spectral sequence. Adams initially introduced this spectral sequence in order to resolve the Hopf invariant one problem \cite{adams_cohomology, adams_hopf_inv_one}. It has been a standard tool in homotopy theory since then. In brief, the Adams spectral sequence takes in information about the cohomology of a space or spectrum and outputs information about its stable homotopy groups.

The Steenrod algebra, which we denote by $\A$, is one of the classical structures in homotopy theory. The mod-$2$ cohomology $H^\ast (X;\Z/2)$ of any space or spectrum $X$ is a module over $\A$. This module is the input to the Adams spectral sequence. Although it can be difficult to compute the $\A$-module structure of the cohomology of an arbitrary space or spectrum $X$, 
we work under the favorable circumstance that the examples we consider are related to the classifying spaces of various Lie groups. With the $\A$-module structure of $H^*(X;\Z/2)$ in hand, and some knowledge of homological algebra over the Steenrod algebra, the $E_2$-page of the Adams spectral sequence can be computed. In the low-dimensional range, we are lucky, and every example we consider is fully computable by hand. 

The aim of this paper is to introduce the reader to enough of the machinery of spectra, the Steenrod algebra and the Adams spectral sequence to understand the computation of the homotopy groups $\pi_\ast MTH$. To illustrate how one applies the theory, we do the computations for a few examples. In particular, we go over the cases when $H$ is $\Spin^c$ and $\Pin^c$ in detail, an exercise which was left to the reader in Section 10 of \cite{FH} and was not covered in \cite{campbell}.

\subsection{Organization}

In order to fully explicate the computations for readers unfamiliar with stable homotopy theory, we include an introduction to spectra in \fullref{sec:spectra}. Among other topics, we discuss the category of spectra and its homotopy category (\fullref{sec:catspectra} and \fullref{sec:homotopycat}), the homotopy groups of spectra (\fullref{sec:homotopygroups}), the Anderson dual (\fullref{sec:andersondual}) and the construction of Thom spectra (\fullref{sec:thomspectra}). These latter are integral to the Freed--Hopkins classification since it is Thom spectra that are tightly linked with cobordism groups and the cobordism hypothesis.

In \fullref{sec:steenrod} we discuss the Steenrod algebra, $\mc{A}$, which is a non-commutative, infinitely generated algebra that acts on the cohomology of all spaces and spectra. In \fullref{sec:ans}, we introduce $\A_1$, an eight dimensional sub-algebra of $\A$ that will play a crucial role in the computations.
In \fullref{sec:compA1} we compute the $\A_1$-module structure for some examples of cobordism spectra. This computation depends on knowing how to determine the $\A_1$-module structure of the cohomology of classifying spaces, along with the Thom isomorphism and the Wu formula. These things are discussed in \fullref{sec:SWclasses}. 

The Adams spectral sequence is introduced in \fullref{sec:ass}. The primary tool for computation with the Adams spectral sequence is homological algebra over $\A$ and, in our examples, over $\A_1$. This section includes a discussion of resolutions (\fullref{sec:MinRes}) and computations of $\Ext_{\A_1}$ for a menagerie of $\A_1$-modules. It includes explanations of Adams charts (\fullref{sec:achart}), of certain mulitplicative structures on $\Ext$ (\fullref{sec:multiplicative}) and a variety of useful tricks. In \fullref{sec:assconstruction}, we formally construct the spectral sequence and in \fullref{sec:usingASS}, we provide a ``user's manual''. 

In \fullref{sec:examples}, we come to the main event. In the range $0 \leq n \leq 4$ we compute $\pi_n MTH(s)$ and $\pi_n MTH^c(s)$ in all of the cases that were not explained in further detail in \cite{campbell}. The computations rely on the Adams spectral sequence, and we use all of the material developed in \fullref{sec:steenrod} and \fullref{sec:ass} to compute the $E_2$-pages. In such a small range, and in these cases, the spectral sequences collapse and the homotopy groups can be read off of the Adams charts.

\subsection*{Acknowledgements} The authors thank Andy Baker, Prasit Bhattacharya, Bert Guillou and the referee for helpful comments.

% !TEX root = cbms-master.tex

\section{A working guide to spectra}\label{sec:spectra}
In this section, we give an introduction to spectra. If one is interested in computing homotopy groups, then one can often get away with an understanding of the properties of the homotopy category of spectra (see \fullref{sec:homotopycat}). Some of the information we include is not strictly necessary for this understanding, but we tried to strike a balance between too little and too much information. 

For a more in-depth introduction to spectra, a starting point would be Section 1.4 of Lurie \cite{lurie} and the introduction by Elmendorff--Kriz--Mandell--May to Chapter 6 of \cite{athandbook} (a book that contains other hidden gems).
One could then move on to Part III of Adams \cite{adams}, Chapter 10.9 of Weibel \cite{weibel} and Chapter 12 of Aguilar--Gitler--Carlos \cite{aguilar}. For serious treatments of different modern models of the category of spectra together with all of its structure, see the first parts of Schwede \cite{schwede}, Mandell--May--Schwede--Shipley \cite{mmss}, Elmendorff--Kriz--Mandell--May \cite{ekmm} or Lurie \cite{lurie}. For the equivariant treatment, see Lewis--May--Steinberger \cite{lms}.

\begin{notation}
We let $\Ab$ be the category of graded abelian groups. We let $\Top_*$ be a category of suitably nice based topological spaces with continuous maps that preserve the base points. 
\end{notation}

Motivation for the category of spectra comes from at least two directions. First, there is Brown's representability theorem that states that a cohomology theory $E^* \colon \Top_*^{\text{op}}  \to \operatorname{Ab}$ has a sequence of representing spaces $E_n$. That is, $E^n(X) \cong [X, E_n]$. We will let $\Sigma(-)$ be the reduced suspension and $\Omega(-)$ be the based loops functor. The isomorphism $[\Sigma X, E_n] \cong [X, \Omega E_n]$ together with the suspension isomorphism $E^n(X)  \cong E^{n+1}(\Sigma X)$ give rise to an isomorphism
\[\xymatrix{ [X, E_n]  \ar[r]^-{\cong} &     [ X, \Omega E_{n+1}]}\]
which is natural in $X$. By the Yoneda Lemma, this corresponds to a weak equivalence $\omega_n \colon E_n \xra{\simeq} \Omega E_{n+1}$. 
Further, to discuss natural transformations between cohomology theories, one is led to discuss maps between these sequences of spaces. It thus behooves us to construct a category which consists of sequences of spaces.

Another motivation is via Freudenthal's suspension theorem. Let $X$ be a $k$-connected topological space. Freudenthal's suspension theorem states that the map $\pi_n (X) \to \pi_{n+1} (\Sigma X)$ is an isomorphism if $n \leq 2k$. For a fixed $n$ and connected $X$, this implies that $\pi_{n+k} (\Sigma^k X)$ stabilizes as $k$ goes to infinity. This motivates the definition of the \emph{$n$th stable homotopy group}
\[\pi_n^s X = \colim_k \pi_{n+k} \Sigma^k X \cong \pi_{n+m} (\Sigma^m X) \ \ \ m \gg 0.  \]
An amazingly useful fact is that $\pi_*^s \colon \Top_* \to \Ab$ is a homology theory, making the stable homotopy groups often (slightly) more computable than the usual, \emph{unstable}, homotopy groups. It is useful to consider the sequences of spaces $\{\Sigma^n X\}$ as the fundamental objects, and we come again to a point where it is necessary to define some category of sequences of spaces. 

\bigskip

We will define the following five categories in the next few sections:
\begin{enumerate}[(1)]
\item The category of prespectra, denoted $\mathrm{PreSp}$. See \fullref{defn:prespectra}.
\item The category of spectra, denoted $\mathrm{Sp}$. See \fullref{defn:spectra}.
\item The category of CW-prespectra, denoted $\mathrm{CWPreSp}$. See \fullref{defn:CWprespectra}.
\item The category of CW-spectra, denoted $\mathrm{CWSp}$. See \fullref{defn:CWspectra}.
\item The homotopy category of spectra, denoted $\mathrm{hSp}$. See \fullref{sec:homotopycat}.
\end{enumerate}
The first four are a means to the fifth. We justify this complication by the following analogy inspired from Chapter 10 of \cite{weibel} and, in particular, Analogy 10.9.7. The reader can skip this analogy now and come back to it at the end of \fullref{sec:homotopycat}.

\begin{analogy}\label{analogy}
To justify having both spectra and prespectra, we make an analogy with the categories of sheaves and presheaves. Although we do homological algebra in the category of sheaves, some constructions are easier to make in the category of presheaves. The forgetful functor from sheaves to presheaves has a left adjoint, the sheafification functor. This allows one to transport constructions from presheaves to sheaves.

In this part of the analogy, spectra are the sheaves and prespectra are the presheaves. The analogue of the sheafification functor is called spectrification and is denoted $L \colon \mathrm{PreSp} \to \mathrm{Sp}$. It is the left adjoint to a forgetful functor from $\mathrm{Sp}$ to $\mathrm{PreSp}$. See \fullref{rem:functorL}.

Now, switching gears, we think of the category $\mathcal{C}$ of bounded below chain complexes of $R$-modules. There are two important kinds of equivalences in this category, the chain homotopy equivalences and the quasi-isomorphisms. The derived category ${D}(\mathcal{C})$ is characterized as the initial category which receives a functor $\mathcal{C} \to {D}(\mathcal{C})$ such that the quasi-isomorphisms are mapped to isomorphisms in ${D}(\mathcal{C})$. Chain homotopy equivalence is an equivalence relation, but the property of being quasi-isomorphic is not. In theory, it takes more work to invert the quasi-isomorphisms than it does to invert the chain homotopy equivalences. However, a quasi-isomorphism between bounded below projective chain complexes is a chain homotopy equivalence
and, further, any chain complex is quasi-isomorphic to a projective one.
Therefore, a model for ${D}(\mathcal{C})$ is the category whose objects are projective chain complexes and morphisms are chain homotopy equivalences of maps. 

In this part of the analogy, the topological spaces are the $R$-modules and the category of spectra is the analogue of $\mathcal{C}$. The chain homotopy equivalences correspond to the homotopy equivalences and the quasi-isomorphisms to the weak homotopy equivalences. The homotopy category of spectra is analogous to ${D}(\mathcal{C})$. The projective chain complexes are the analogues to CW-spectra, and a model for the homotopy category of spectra is the category of CW-spectra together with homotopy classes of maps between them. 

We have not mentioned CW-prespectra and use it to tie the knot between the two analogies: CW-prespectra are easy to define in prespectra, and the spectrification functor is used to transfer the definition to spectra.
\end{analogy}

\subsection{The categories}\label{sec:catspectra}

\begin{defn}[Prespectra]\label{defn:prespectra}
A \emph{prespectrum} $X$ is a sequence of spaces $ X_n \in \Top_*$ for $n\geq 0$ and continuous maps $\sigma_n \colon \Sigma X_n \to X_{n+1}$. We let $\omega_n \colon X_n \to \Omega X_{n+1}$ be the adjoint of $\sigma_n$ and note that giving the structure maps $\sigma_n$ of a prespectrum is equivalent to specifying the maps $\omega_n$. 
A map of prespectra $f \colon X \to Y$ of degree $r$ is a sequence of continuous, based maps $f_n \colon X_n \to Y_{n-r}$ such that the following diagram commutes:
\[\xymatrix{ \Sigma X_n \ar[r]^-{f_n} \ar[d]_{\sigma_n} &  \Sigma  Y_{n-r}  \ar[d]^{\sigma_{n-r}}  \\ 
X_{n+1} \ar[r]_-{f_{n+1}} & Y_{n+1-r}.}\]
We let $\mathrm{PreSp}$ denote the category of prespectra. (The plural of prespectrum is \emph{prespectra}.)
\end{defn}

\begin{rem}
If the maps $\omega_n$ are weak homotopy equivalences, then $X$ is often called an \emph{$\Omega$-prespectrum}.
\end{rem}

\begin{defn}[Spectra]\label{defn:spectra}
A prespectrum is called a \emph{spectrum} if the maps $\omega_n$ are homeomorphisms. We let $\mathrm{Sp}$ denote the full subcategory of prespectra generated by the objects which are spectra.
\end{defn}

\begin{defn}[CW-prespectra]\label{defn:CWprespectra}
We call a prespectrum a \emph{CW-prespectrum} if the spaces $X_n$ are CW-complexes and the maps $\Sigma X_n \to X_{n+1}$ are cellular inclusions. We let $\mathrm{CWPreSp}$ denote the full subcategory of prespectra generated by the objects which are CW-prespectra.
\end{defn}

\begin{ex}
  The standard example is the suspension prespectrum $\Sigma^{\infty}A$ of a based topological space $A$. Its $n$th space is given by $\Sigma^n A$ and the structure maps are identities $\Sigma \Sigma^n A \cong \Sigma^{n+1}A \to  \Sigma^{n+1}A$. In fact, this extends to a functor $\Sigma^{\infty}\colon \Top_* \to \mathrm{PreSp}$ which sends a space $A$ to $\Sigma^{\infty}A$. The functor $\Sigma^{\infty}$ is left adjoint to the functor  $\Omega^{\infty} \colon \mathrm{PreSp} \to \Top_*$ which sends a prespectrum to its zeroth space.
    \end{ex}

\begin{ex}
The Eilenberg-MacLane prespectrum $HG$, where $G$ is an abelian group, has $n$th space $K(G,n)$. The structure maps of $HG$ are the adjoints to the homotopy equivalences $\omega_n  \colon K(G,n) \to \Omega K(G,n+1)$. A homomorphism of abelian groups $G_1 \to G_2$ give rise to a map of prespectra $HG_1 \to HG_2$.
  \end{ex}

\begin{ex}
Another example is given by $K$-theory. The odd spaces of $K$ are the infinite unitary group $U$ and the even spaces are $\Z \times BU$, where $BU$ is the classifying space of $U$. The structure maps $\omega_n$ are the equivalences given by Bott Periodicity. Similarly, real $K$-theory is denoted by $KO$. Its spaces repeat with period eight starting with $\Z \times BO$, where $BO$ is the classifying space of the infinite orthogonal group $O$.
    \end{ex}
    
  \begin{ex}
If $X=X_0$ is an infinite loop space so that there exists spaces $X_k$ so that $X \simeq \Omega^k X_k $ for all $k\geq 0$, then the $X_k$ assemble into a prespectrum.
    \end{ex}

\begin{defn}\label{def:functorL}\label{rem:functorL}
The \emph{spectrification functor} $L \colon \mathrm{PreSp} \to \mathrm{Sp}$ is the left adjoint to the forgetful functor $U \colon \mathrm{Sp} \to \mathrm{PreSp}$.
\end{defn}

\begin{rem}
The functor $L$ exists by Freyd's adjoint functor theorem. It can be constructed easily if the maps $\omega_n$ are inclusions (for example, if $X$ is a CW-prespectrum). In this case, $LX$ is the spectrum whose $k$th space is 
\[ LX_k = \colim_n \Omega^{n+k}X_n,\]
where the colimit is taken over $ \Omega^{n+k}(\omega_n) \colon \Omega^{n+k}X_n \to  \Omega^{n+k+1}X_{n+1}$. If $X$ is already a spectrum, then $LX \cong X$ as these maps are all homeomorphisms. For a general definition, we refer the reader to Appendix A.1 of \cite{lms}. 
\end{rem}
\begin{warn}
We abuse notation and write $ULX$ simply as $LX$. Further, we often omit the $L$ if we are not emphasizing the replacement. For example, we
write $\Sigma^{\infty}A = L( \Sigma^{\infty}A  )$, $HG =L(HG)$, etc..
\end{warn}

\begin{defn}[CW-Spectra]\label{defn:CWspectra}
The category of \emph{CW-spectra}, denoted $\mathrm{CWSp}$, is the full subcategory of spectra generated by the image of the restriction of $L$ to CW-prespectra. That is, $X \in \mathrm{CWSp}$ if it is of the form $LY$ for some $Y \in \mathrm{CWPreSp}$.
\end{defn}

We summarize the discussion by the following diagram of adjunctions, where $\Omega^{\infty} \colon \mathrm{Sp} \to \Top_*$ is also the zeroth space functor:
\[\xymatrix{ \Top_* \ar@<.5ex>[rr]^-{\Sigma^{\infty}}  \ar@/^2pc/[rrrr]^{\Sigma^{\infty}} & & \mathrm{PreSp}  \ar@<.5ex>[ll]^-{\Omega^{\infty}}  \ar@<.5ex>[rr]^-{L} & &\mathrm{Sp} \ar@<.5ex>[ll]^-U 
\ar@/^2pc/[llll]^-{\Omega^{\infty}}} \]

The coproduct in $\Top_*$ is the wedge $A \vee B$. The category $\Top_*$ is a closed symmetric monoidal category, where the hom objects are the spaces of continuous based maps $\Maps(A, B)$ and the symmetric monoidal product is the smash product $A \smsh B$. There is an associated homeomorphism
\[ \Maps(A \smsh B, C) \cong \Maps(A, \Maps(B, C)).\]
We briefly discuss related constructions in (pre)spectra.

For $X$ a prespectrum and $A$ a based topological space, we let
$X \smsh A$ be the prespectrum whose spaces are given by $X_n \smsh A$ and structure maps by $\sigma_n \smsh \id_A$. We define $\Sigma^r X = X \smsh S^r$ with $\Sigma =\Sigma^1$. 

Similarly, we let $F(A, X)$ be the prespectrum whose $n$th space is $\Maps(A,X_n)$ and whose structure maps are given $f \mapsto \omega_n \circ f$, using the identification \[\Omega \Maps(A, X_{n+1}) \cong  \Maps(A, \Omega X_{n+1}) .\] We let $ \Omega(X) = F(S^1, X)$. In the homotopy category (defined in \fullref{sec:homotopycat}), the functors $\Omega(-)$ and $\Sigma(-)$ become inverses, so we let $\Sigma^{-1}(-)=\Omega(-)$.

In prespectra, the coproduct is also a wedge construction. The spaces of $X \vee Y$ are $X_n \vee Y_n$ with structure maps $\sigma_n \vee \sigma_n$, using the fact that $\Sigma(X_n \vee Y_n) \cong \Sigma X_n \vee \Sigma Y_n$. 

These constructions transfer to spectra via the spectrification functor $L$, and we abuse notation by dropping the $L$ from the notation. For example, we write $\Sigma X= L(\Sigma X) $.

\begin{rem}\label{rem:smashfun}
Smash products of spectra and function spectra are harder to construct, and we will not do this here. We do note however that there are versions of the category of spectra which are closed symmetric monoidal with respect to an appropriate smash product. The first such construction is due to Elmendorf--Kriz--Mandell--May \cite{ekmm}. However, up to homotopy (see the definition of homotopy in these categories below), the smash product $X \smsh Y$ was constructed directly two decades prior. This is called Boardman's \emph{handicrafted} smash product and the construction is described in \cite{adams}. One can also construct a function spectrum $F(X,Y)$ so that $F(X,F(Y,Z)) \simeq F(X\smsh Y,Z)$. We will only use these constructions up to homotopy and we take them for granted.
\end{rem}

\subsection{Homotopies and homotopy groups}\label{sec:homotopygroups}
Let $I_+$ be the unit interval $[0,1]$ with a disjoint basepoint. Then the prespectrum $X\smsh I_+$ admits a map 
\[X \vee X \xra{i_0 \vee i_1} X\smsh I_+\]
defined levelwise on each factor by the inclusions at $0$ and $1$ respectively.  As in $\Top_*$, we can use the prespectrum $X\smsh I_+$ to define homotopies between maps. 

Two maps of prespectra $f,g \colon X \to Y$ are \emph{homotopic}, denoted $f \simeq g$, if there is a map $H \colon X\smsh I_+ \to Y $ which restricts to $f \vee g$ along the inclusion 
\[X \vee X  \xra{i_0 \vee i_1}  X\smsh I_+ \xra{H} Y.\] 
Maps of spectra are homotopic if they are homotopic as maps of prespectra. We will let the set of homotopy classes of maps between two (pre)spectra $X$ and $Y$ be denoted by $\{X,Y\}$. If $Y$ is an $\Omega$-prespectrum, this is in fact an abelian group. Similarly, homotopy classes of maps of degree $r$ are denoted by $\{X,Y\}_r$. Two (pre)spectra $X$ and $Y$ are \emph{homotopy equivalent} if there are maps $f \colon X \to Y$ and $g \colon Y \to X$ such that $f\circ g \simeq \id_{Y}$ and $g \circ f \simeq \id_{X}$.

\begin{defn}
Let $X$ be a (pre)spectrum and $n\in \Z$. The \emph{$n$th homotopy group} of $X$ is
\[\pi_n X = \colim_{k} \pi_{n+k}X_k\]
where the maps in the colimit take an element $S^{n+k} \to X_k$ to the composite $S^{n+k+1} \to \Sigma X_k \xra{\sigma_{k}} X_{k+1}$. 
A map of (pre)spectra is a \emph{weak homotopy equivalence} if it induces an isomorphism on homotopy groups. 
\end{defn}

\begin{rem}
The unit of the adjunction $X \to LX$ is a functorial replacement of $X$ by the weakly homotopy equivalent spectrum $LX$.
\end{rem}

\begin{rem}[Whitehead's theorem]
A map of CW-spectra which is a weak homotopy equivalence is also a homotopy equivalence.
\end{rem}

\subsection{The homotopy category of spectra and its triangulation}\label{sec:homotopycat}
First, we recall the analogous object for $\Top_*$. The homotopy category of based topological spaces $\mathrm{hTop}_*$ is the initial category receiving a functor from $\Top_*$ which sends weak homotopy equivalences to isomorphisms. Using Whitehead's theorem and CW-approximation, one model for $\mathrm{hTop}_*$ has objects the pointed CW-complexes and morphisms the based homotopy classes of maps between them. The map $\Top_* \to \mathrm{hTop}_*$ sends $A$ to a CW-approximation $\Gamma A$, which is functorial up to homotopy, and a map $f$ to the homotopy equivalence class of $\Gamma f$.

There are many constructions of the homotopy category of spectra, which we denote by $\mathrm{hSp}$, including through the theory of $\infty$-categories. These all give equivalent categories and $\mathrm{hSp}$ is one of the modern settings for homotopy theory. In this section, we give some of the standard tools to work in $\mathrm{hSp}$. 

The homotopy category $\mathrm{hSp}$ is initial among categories that admit a functor out of $\mathrm{(Pre)Sp}$ which sends the weak homotopy equivalences to isomorphisms. In particular, any functor $\mathrm{(Pre)Sp}  \to \mathcal{D} $ with this property factors through the functor $\mathrm{(Pre)Sp} \to \mathrm{hSp} $:
\[\xymatrix{ \mathrm{PreSp} \ar[dr] \ar[r]^-{L} & \mathrm{Sp} \ar[d] \ar[r] & \mathrm{hSp} \ar@{-->}[dl] \\
& \mathcal{D} &  }\]

The objects of $\mathrm{hSp}$ are simply called \emph{spectra}. The category $\mathrm{hSp}$ is a triangulated category with shift operator given by the suspension $\Sigma(-)$, which in $\mathrm{hSp}$ becomes inverse to $\Sigma^{-1}(-) = \Omega(-)$. We define
\[[X,Y]_r := \mathrm{hSp}_r(X,Y) = \mathrm{hSp}(\Sigma^r  X, Y) \]
and let $[X,Y] = [X,Y]_0$. These are abelian groups for all $X$, $Y$ and $r$. The isomorphisms in $\mathrm{hSp}$ are denoted by $\simeq$ because of their relationship to the weak homotopy equivalences. 

\begin{rem} 
The morphisms in $\mathrm{hSp}$ must be computed with care, and we remind the reader of \fullref{analogy}. With this analogy in mind, note that $[X,Y]$ is not in general isomorphic to $\{X,Y\}$. Here, $\{X,Y\}$ denotes the homotopy classes of maps as defined in \fullref{sec:homotopygroups}. This is the essence of the ``cells now --- maps later'' discussion on p.142 of \cite{adams}.
\end{rem}

\begin{rem}[CW-approximation]\label{rem:CWapprox}
For any prespectrum $X$, there is a CW-spectrum $\Gamma X$ connected to $X$ by a zig-zag of weak homotopy equivalences. The construction is functorial up to homotopy.
\end{rem}

\begin{rem}
We use CW-approximation to describe models for $ \mathrm{hSp}$. The first has objects CW-spectra and morphisms homotopy classes of maps between them. In particular, if $X$ and $Y$ are CW-spectra, then
$[X,Y]\cong \{X,Y\}$. The functor $ \mathrm{(Pre)Sp} \to \mathrm{hSp}$ sends $X$ to $\Gamma X$ and a map $f$ to the homotopy equivalence class of $\Gamma f$.  A slightly larger model is to let the objects be CW-prespectra and morphisms
$[X,Y] \cong \{LX, LY\} \cong  \{X, LY\} $.
One can also take objects to be all prespectra and morphisms to be $[X,Y] \cong \{\Gamma X, \Gamma Y\}$.

The point we want to stress here is that, for any two $X$ and $Y$, whether they be prespectra, spectra, CW-prespectra or CW-spectra, it makes sense to write down $[X,Y]$. Every point of view yields isomorphic abelian groups. In $\mathrm{hSp}$, we forget the distinctions: All objects have equal dignity and are called \emph{spectra}.
\end{rem}

We extend $\Sigma^{\infty}$ to a functor $\Sigma^{\infty} \colon \Top_* \to \mathrm{hSp}$ by sending $A$ to the image of $\Sigma^{\infty}A \in  \mathrm{hSp}$. 
We often simply write $A$ to denote $\Sigma^{\infty}A \in  \mathrm{hSp}$. For example, $S^t$ as a spectrum is 
\[S^t \simeq \Sigma^{\infty} S^t \simeq \Sigma^t  \Sigma^{\infty} S^0 \simeq \Sigma^t  S^0. \]
The \emph{sphere spectrum} is the spectrum $S^0$. On the other hand, $\Omega^{\infty}$ induces a functor $\Omega^{\infty} \colon \mathrm{hSp} \to \mathrm{hTop}_*$.

We let $F(X,Y)$ and $X \smsh Y$ be the function spectrum and smash product in $\mathrm{hSp}$. See \fullref{rem:smashfun}. The category $\mathrm{hSp}$ is a closed symmetric monoidal category so that
\begin{align}\label{eq:adjunctionF}
 F(X \smsh Y, Z ) \simeq   F(X ,F(Y, Z) ). \end{align}
The sphere spectrum $S^0$ is the unit for the symmetric monoidal structure and
\begin{align*}
S^0 \smsh X &\simeq X,  &  F(S^0,X)&\simeq X. \end{align*}
If is $X$ a spectrum and $A$ is a based topological space, for the constructions described in \fullref{sec:catspectra}, we have $A \smsh X \simeq (\Sigma^{\infty}A) \smsh X$ and $F(A,X) \simeq F(\Sigma^{\infty}A, X)$.

There is an identity
\[[X,Y]_t = \pi_tF(X,Y).\]
In particular, if $\pi_*X$ denotes the homotopy groups of $X$,
\[\pi_tX \cong [S^t, X] \cong \pi_{0}F(S^t, X) \cong  \pi_t F(S^0, X) .\]

The category $\mathrm{hSp}$ has arbitary products and coproducts. Further, for a collection of objects $X_{\alpha}$, $\alpha\in I$ with the property that, for every $k \in \Z$, $\pi_kX_\alpha =0$ for all but finitely many $\alpha \in I$, the map
\begin{equation}\label{eq:prodcoprod}
\xymatrix{\bigvee_{\alpha \in I} X_\alpha \ar[r]^-{\simeq} & \prod_{\alpha \in I} X_\alpha}  \end{equation}
is an isomorphism.

Pushout and pullback diagrams also coincide in $\mathrm{hSp}$.
The exact triangles 
\[X \to Y \to Z \to \Sigma X\]
are equivalently called \emph{cofiber} and \emph{fiber} sequences. The spectrum $Z$ is called the \emph{cofiber} of $X \to Y$, while $X$ is called the \emph{fiber} of $Y \to Z$. A map $X \to Y$ is null homotopic if and only if $Z \simeq Y \vee \Sigma X$.

A standard example of an exact triangle in $\mathrm{hSp}$ is constructed by killing an element in homotopy. For example, if $\alpha \colon S^n \to S^m$ is an element of $\pi_n S^m$, then $C(\alpha)$ is defined by the exact triangle
\[S^m \xra{\alpha} S^n \to C(\alpha) \to S^{m+1}.\]

If $X \to Y \to Z \to \Sigma X$ is an exact triangle, then so are the four term sequences obtained by applying $W \smsh (-)$, $F(W,-)$ or $F(-,W)$. Further, applying either of $[-,X]$ and $[X,-]$ to an exact triangle gives rise to a long exact sequence of abelian groups. In particular, there are long exact sequences on homotopy groups $\pi_*(-)$. 

A useful fact about the functor $\Sigma^{\infty}$ is that it commutes with $\smsh$ and with $\vee$. Also, applying $\Sigma^{\infty}$ to a homotopy cofiber sequence $A \to B \to C$ of spaces gives an exact triangle in $ \mathrm{hSp}$. In particular, the cofiber sequence $A \vee B \to A \times B \to A \smsh B$
gives rise to a split cofiber sequence of spectra so that
\[\Sigma^{\infty}(A \times B) \simeq \Sigma^{\infty}A \vee \Sigma^{\infty} B \vee \Sigma^{\infty}(A \smsh B).\]

\begin{warn}
From this point onwards, when we say ``spectrum'', we mean an element of $\mathrm{hSp}$ unless otherwise specified.
\end{warn}

\subsection{Cohomology and Homology Theories.}
A generalized homology theory is a collection of functors $E_n \colon \mathrm{hSp} \to \Ab$ indexed by $\Z$, together with natural isomorphisms $E_{n+1}(\Sigma -) \xra{\cong} E_{n}( -)$ such that $E_n$ takes arbitrary coproducts to direct sums and exact triangles to exact sequences. A generalized cohomology theory is a collection of contravariant functors $E^n \colon \mathrm{hSp}^{\mathrm{op}} \to \Ab$ indexed by $\Z$ and natural isomorphisms $E^{n}( -) \xra{\cong} E^{n+1}(\Sigma -)$ such that $E^n$ that takes arbitrary coproducts to direct products and exact triangles to exact sequences. We refer the reader to Whitehead \cite[Section 5]{whitehead} for more on generalized homology and cohomology theories.

Any spectrum in $\mathrm{hSp}$ gives rise to generalized homology and cohomology theories
$E_* \colon \mathrm{hSp} \to \Ab$ and $E^* \colon \mathrm{hSp}^{\mathrm{op}} \to \Ab$. Further, by precomposing with $\Sigma^{\infty} \colon \Top_* \to \mathrm{hSp}$, we obtain (reduced) theories defined on topological spaces.
If $E \in \mathrm{hSp}$,
\[ E^n(X) = [X, E]_{-n}\cong \pi_{-n}F(X,E) \cong [X, \Sigma^n E] \]
and 
\[ E_n(X) = \pi_n( E \smsh X).\]
Conversely, the Brown representability theorem implies that any homology or cohomology theory is represented by a spectrum $E = \{E_n\}$ so that $E^n(X) = [X, E_n]$. 

\begin{rem}
If $E \in \mathrm{PreSp}$ is a prespectrum, and $A$ is a topological space, 
\[ E^n(A) \cong [A , (LE)_n]  \]
where the right hand side denotes homotopy classes of maps in $\Top_*$. In particular, if $E \in \mathrm{Sp}$, then $E^n(A) \cong [A,E_n]$. In fact, for this to hold, it is enough that the structure maps $\omega_n$ be weak homotopy equivalences (i.e., that $E$ be an $\Omega$-prespectrum).

If $E \in \mathrm{CWPreSp}$ is such that $E_n$ is $n-1$-connected, then
\[E_n(A) \cong \pi_n(E \smsh A) \cong  \colim_{k} \pi_{n+k} (E_k \smsh A). \]
\end{rem}

\begin{ex}
If $E = HG$, the Eilenberg--MacLane spectrum for an abelian group $G$ and $A$ is a based space, or $B_+$ is an unbased space with a disjoint base point, then
\begin{align*}
HG^*(A) &= \widetilde{H}^*(A ; G)  &  HG^*(B_+) &= {H}^*(B ; G).
\end{align*}
Further, by definition, $\widetilde{H}^*(A;G) \cong HG^*(\Sigma^{\infty}A)$.
\end{ex}

\subsection{Connective spectra}\label{sec:connective}
Let $A \in \mathrm{Top}_*$ be a connected CW-complex. For every $m\geq 0$, there is a space $A_{\tau \geq m} $ with the property that $\pi_nA_{\tau \geq m}=0$ if $n<m$, together with a map $A_{\tau \geq m} \to A$ which is an isomorphism  on $\pi_n$ if $n \geq m$. 
The space $A_{\tau \geq m}$ is called the \emph{$m$th connective cover of $A$}, and is obtained as the $m$th stage of the Whitehead tower. This can be done functorially and the spaces $A_{\tau \geq m}$ are unique up to canonical isomorphism in $\mathrm{hTop}_*$.

 Note that the homotopy groups of spectra are defined for any integer $n \in\Z$. In particular, some spectra have negative homotopy groups. If $X \in \mathrm{hSp}$ and $m \in \Z$, the \emph{$m$th connective cover} of $X$  is a spectrum $X_{\tau \geq m} $ with the property that $\pi_nX_{\tau \geq m} =0$ for $n<m$, together with a map $X_{\tau \geq m} \to X$ which is an isomorphism on $\pi_n$ if $n\geq m$. If $X \in \mathrm{hSp}$ is represented by a prespectrum with spaces $X_n$, then $X_{\tau \geq m}$ is represented by a prespectrum whose spaces are $(X_{\tau \geq m})_n  =  (X_n)_{\tau \geq m+n}$ and whose structure maps are obtained from those of $X$ using the functoriality and uniqueness.
The spectrum $X_{\tau \geq 0} $ is called the \emph{connective cover} of $X$. 
\begin{notation}
The spectrum $\ku$ denotes the connective cover of the $K$-theory spectrum $K$. The spectrum $\ko$ denotes the connective cover of the real $K$-theory spectrum $KO$.
\end{notation}

\subsection{Multiplicative Homology Theories.} One of the main reasons for introducing a symmetric monoidal products on the category of spectra $\mathrm{Sp}$ or on its homotopy category $\mathrm{hSp}$ is the discussion of \textit{ring spectra}. The cohomology theory that one first encounters, singular cohomology, has the structure of a graded ring. By the Brown representability theorem, this gives rise to maps $H \Z \smsh H \Z \to H\Z$ for the Eilenberg--MacLane spectrum $H\Z$. Many cohomology theories come equipped with this structure; for example, $K$ and $KO$-theory and nearly all cobordism theories. 

We give some definitions in $\mathrm{hSp}$. A \emph{ring spectrum} is a spectrum $R \in \mathrm{hSp}$ together with a multiplication map $\mu \colon R \smsh R \to R$ and a unit map $\eta \colon S^0 \to R$ such that the diagram
\[
\xymatrix{
  S^0 \smsh R\ar[dr]_-{\simeq} \ar[r]^-{\eta \smsh \id_R} & R \smsh R \ar[d]^\mu &  R \smsh S\ar[dl]^-{\simeq} \ar[l]_-{\id_R \smsh \eta}\\
   & R& 
}
\]
commutes (in $\mathrm{hSp}$). Granted a notion of ring spectrum, we can define commutative ring spectra, and module spectra. A commutative ring spectrum is one such that the diagram
\[
\xymatrix{
  R \smsh R \ar[dr]_-{\mu}\ar[rr]^-{\text{tw}} & &   R \smsh R \ar[dl]^-{\mu} \\
  & R 
}
\]
commutes, where $\operatorname{tw}$ is the map that exchanges the two copies of $R$. For $R$ a ring spectrum, an $R$-module spectrum is a spectrum $M$ together with a map $R \smsh M \to M$ which fits into the commutative diagrams that categorify the notion of a module over a ring.

Much of the intuition from homological algebra can be carried over to the context of ring spectra and module spectra. For example, one can define resolutions in this context. The homotopy groups of a resolution will reflect properties of the homotopy groups of the spectrum it resolves. See, for example, \cite{millerrelations}. This is one of the ideas in the construction of the Adams spectral sequence. See \fullref{sec:assconstruction}.

A construction from algebra that requires more care with the smash product when being adapted to spectra is the notion of quotient modules. This is solved in the modern categories of spectra, but is not needed here.

\subsection{Spanier--Whitehead duality}
The functional dual of a spectrum $X$ is the function spectrum $F(X, S^0)$. This is often denoted by $DX$ in analogy with Spanier--Whitehead duality. If $X \simeq \Sigma^{\infty}A$ for a finite CW-complex $A$, then $DX$ is the classical Spanier--Whitehead dual of $A$.

The enriched adjunction \eqref{eq:adjunctionF} gives rise to certain important maps. First, there are the units and the counits which are ``coevaluations'' and ``evaluations'' respectively:
\begin{align*}
coev &\colon  Y \to F(X, X\smsh Y)    & ev &\colon  X \smsh F(X,Y) \to Y
\end{align*}
Using the adjunction \eqref{eq:adjunctionF} and $ev$, for any spectra $X$, $Y$ and $Z$, the adjoint to the $X \smsh F(X,Y) \smsh Z \xra{ev \smsh Z} Y \smsh Z$ gives a map
\begin{equation}\label{eq:dualstruc}
 F(X, Y)\smsh Z \to F(X, Y\smsh Z)
 \end{equation}
which may or may not be an isomorphism in $\mathrm{hSp}$. 

The spectrum $Z$ is called \emph{dualizable} if this is an isomorphism in $\mathrm{hSp}$ for all spectra $X$ and $Y$. Examples of dualizable spectra are the spheres $S^t = \Sigma^tS^0$ and, more generally, the suspension spectrum $\Sigma^{\infty}A$ of any finite CW-complex $A$. Finally, to verify that $Z$ is dualizable, it is enough to check that
\[ DZ \smsh Z \simeq F(Z, S^0)\smsh Z \to F(Z,  Z)\]
is a weak equivalence.

\subsection{Brown-Comenetz and Anderson duality}\label{sec:andersondual}
For any injective abelian group $A$, the functor from $\Top_*$ to abelian groups given by
\[ I_A^n(X) = \Hom_{\Z}(\pi_{n}( \Sigma^{\infty}X) ,A)\]
defines a cohomology theory, which is represented by a spectrum denoted $I_{A}$. For example, if $A=\Q$, then 
\[I_{\Q}^n(X) \cong \widetilde{H}^{n}(X ; \Q) ,\] 
and $I_{\Q}$ is equivalent to $H\Q$.

Since $\Q/\Z$ is an injective abelian group, we also obtain a spectrum $I_{\Q/\Z}$, which is often called the \emph{Brown-Comenetz spectrum}. The natural map $\Q \to \Q/\Z$ together with the Yoneda Lemma gives rise to a map of spectra $I_{\Q} \to I_{\Q/\Z}$. Then $I_{\Z}$ is defined by the exact triangle in $\mathrm{hSp}$
\begin{align}\label{eq:fibIZ} 
I_{\Z} \to I_{\Q} \to I_{\Q/\Z}   \to \Sigma I_{\Z}.
\end{align}
The spectrum $I_{\Z}$ is called the \emph{Anderson dual spectrum}. 

Associated to \eqref{eq:fibIZ} is a long exact sequence on cohomology
\[   \ldots   \to  {I_\Q}^{*-1}(X)  \to  I_{\Q/\Z}^{*-1}(X) \to I_{\Z}^*(X) \to {I_\Q}^*(X) \to I_{\Q/\Z}^*(X)  \to \ldots  \]
If the homotopy groups $\pi_*(\Sigma^{\infty}X)$ are finitely generated abelian groups in each degree, one can deduce from this long exact sequence that there is an isomorphism
\[  I_{\Z}^*(X)   \cong   \mathrm{Torsion}(\pi_{*-1}(\Sigma^{\infty}X) )   \oplus \mathrm{Free}(\pi_{*}X). \]
So, computing $I_{\Z}^*(X) \cong [X, \Sigma^*I_{\Z}]$ is equivalent to computing the (stable) homotopy groups of $X$.

\subsection{Thom Spectra}\label{sec:thomspectra} Let $B$ be a topological space and $\nu \colon E \to B$ be a $n$-dimensional real vector bundle on $B$. Then $\mathrm{Sph}(\nu)  \colon \mathrm{Sph}(E)  \to B  $ is the $n$-sphere bundle whose fibers are the one-point compactification of the fibers of $\nu$. The bundle $ \mathrm{Sph}(E)$ has a section $s \colon B \to  \mathrm{Sph}(E)$ which sends $b$ to the point at infinity in the fiber $ \mathrm{Sph}(E)_b$. Then the Thom space of $\nu$ is defined as
\[ B^{\nu}  = \mathrm{Sph}(E)/s(B)   .\]
The \emph{Thom spectrum}, also denoted by $B^{\nu}$, is the suspension spectrum of the Thom space. The composite 
\[\mathrm{Sph}(E)  \to \mathrm{Sph}(E)  \times \mathrm{Sph}(E)  \to B \times B^{\nu},\] 
which is the diagonal map followed by the product of $\mathrm{Sph}(\nu)$ and the quotient map, induces a map $B^{\nu} \to B_+ \smsh B^{\nu}$ called the \emph{Thom diagonal}.

If $\nu = \alpha \oplus \mathbf{n}$ where $\mathbf{n}$ is the trivial $n$-dimensional bundle, then 
\begin{equation}\label{eq:trivialsus}B^{\nu} \simeq \Sigma^n B^{\alpha}. \end{equation}
In particular, if $\mathbf{0}$ is the zero bundle, then $B^{\mathbf{0}} = \Sigma^{\infty} B_{+}$.

The identity \eqref{eq:trivialsus} motivates the definition of Thom spectra for virtual bundles. We give the definition for based spaces $B$ which are CW-complexes with finitely many cells in each dimension. Recall that a virtual bundle $\nu$ over $B$ is the formal difference $\nu = \alpha-\beta$ of vector bundles $\alpha$ and $\beta$ over $B$. If $\alpha$ is an $n$-dimensional bundle and $\beta$ is an $m$-dimensional bundle, we say that $\nu$ has dimension $n-m$. 

If $B$ is compact, we can choose a bundle $\beta^{\perp}$ and an integer $k$ so that $\beta \oplus \beta^{\perp} \cong \mathbf{k}$. In this case, we define
\[ B^{\nu} :=  \Sigma^{-k}B^{\alpha \oplus \beta^{\perp}} .\]
This is independent of the choice of complement $\beta^{\perp}$. Now, let $B_q$ be the $q$-skeleton of $B$. By our assumption on $B$, the space $B_q$ is compact. The bundle $\nu$ pulls back to virtual bundles $\nu_q$ over $B_q$ for each $q$. There are induced maps of Thom spectra $B_{q}^{-\nu_q} \to  B_{q+1}^{-\nu_{q+1}}$, and 
\[B^{-\nu} := \colim_{q} B_{q}^{-\nu_q} .\]

\begin{ex}
Let $O_n$ be the $n$th orthogonal group and $BO_n$ its classifying space. A model for $BO_n$ is given by the Grassmanian $G_n= \varinjlim_{k}\mathrm{Gr}_n(\R^{k})$, where $\mathrm{Gr}_n(\R^{k})$ is the space of $n$-dimensional subspaces of $\R^{k}$ and the maps in the colimit are induced by the inclusions $\R^{k}\subseteq \R^{k+1}$ into the first $k$-coordinates. This has the homotopy type of a CW-complex with finitely many cells in each dimension. Consider the subspace of $G_n \times \R^{\infty}$ given by
\[ E_n= \{ (P,v) \in G_n \times \R^{\infty} : P \in G_n, v\in P \}.  \]
The map 
\[\gamma_n \colon E_n \to G_n\] 
which sends $(P,v)$ to $P$ is an $n$-dimensional vector bundle. This is often called the \emph{universal bundle} over $BO_n$. The associated Thom space is denoted by $MO_n$, which is also used to denote the associated Thom spectrum. 

If $H_n \to O_n$ is a group homomorphism, then the universal bundle $\gamma_n$ pulls back to a bundle over $BH_n$ that we will also denote by $\gamma_n$. The associated Thom space/spectrum is denoted by $MH_n$.

Finally, in these examples, the Thom spectrum of the virtual bundle $-\gamma_n$ is denoted by $MTH_n$ and is called the \emph{Madsen--Tillmann spectrum}.
\end{ex}

\begin{rem}Thom spectra are related to Spanier--Whitehead duality via the \emph{Atiyah duality} isomorphism.
Let $M$ be an $n$-manifold and $TM$ be the tangent space of $M$, then Atiyah duality is the equivalence $M^{-TM} \simeq D(\Sigma^{\infty} M_+)$.
\end{rem}

The cohomology of the Thom space is related to the cohomology of the base space. We treat the case $H^*(-;\Z/2 )$ as it comes free of orientability conditions. Given any virtual $n$-bundle $\nu$, there is an isomorphism 
\[  \mathrm{Th} \colon   H^*(B;\Z/2) \cong  \widetilde{H}^*(B^{\mathbf{0}};\Z/2) \to \widetilde{H}^{*+n}(B^{\nu};\Z/2). \]
called the \emph{Thom isomorphism}. The isomorphism is given by an external cup product with a class 
\[U  = U(\nu) \in \widetilde{H}^n(B^{\nu} ; \Z/2) \] 
called the \emph{Thom class}.

% !TEX root = cbms-master.tex

\section{The Steenrod algebra}\label{sec:steenrod}

In this section, we review some basic facts about the Steenrod algebra $\A$ at the prime $p=2$. A very good reference for this material is Mosher--Tangora \cite{MosherTangora} and the interested reader should consult it for a more thorough presentation.

We focus on the prime $p=2$, although much of this story has an analogue at odd primes. We will let 
\[H^*(X) = \widetilde{H}^*(X ;\Z/2)\]
denote the reduced mod $2$ cohomology of $X$ if it is a space, or simply the mod $2$ cohomology of $X$ if it is a (pre)spectrum. If $X \in \Top$ and we want to refer to the unreduced cohomology, we will use the notation $H^*(X;\Z/2)$.

\subsection{Cohomology operations and the Steenrod algebra}\label{sec:steenroddesc}
Let $ \Vect(\Z/2 )$ denote the category of $\Z$-graded $\Z/2 $ vector spaces, so that mod $2$ cohomology is a functor 
\[H^*(-;\Z/2) \colon  \Top \to \Vect(\Z/2 ).\] 
A cohomology operation of degree $k$ is a natural transformation
\[ \gamma \colon H^*(-;\Z/2) \to H^{*+k}(-;\Z/2).\]
The operation $\gamma$ is said to be \emph{stable} if it commutes with the suspension isomorphism
\[\Sigma \colon H^*(-) \xra{\cong} {H}^{*+1}(\Sigma(-)).\]

\begin{ex}
The short exact sequence
\[ 0 \to \Z/2 \to \Z/4 \to \Z/2 \to 0\]
induces a long exact sequence on cohomology
\[ \xymatrix{ \ldots \ar[r] & H^*(-; \Z/4) \ar[r] & H^*(-; \Z/2) \ar[r] & H^{*+1}(-; \Z/2)  \ar[r] & \ldots   }\]
The connecting homomorphism $H^*(-; \Z/2)  \to H^{*+1}(-; \Z/2)  $ is natural and commutes with the suspension isomorphism, so it is a stable cohomology operation of degree one. We call this operation $Sq^1$; it is also known as the Bockstein homomorphism. 
\end{ex}

\begin{ex}
Consider the real projective plane $ \R P^2$. Then 
\[H^*( \R P^2;\Z/2) \cong \Z/2[w_1]/w_1^3\] for a class $w_1$ in degree $1$. (The name $w_1$ will reappear in \fullref{sec:SWclasses} and is used consistently here.) Then $Sq^1(w_1) =w_1^2$. In fact, $\R P^2$ can be constructed from the circle $S^1$ via the following pushout diagram:
\[
\xymatrix{S^1 \ar[r] \ar[d]_{2} & D^2\ar[d] \\ 
S^1 \ar[r] & \R P^2}
\]
The element $w_1$ is dual to the homology class represented by the $1$-cell and the element $w_1^2$ is dual to that represented by the $2$-cell. The cohomology operation $Sq^1(w_1) = w_1^2$ is recording the fact that the $2$-cell of $\R P^2$ is attached to the $1$-cell via the multiplication by $2$ map. See \fullref{fig:RP2CP2RPinf}.
\end{ex}

\begin{defn}
The Steenrod algebra $\mathcal{A}$ is the graded non-commmutative $\F_2$-algebra generated in degree $k$ by the stable cohomology operations of that degree and with multiplication given by composition of operations. 
\end{defn}

\begin{rem}
Let $ \HF $ be the mod-$2$ Eilenberg--MacLane spectrum whose $n$th space is given by $K(\Z/2,n)$.
Since 
\[H^t(-;\Z/2) \cong [(-)_+, K(\Z/2,n)] \cong [ \Sigma^{\infty}(-)_+, \Sigma^{t}  \HF ] \]
it follows from the Yoneda Lemma that degree $t$ cohomology operations are in one to one correspondence with maps $[ \HF , \Sigma^t  \HF ]$. Therefore,
\[\mathcal{A} \cong  \HF ^*( \HF ).\]
\end{rem}

Constructing all cohomology operations is rather difficult and a good reference is given by \cite{MosherTangora}. However, $\mathcal{A}$ can be described axiomatically and this is the approach we take here. 
\begin{thm}\label{thm:squnstable}
For each $k \geq 0$, there exists a stable cohomology operation of degree $k$ 
\[Sq^k \colon H^*(-;\Z/2 ) \to H^{*+k}(-;\Z/2 ) \]
called the \emph{$k$th Steenrod square}. For $X$ a topological space, the Steenrod squares satisfy the following properties:
\begin{enumerate}[(a)]
\item $Sq^0 =1$
\item For $x \in H^k(X;\Z/2)$, $Sq^k(x) = x^2$.
\item If $x\in H^i(X;\Z/2)$ and  $i<k$, then $Sq^k(x) =0$.
\item (Cartan Formula) $Sq^k(xy) = \sum_{i+j = k } Sq^i(x)Sq^j(y)$, where the multiplication on $H^*(X;\Z/2)$ is given by the cup product.
\end{enumerate}
\end{thm}

\begin{rem}\label{rem:twoforms}
In \fullref{thm:squnstable}, the Cartan Formula is only expressed for the cup product of elements in $H^*(X;\Z/2 )$. However, it also holds for the cross product. That is, if $x \in H^*(X;\Z/2)$ and $y \in H^*(Y;\Z/2)$, then for 
\[x \otimes y \in H^*(X \times Y;\Z/2) \cong H^*(X;\Z/2)\otimes_{\Z/2} H^*(Y;\Z/2),\] 
then
\[ Sq^k(x \otimes y) = \sum_{i+j = k } Sq^i(x) \otimes Sq^j(y).\] 
If one is working with the reduced cohomology groups, then the same formula holds for $H^*(X \smsh Y) \cong H^*(X) \otimes_{\Z/2} H^*(Y)$.

Finally, if there is a continuous map $Y \to X \times Y $, so that $H^*(Y;\Z/2)$ becomes a module over $H^*(X;\Z/2)$, then the Cartan Formula implies that 
\[ Sq^k(x \cdot y) = \sum_{i+j = k } Sq^i(x) \cdot Sq^j(y)\]
where $\cdot$ denotes the action of $H^*(X;\Z/2)$ on $H^*(Y;\Z/2)$.
\end{rem}

\begin{thm}
The Steenrod algebra $\mathcal{A}$ is the tensor algebra over $\Z/2$ generated by the $Sq^i$ subject to the following relations:
\begin{enumerate}[(1)]
\item $Sq^0 =1$ 
\item The Adem relations: For $0< a <2b$,
\[Sq^aSq^b = \sum_{c=0}^{[a/2]} \binom{b-c-1}{a-2c} Sq^{a+b-c} Sq^c.\]
\end{enumerate}
\end{thm}

\begin{rem}\label{rem:hopfalgebra}
The Steenrod algebra $\A$ is a graded, non-commutative, augmented algebra. In fact, it is a cocommutative Hopf algebra over $\Z/2 $\footnote{The authors have heard the following anecdote from Doug Ravenel: During a lecture of Milnor on Hopf algebras at Princeton many years ago, Steenrod asked if there were any interesting examples.} whose coproduct $\psi \colon \A \to \A \otimes \A$ is determined by
\[\psi(Sq^k)  = \sum_{i+j =k} Sq^i \otimes Sq^j.\]
The antipode $\chi \colon \A \to \A$ is defined inductively by the identities
\begin{align*}
 \chi(Sq^0)&=Sq^0,  &  \sum_{i=0}^k Sq^i \chi(Sq^{k-i})&=0, \ k>0.
 \end{align*}
We note that $\A_0 =\Z/2 $ and let $I(\A)$ be the kernel of the augmentation $\varepsilon \colon \A \to \Z/2 $.
\end{rem}

\begin{rem}\label{rem:modA}
We let $\Mod_{\mathcal{A}}$ be the category of graded left modules over $\mathcal{A}$. These are $\Z$-graded $\Z/2 $-vector spaces together with a left action of $\mathcal{A}$. Given $M$ and $N$ in $\Mod_{\mathcal{A}}$, we let $M\otimes_{\Z/2 }N$ be the module whose structure is given by $a(m\otimes n) = \sum a_im\otimes a_jn$, where $a\in \A$ and $ \psi(a) = \sum a_i\otimes a_j$.

Modules which satisfy the conditions of \fullref{thm:squnstable} are called unstable modules. The cohomology of a spectrum need not be an unstable module in general.
\end{rem}

 To specify an $\mathcal{A}$-module structure on a graded $\Z/2$-vector space $M$, one must describe the action of the Steenrod squares on $M$. We record this information in a picture we call an \emph{cell diagram}. See \fullref{fig:cellexplanation}. The following result implies that specifying the action of $Sq^{2^n}$ for $n\geq 0$ is enough to describe an $\A$-module.
\begin{thm}
$\mathcal{A}$ is generated as an algebra by $Sq^{2^n}$ for $n\geq 0$.
\end{thm}

\begin{ex}
Consider the complex projective plane $ \C P^2$. Then 
\[H^*( \C P^2;\Z/2) \cong \Z/2[w_2]/w_2^3\] for a class $w_2$ in degree $2$. (The name $w_2$ reappears in \fullref{sec:SWclasses} and is used consistently here.) It follows from the properties of the squares that $Sq^2(w_2) =w_2^2$. In fact, $\C P^2$ can be constructed from the sphere $S^2$ via the following pushout diagram:
\[
\xymatrix{S^3 \ar[r] \ar[d]_{\eta} & D^4\ar[d] \\ 
S^2 \ar[r] & \C P^2}
\]
where $\eta \colon S^3 \to S^2$ is the Hopf fibration. The element $w_2$ is dual to the homology class represented by the $2$-cell and the element $w_2^2$ is dual to that represented by the $4$-cell. The cohomology operation $Sq^2(w_2)=w_2^2$ is recording the fact that the $4$-cell of $\C P^2$ is attached to the $2$-cell via the map $\eta$. See \fullref{fig:RP2CP2RPinf}.
\end{ex}

\begin{ex}
The Steenrod operations for the cohomology of $\R P^{\infty} \simeq BO_1$ are completely explicit.
Writing $H^*(\R P^{\infty};\Z/2) \cong \Z/2[w_1]$ for $w_1$ in degree $1$, we have
\[ Sq^n(w_1^m) = \binom{m}{n}w_1^{m+n}.\]
Using the naturality of the squares, this example often comes in handy in computing operations in the cohomology of other spaces. See \fullref{fig:MO1MU1}
\end{ex}

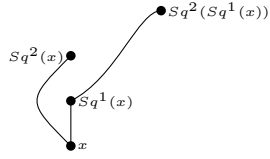
\begin{figure}[ht]
\center
  {\scalefont{.5} \begin{tikzpicture}[scale = .6]
  \fill  (0,0) circle (3pt) node[anchor=west] {$x$} ;
    \sqone(0,0,black);
        \sqtwoL(0,0,black);
  \fill  (0,1) circle (3pt) node[anchor=west] {$Sq^1(x)$} ;
        \sqtwoCR(0,1,black);
          \fill  (2,3) circle (3pt) node[anchor=west] {$Sq^2(Sq^1(x))$} ;
  \fill  (0,2) circle (3pt) node[anchor=east] {$Sq^2(x)$} ;
\end{tikzpicture}}
\caption{A \emph{cell diagram}, used to depict the Steenrod operations on the cohomology of a space or spectrum. Each $\bullet$ denotes a generator of $\Z/2 $. The difference in cohomological degree of the generators is represented vertically. Straight lines denote the action of $Sq^1$ and curved lines denote the action of $Sq^2$.} 
\label{fig:cellexplanation}
\end{figure}

\begin{figure}[ht]
\center
{\tiny\begin{tikzpicture}[scale = .6]
\fill (0,1) circle (3pt) node[anchor=east] {$w_1$};
\fill (0,2) circle (3pt) ;
\sqone(0,1,black);
\fill  (3,2) circle (3pt) node[anchor=east] {$w_2$};
\fill  (3,4) circle (3pt);
\sqtwoR (3,2,black);
\end{tikzpicture}}
\caption{The structure of $H^*(\R P^2)$ (left), $H^*(\C P^2)$ (right) 
as modules over $\A$. The class $w_1$ is in $H^1(\R P^{2})$.
The class $w_2$ is in $H^2(\C P^2)$.} 
\label{fig:RP2CP2RPinf}
\end{figure}
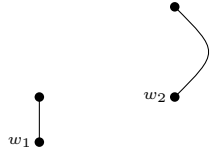

\subsection{The Subalgebras $\A_n $}\label{sec:ans}

The Steenrod algebra is an infinitely generated non-commutative algebra. However, it is finitely generated in each degree. In fact, it is filtered by the finite sub-Hopf algebras generated by $Sq^1, \ldots, Sq^{2^n}$, which are denoted $\A_n $. Further, each algebra $\A_n $ contains a commutative subalgebra generated by elements $Q_0, \ldots, Q_n$ which are defined inductively by
\begin{align*}
Q_0 &= Sq^1, \\
Q_i &= Sq^{2^i}Q_{i-1} + Q_{i-1} Sq^{2^i}.
\end{align*}
In fact, the $Q_i$'s generate an exterior algebra and we let $\E_n= E(Q_{0}, \ldots, Q_n)$.

For example, the algebra $\A_1  $ is the subalgebra of $\A$ generated by $Sq^1$ and $Sq^2$. As a module over itself, $\A_1  $ admits the cell diagram depicted in \fullref{fig:A1E1}.

\begin{figure}[ht]
\center
\begin{tikzpicture}[scale=.6]
  \Aone (0,0)
  
  \fill (6,0) circle (3pt) ;
  \sqone (6,0,black);
    \fill (6,1) circle (3pt) ;
     \draw[dashed]  (6,0) ..controls (5,1.5)..  (6,3);
     \draw[dashed]  (6,1) ..controls (7,2.5)..  (6,4);
         \fill (6,3) circle (3pt) ;
              \fill (6,4) circle (3pt) ;
           \sqone (6,3,black);
\end{tikzpicture}
\caption{$\A_1$ (left) and its subalgebra $\E_1$ (right). The dashed lines represent the action of $Q_1 = Sq^1Sq^2 + Sq^2 Sq^1$.}
\label{fig:A1E1}
\end{figure}
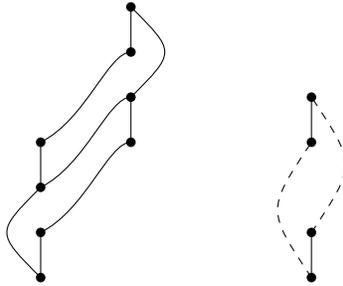

\begin{defn}\label{defn:AmmB}
Let $\mathcal{B}$ be a subalgebra of a $\Z/2 $-algebra $\mathcal{C}$. Then
\[\mathcal{C} /\!\!/ \mathcal{B} := \mathcal{C}\otimes_{\mathcal{B}} \Z/2   \]
where $\Z/2 $ denotes the trivial $\cB$ module concentrated in degree zero.
\end{defn}

\begin{rem}The subalgebras $\A_1 $ and $\E_1 $ appear naturally in classical computations as they are related to $K$-theory. Let $\ku$ be the connective $K$-theory spectrum and $\ko$ its real version. See \fullref{sec:connective}. Then there are isomorphisms of $\A$-modules
\begin{align*}
H^*(\ku) &\cong \A  /\!\!/ \E_1 , & H^*(\ko) &\cong \A  /\!\!/  \A_1 .
\end{align*}
Similarly, if  $\tmf$ is the connective spectrum of topological modular forms and $BP\langle 2\rangle$ is a spectrum obtained from the Brown-Peterson spectrum $BP$ by killing a choice of generators $v_k$ for $k\geq 3$, then
\begin{align*}
H^*(BP\langle 2\rangle ) &\cong \A  /\!\!/ \E_2, & H^*(\tmf) &\cong \A  /\!\!/  \A_2 .
\end{align*}
These spectra are the chromatic height $2$ analogues of $\ku$ and $\ko$ respectively.\end{rem}

\subsection{Thom Spectra and Stiefel--Whitney Classes}\label{sec:SWclasses}

Given an $n$-dimensional vector bundle $\nu \colon E \to B$, we recalled the definition of the Thom space $B^{\nu}$ of $\nu$ in \fullref{sec:thomspectra}. Further, we recalled the Thom isomorphism
\[ \mathrm{Th} \colon H^*(B;\Z/2) \to \widetilde{H}^{*+n}(B^{\nu};\Z/2)\]
which was given by the cup product with a Thom class $U \in \widetilde{H}^n(B^{\nu} ; \Z/2)$. We note that $\mathrm{Th}(1)=U$ and write 
\[\mathrm{Th}(x) = xU.\]

\begin{warn}
The Steenrod operations do not commute with the Thom isomorphism. This fact is crucial for \fullref{defn:SW} below.
\end{warn}

The Thom isomorphism is used to define classical invariants of a bundle $\nu$ called the \emph{Stiefel--Whitney classes}.
\begin{defn}\label{defn:SW}
The \emph{$i$th Stiefel--Whitney class} $w_i=w_i(\nu)$ of a vector bundle $\nu$ is defined by
\[ w_i = \mathrm{Th}^{-1}( Sq^i(U)) \in H^i(B;\Z/2).\]
In particular, they satisfy the identity
\[w_i U = Sq^i(U).\]
The \emph{total Stiefel--Whitney class} is the formal sum
\[w = w(\nu) = 1+w_1 + w_2 +\ldots  \]
\end{defn}

\begin{rem}\label{rem:totalSW}
If $\nu$ is a trivial bundle, the Stiefel--Whitney classes are trivial except for $w_0 =1$.  Given two vector bundles $\nu$ and $\eta$, 
one can show that
\[ w(\nu \oplus \eta) = w(\nu) w(\eta).\]
If follows that
\[w(\nu \oplus \nu^{\perp}) =1 \]
for any orthogonal complement of an embedding of $\nu$ into a trivial bundle $\mathbf{m}$. This identity allows us to determine the Stiefel--Whitney classes of $\nu^{\perp}$ given those of $\nu$. It also allows us to define the Stiefel--Whitney classes of a virtual bundle. In particular, 
\[w(-\nu) = w(\nu)^{-1}. \]
\end{rem}

The effect of the Steenrod squares on the Stiefel--Whitney classes is given by the Wu formula.
\begin{thm}[Wu Formula]
Let $\nu \colon E \to B$ be a vector bundle over $B$. Then
\[Sq^i(w_j) = \sum_{k=0}^{i} \binom{(j-i)+(k-1)}{k}w_{i-k}w_{j+k}.\]
\end{thm}

\begin{rem}
Applying $\mathrm{Th}(-)$ to both sides of the display in \fullref{defn:SW}, one deduces that 
\[Sq^i(U) = w_iU \in H^{i+n}(B^{\nu}). \]
Further, the Thom diagonal gives $H^*(B^{\nu})$ the structure of an $H^*(B; \Z/2)$-module. So
using \fullref{rem:twoforms}, for $x \in H^*(B;\Z/2)$ we have
\[Sq^k(x U) = \sum_{i+j=k} Sq^i(x)Sq^j(U) =  \sum_{i+j=k} Sq^i(x)w_j U. \]
This determines the structure of $H^*(B^{\nu})$ as an $\A$-module based on that of $H^*(B)$.
\end{rem}

\subsection{Examples of computations of Steenrod operations}\label{sec:compA1}
In this section, we go through a few selected computations of Steenrod operations. Most of the examples play a role in Section 10 of \cite{FH}. Further, the computations illustrate many of the concepts and techniques mentioned above. We do not do all the computations in detail but try to give enough information for the reader to learn the techniques and be able to reproduce them on their own.
\begin{ex}\label{ex:bogen}
The classifying space $BO_n$ carries the universal $n$-plane bundle $\gamma_n$, and its Thom space is denoted $MO_n$. The cohomology of $BO_n$ is 
\begin{align*}H^*(BO_n;\Z/2) &= \Z/2  [w_1, \ldots, w_n], & H^*(MO_n) &= Z/2 [w_1, \ldots, w_n]\{U\}.\end{align*}
Similarly, $BSO_n$ carries the universal oriented $n$-plane bundle and its Thom space is denoted by $MSO_n$. A bundle is oriented if and only if $w_1=0$, and so
\begin{align*} H^*(BSO_n;\Z/2) &= \Z/2  [w_2, \ldots, w_n], & H^*(MSO_n) &=  \Z/2 [w_2, \ldots, w_n]\{U\}.\end{align*}
\end{ex}

\begin{ex}\label{ex:mo1mu1}
As special cases of \fullref{ex:bogen} we have
\begin{align*}
H^*(\R P^{\infty};\Z/2) &\cong H^*(BO_1;\Z/2) \cong \Z/2 [w_1],  & H^*(MO_1) \cong  \Z/2 [w_1]\{U\} 
\end{align*}
Further, $Sq^1(w_1^{k}U) = w_1^{k+1}U$ if $k$ is even and zero if $k$ is odd. Using the Cartan Formula as in \fullref{rem:twoforms}, one deduces that
$Sq^2(w_1^{k}U) = \binom{k-1}{2} w_1^{k+2}U$. In fact, $MO_1 \simeq \R P^{\infty}$.

Similarly, 
\begin{align*}
H^*(\C P^{\infty};\Z/2) &\cong H^*(BSO_2;\Z/2) \cong \Z/2 [w_2], & H^*(MU_1) \cong  \Z/2 [w_2]\{U\},
\end{align*}
and $Sq^2(w_2^{k}U) = w_2^{k+1}U$ if $k$ is even and zero if $k$ is odd. All of the $Sq^1$s are zero. In fact, $MU_1 \simeq \C P^{\infty}$.
\begin{figure}[ht]
\center
{\tiny
  \begin{tikzpicture}[scale = .6]
    \fill (0,0) circle (3pt) node[anchor=east]  {$w_1$} ;  
    \sqone(0,0,black);
        \fill (0,1) circle (3pt) node[anchor=east]  {$w_1^2$};  
         \sqtwoR(0,1,black);
          \fill (0,2) circle (3pt) node[anchor=east]  {$w_1^3$};  
              \sqone(0,2,black);
                      \sqtwoL(0,2,black);
                    \fill (0,3) circle (3pt) node[anchor=east]  {$w_1^4$};  
                    \fill (0,4) circle (3pt) node[anchor=east]  {$w_1^5$};  
                               \sqone(0,4,black);
                       \fill (0,5) circle (3pt) node[anchor=east]  {$w_1^6$};  
                         \sqtwoL(0,5,black);
                         \fill (0,6) circle (3pt) node[anchor=east]  {$w_1^7$};  
                             \sqone(0,6,black);
                           \fill (0,7) circle (3pt) node[anchor=east]  {$w_1^8$};  
                         \sqtwoR(0,6,black);
                       
                           \fill (4,0) circle (3pt) node[anchor=east]  {$U$} ;  
    \sqone(4,0,black);
        \fill (4,1) circle (3pt) node[anchor=east]  {$w_1U$};  
         \sqtwoR(4,1,black);
          \fill (4,2) circle (3pt) node[anchor=east]  {$w_1^2U$};  
              \sqone(4,2,black);
                      \sqtwoL(4,2,black);
                    \fill (4,3) circle (3pt) node[anchor=east]  {$w_1^3U$};  
                    \fill (4,4) circle (3pt) node[anchor=east]  {$w_1^4U$};  
                               \sqone(4,4,black);
                       \fill (4,5) circle (3pt) node[anchor=east]  {$w_1^5U$};  
                                  \sqtwoL(4,5,black);
                         \fill (4,6) circle (3pt) node[anchor=east]  {$w_1^6U$};  
                             \sqone(4,6,black);
                           \fill (4,7) circle (3pt) node[anchor=east]  {$w_1^7U$};  
                              \sqtwoR(4,6,black);
                       
                           \fill (8,1) circle (3pt) node[anchor=east]  {$w_2$} ;  
                                  \sqtwoL(8,1,black);
                              \fill (8,3) circle (3pt) node[anchor=east]  {$w_2^2$} ; 
                               \fill (8,5) circle (3pt) node[anchor=east]  {$w_2^3$} ;  
                                    \sqtwoL(8,5,black);
                       \fill (8,7) circle (3pt) node[anchor=east]  {$w_2^4$} ;

                           \fill (12,1) circle (3pt) node[anchor=east]  {$U$} ;  
                                  \sqtwoL(12,1,black);
                              \fill (12,3) circle (3pt) node[anchor=east]  {$w_2 U$} ; 
                               \fill (12,5) circle (3pt) node[anchor=east]  {$w_2^2 U$} ;  
                                    \sqtwoL(12,5,black);
                       \fill (12,7) circle (3pt) node[anchor=east]  {$w_2^3 U$} ;  
                       
\end{tikzpicture}}
\caption{From the left, the structures of $H^*(BO_1)$,  $H^*(MO_1)$, $H^*(BU_1)$ and  $H^*(MU_1)$ as $\A_1$-modules.}
\label{fig:MO1MU1}
\end{figure}
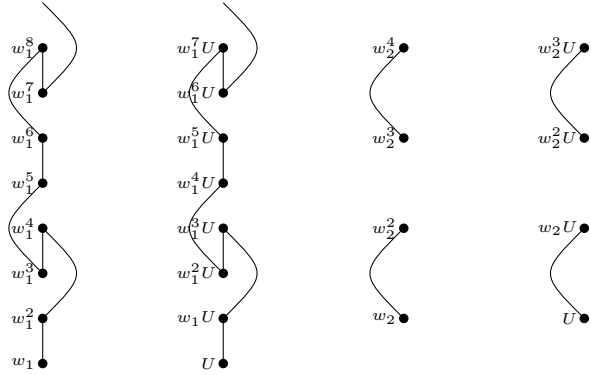
\end{ex}

\begin{ex}
As an exercise that will be relevant in \fullref{sec:examples}, we consider the structure of $H^*(MU_1 \smsh MO_1)$ as modules over $\A_1$. By the K\"unneth isomorphism, we have
\[ H^*(MU_1 \smsh MO_1) \cong H^*(MU_1) \otimes_{\Z/2 } H^*(MO_1). \]
We use the Cartan formula as discussed in \fullref{rem:twoforms}. Since all of the $Sq^1$s vanish in $H^*(MU_1)$, we deduce from the Cartan formula that for any $a \in H^*(MU_1)$ and $b \in H^*(MO_1)$,
\begin{align*}
Sq^1(a \otimes b) &= Sq^1(a) \otimes b + a \otimes Sq^1(b) = a \otimes Sq^1(b)  ,  \\
 Sq^2(a \otimes b) &=Sq^2(a) \otimes b +  Sq^1(a) \otimes Sq^1(b) + a \otimes Sq^2(b) = Sq^2(a) \otimes b +a \otimes Sq^2(b) .
\end{align*}
The $\A_1$-module structure is illustrated in a small range in \fullref{fig:MU1smshMO1}.

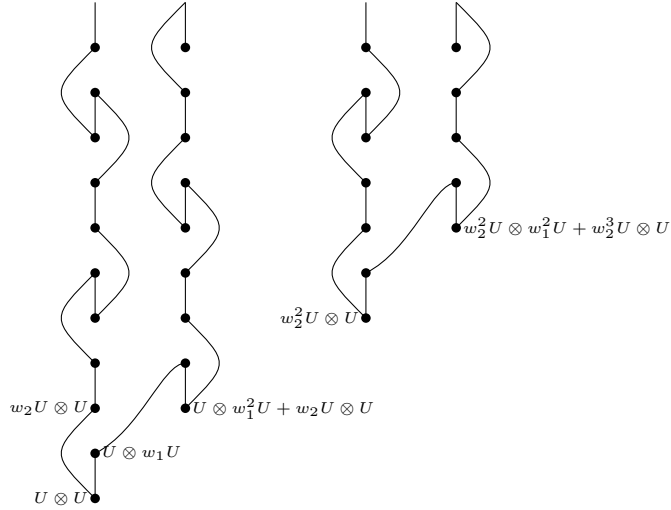
\begin{figure}[ht]
\center
{\tiny
  \begin{tikzpicture}[scale = .6]
        \fill (0,1) circle (3pt) node[anchor=east]  {$U \otimes U$};  
           \sqone(0,1,black);
         \sqtwoL(0,1,black);
          \fill (0,2) circle (3pt) node[anchor=west]  {$U\otimes w_1 U$};  
              \sqtwoCR(0,2,black);
        
                    \fill (0,3) circle (3pt) node[anchor=east]  {$w_2 U\otimes U$};  
                      
                                   \sqone(0,3,black);
                    \fill (0,4) circle (3pt);
                                \sqtwoL(0,4,black);
                                      
                       \fill (0,5) circle (3pt) ;  
                      \sqone(0,5,black);
                         \sqtwoR(0,5,black);
                         \fill (0,6) circle (3pt) ;  

                           \fill (0,7) circle (3pt) ;  
                              \sqone(0,7,black);
                               \fill (0,8) circle (3pt) ;  
                                 \sqtwoR(0,8,black);
                                  \fill (0,9) circle (3pt) ;  
                                   \sqone(0,9,black);
                                      \sqtwoL(0,9,black);
                                      \fill (0,10) circle (3pt) ;  
                                        \fill (0,11) circle (3pt) ; 
                                        \sqone(0,11,black); 
                           
\fill (2,3) circle (3pt) node[anchor=west]  {$U \otimes w_1^2U   +  w_2U \otimes U $};  
\sqone(2,3,black);
\sqtwoR(2,3,black);
\fill (2,4) circle (3pt) ;
\fill (2,5) circle (3pt) ;
\sqone(2,5,black);
\fill (2,6) circle (3pt);
\sqtwoR(2,6,black);
\fill (2,7) circle (3pt);
\sqtwoL(2,7,black);
\sqone(2,7,black);
\fill (2,8) circle (3pt);
\fill (2,9) circle (3pt);
\sqone(2,9,black);
\fill (2,10) circle (3pt);
\sqtwoL(2,10,black);
\fill (2,11) circle (3pt);
\sqone(2,11,black);

                              \fill (6,5) circle (3pt) node[anchor=east]  {$w_2^2U \otimes U $};  
           \sqone(6,5,black);
         \sqtwoL(6,5,black);
          \fill (6,6) circle (3pt) ;
       
              \sqtwoCR(6,6,black);
        
                    \fill (6,7) circle (3pt) ;

                                   \sqone(6,7,black);
                    \fill (6,8) circle (3pt)  ;  
                                \sqtwoL(6,8,black);
                                      
                       \fill (6,9) circle (3pt) ;  
                      \sqone(6,9,black);
                         \sqtwoR(6,9,black);
                         \fill (6,10) circle (3pt) ;

                           \fill (6,11) circle (3pt) ;  
                            \sqone(6,11,black); 
                           
\fill (8,7) circle (3pt) node[anchor=west]  {$w_2^2U \otimes w_1^2U   + w_2^3U \otimes U  $};
\sqone(8,7,black);
\sqtwoR(8,7,black);
\fill (8,8) circle (3pt) ;
\fill (8,9) circle (3pt) ;
\sqone(8,9,black);
\fill (8,10) circle (3pt);
\sqtwoR(8,10,black);
\fill (8,11) circle (3pt);
\sqone(8,11,black);                       
\end{tikzpicture}}
\caption{The $\A_1$-submodule of $H^*(MU_1 \smsh MO_1)$ generated by $U \otimes U$, $U\otimes  w_1^2U  + w_2U \otimes U  $,  $w_2^2U \otimes U  $ and $w_2^2U \otimes w_1^2U   + w_2^3U \otimes U$. The class $U\otimes U \in H^3( MU_1 \smsh MO_1)$. All classes of   of degree $* \leq 5$ in $H^*(\Sigma^{-3}  MU_1 \smsh MO_1)$ are contained in this submodule. }
\label{fig:MU1smshMO1}
\end{figure}
\end{ex}

\begin{ex}
In this example, we compute part of the structure of $H^*(BO_3)$ as a module over $\A_1$. We recall from \fullref{sec:SWclasses} that
\[H^*(BO_3;\Z/2) \cong \Z/2 [w_1,w_2,w_3]\]
and using the Wu formula, we compute that
\begin{align*}
Sq^1(w_1)&= w_1^2 & Sq^1(w_2) &= w_1w_2+ w_3 & Sq^1(w_3) &= w_1w_3  & \\
Sq^2(w_1) &=0 &  Sq^2(w_2) &= w_2^2 &Sq^2(w_3) &= w_2w_3.
\end{align*}
With the Cartan formula, this determines all of the operations for $\A_1$ on $H^*(BO_3)$. For example,
\begin{align*}
Sq^2(Sq^1(w_2)) &= Sq^2(w_1w_2)+ Sq^2(w_3) \\
& = Sq^2(w_1)Sq^0(w_2)+Sq^1(w_1)Sq^1(w_2) + Sq^0(w_1)Sq^2(w_2) +Sq^2(w_3)  \\
&= w_1^2Sq^1(w_2) + w_1w_2^2+w_2w_3   \\
&= (w_1^2+w_2)Sq^1(w_2).
\end{align*}
We let 
\begin{align}\label{eq:namex}
x&= Sq^1(w_2) =w_1w_2+ w_3, &
\w_2&= w_1^2+w_2.
\end{align} 
A part of the cell diagram for $H^*(BO_3)$ is depicted in \fullref{fig:HBO3}.
\end{ex}

\begin{ex}\label{ex:MO3}
To compute the structure of $H^*(MO_3)$ as a module over $\A_1$, we use the Thom isomorphism and \fullref{rem:twoforms}. The former gives the identification 
\[H^*(MO_3) \cong  \Z/2 [w_1, w_2,w_3]\{U\}\]
where the Thom class $U$ is in $H^3(MO_3)$. \fullref{rem:twoforms} allows us to compute the action of $\A_1$ on $H^*(MO_3)$ and the result is illustrated in \fullref{fig:HMO3}. For example, 
\begin{align*}
Sq^2(w_2U) &= Sq^2(w_2)U + Sq^1(w_2)Sq^1(U) + w_2 Sq^2(U) \\
&= w_2^2U+xw_1U+w_2^2U =w_1xU.
\end{align*}
A few other relations are given by
\begin{align*}
Sq^1(U) &=w_1U & Sq^1(w_1U)&=0 & Sq^1(w_2U)&=w_3 U  & Sq^1(w_3U)&=0   \\
Sq^2( U)&=w_2U     & Sq^2(w_1U) &= w_1\w_2U  & Sq^2(w_2U) &= w_1xU  & Sq^2(w_3U) &= w_1^2w_3 U .
\end{align*}
\end{ex}

\begin{figure}[ht]
\center
{\tiny
  \begin{tikzpicture}[scale = .6]
        \fill (0,1) circle (3pt) node[anchor=east]  {$w_1$} ;
          \sqone(0,1,black); 
        \fill (0,2) circle (3pt) node[anchor=east] {$w_1^2$};
        \sqtwoR(0,2,black);
           \fill (0,3) circle (3pt) node[anchor=east] {$w_1^3$};
             \sqone(0,3,black); 
             \sqtwoL(0,3,black);
               \fill (0,4) circle (3pt) node[anchor=east] {$w_1^4$} ;
                 \fill (0,5) circle (3pt) node[anchor=east] {$w_1^5$} ;
                   \sqone(0,5,black); 
                       \fill (0,6) circle (3pt) node[anchor=east] {$w_1^6$} ;
                             \sqtwoL(0,6,black);
                            \fill (0,7) circle (3pt) node[anchor=east] {$w_1^7$} ;
                               \sqone(0,7,black); 
                               \sqtwoR(0,7,black);
                                   \fill (0,8) circle (3pt) node[anchor=east] {$w_1^8$} ;
                                    \fill (0,9) circle (3pt) node[anchor=east] {$w_1^9$} ;
                                     \sqone(0,9,black);

  %%w2%%
    \fill (4,2) circle (3pt) node[anchor=east] {$w_2$} ;
        \sqone(4,2,black); 
          \sqtwoL(4,2,black); 
       \fill (4,3) circle (3pt) node[anchor=east] {$x$} ;
           \sqtwoR(4,3,black)
           \fill (4,4) circle (3pt) node[anchor=east] {$w_2^2$} ;
                 \sqtwoL(4,4,black); 
             \fill (4,5) circle (3pt) node[anchor=east] {$x\w_2$} ;
              \sqone(4,5,black); 
                \fill (4,6) circle (3pt) node[anchor=east] {$x^2$} ;
  %%w3%%
    \fill (8,3) circle (3pt) node[anchor=east] {$w_3$} ;
        \sqone(8,3,black); 
         \sqtwoL(8,3,black); 
            \fill (8,4) circle (3pt) node[anchor=west] {$w_1w_3$} ;
            \sqtwoCR(8,4,black);
                   \fill (8,5) circle (3pt) node[anchor=east] {$w_2w_3$} ;
                      \sqone(8,5,black); 
            \sqtwoCR(8,5,black);
               \fill (8,6) circle (3pt) node[anchor=east] {$w_3^2$} ;
                   \sqtwoL(8,6,black);
               \fill (8,7) circle (3pt) node[anchor=east] {$w_1w_3^2$} ;
                 \sqone(8,7,black); 
                 \sqtwoR(8,7,black);
                \fill (8,8) circle (3pt) node[anchor=east] {$w_1^2w_3^2$} ;
                 \fill (8,9) circle (3pt) node[anchor=east] {$w_1^3w_3^2$} ;
                  \sqone(8,9,black);

               \fill (10,6) circle (3pt) node[anchor=east] {$w_1w_3\w_2$} ;
                   \sqone(10,6,black); 
                 \fill (10,7) circle (3pt) node[anchor=east] {$w_1w_3x$} ;
                 
    %%w1w2^2%%
 \fill (14,4) circle (3pt) node[anchor=east] {$w_1^2w_2$} ;
\sqone(14,4,black); 
\sqtwoL(14,4,black);
 \fill (14,5) circle (3pt) node[anchor=east] {$w_1^2x$} ;
 \sqtwoCR(8,5,black);
\fill (14,5) circle (3pt) node[anchor=east] {$w_1^2x$} ;
\fill (14,6) circle (3pt) node[anchor=east] {$w_1^2w_2\w_2$} ;
\sqone(14,6,black); 
 \sqtwoCR(14,6,black);
 \fill (14,7) circle (3pt) node[anchor=east] {$w_1^4x$} ;
 \sqtwoCR(14,7,black);
  \fill (16,7) circle (3pt) node[anchor=east] {$w_1^2w_2x$} ;
    \sqone(16,7,black);
       \fill (16,8) circle (3pt) node[anchor=east] {$w_1^2 x^2$} ;
 \sqtwoR(16,8,black);
   \fill (16,9) circle (3pt) node[anchor=east] {$w_1^4x\w_2$} ;
       \sqone(16,9,black);
 \fill (16,10) circle (3pt) node[anchor=east] {$w_1^4x^2$} ;
\end{tikzpicture}}
\caption{The $\A_1$-submodule of $H^*(BO_3)$ generated by $w_1$, $w_2$, $w_3$ and $w_1^2w_2$. The class $w_1 \in H^1(BO_3)$ and $w_2 \in H^2(BO_3)$.}
\label{fig:HBO3}
\end{figure}

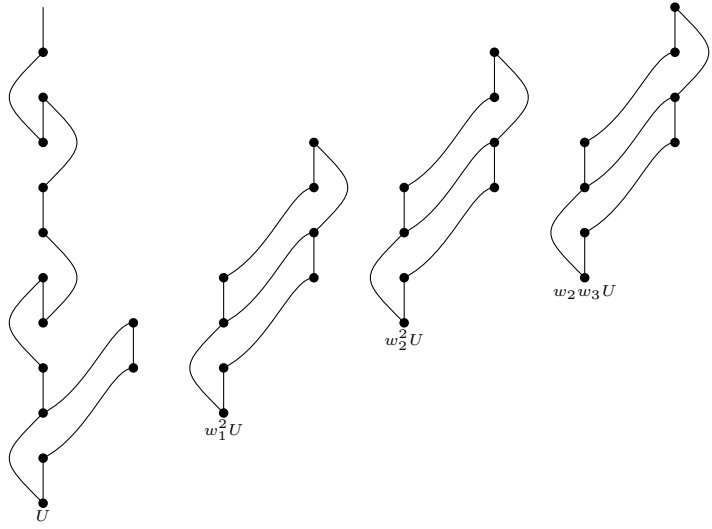
\begin{figure}[ht]
\center
{\tiny
 \begin{tikzpicture}[scale = .6]
    \fill (0,0) circle (3pt) node[anchor=north] {$U$} ;
        \sqone(0,0,black); 
         \sqtwoL(0,0,black); 
            \fill (0,1) circle (3pt) ;
            \sqtwoCR(0,1,black);
                   \fill (0,2) circle (3pt)  ;
                      \sqone(0,2,black); 
            \sqtwoCR(0,2,black);
               \fill (0,3) circle (3pt) ;
               \sqtwoL(0,3, black);
               \fill (0,4) circle (3pt)   ;
               \sqone(0,4,black); 
                    \sqtwoR(0,4, black);
               \fill (0,5) circle (3pt) ;
                 \fill (0,6) circle (3pt) ;
                   \sqone(0,6,black); 
                     \fill (0,7) circle (3pt);
                     \sqtwoR(0,7,black);
                        \fill (0,8) circle (3pt) ;
                        \sqone(0,8,black); 
                           \sqtwoL(0,8,black);
                                 \fill (0,9) circle (3pt) ;
                                    \fill (0,10) circle (3pt) ;
                                        \sqone(0,10,black); 
                     
                      \fill (2,3) circle (3pt);
                             \sqone(2,3,black); 
                               \fill (2,4) circle (3pt) ;
   \Aone(12,5) ;
   \Aone(4,2) ;
   \node[anchor=north] at (4,2) {$w_1^2U$} ;
                                     \Aone(8,4) ;
   \node[anchor=north] at (8,4) {$w_2^2U$} ;       
      \node[anchor=north] at (12,5) {$w_2 w_3U$} ;           
\end{tikzpicture}}
\caption{The $\A_1$-submodule of $H^*(MO_3)$ generated by the classes $U$, $w_1^2U$, $w_2^2U$ and $w_2w_3U$. The class $U \in H^3(MO_3)$. This submodule contains all cohomology classes in $H^*(\Sigma^{-3}MO_3)$ of degree $*\leq 5$.}
\label{fig:HMO3}
\end{figure}

\begin{ex}\label{ex:MTO3}
We turn to the computation of part of the structure of the cohomology of $MTO_3$ as a module over $\A_1$.
Recall that $MTO_3$ is the Thom space for the virtual bundle $-\gamma_3$ over $BO_3$. Again, we have a Thom isomorphism
\[H^*(MTO_3) \cong  \Z/2 [w_1, w_2,w_3]\{\U\} \]
where the Thom class $\U= U(-\gamma_3)$ is in degree $-3$. However, here $w_i = w_i(\gamma_3)$, the Stiefel--Whitney classes of the universal bundle. Let $\w_i = w_i(-\gamma_3)$.
To compute the Steenrod operations using the formula
\[Sq^i(\U) = \w_i  \U,\]
of \fullref{defn:SW}, we need a formula for the $\w_i$s in terms of the $w_i$s. Letting $w=w(\gamma_3)$ and $\w = w(-\gamma_3)$ be the total Stiefel--Whitney classes, \fullref{rem:totalSW} gives an identity
\begin{align*}
\w &=w^{-1}= \frac{1}{1+w_1+w_2+w_3} = \sum_{i\geq 0} (w_1+w_2+w_3)^i .
\end{align*}
Collecting the terms of the same degree, we get that
\begin{align*}
\w_1 &= w_1 \\
\w_2 &= w_1^2+w_2  .
\end{align*}
Therefore, a few relations are given by
\begin{align*}
Sq^1(\U) &=w_1\U, &Sq^1(w_1\U)&= 0 & Sq^1(w_2\U)&= w_3 \U & Sq^1(w_3\U)&= 0  \\
Sq^2( \U)&=(w_1^2+w_2)\U  & Sq^2(w_1\U) &= w_1w_2 \U & Sq^2(w_2\U) &= w_1w_3 \U    & Sq^2(w_3\U) &=   0. 
\end{align*}
A part of the cell diagram for $H^*(MTO_3)$ is depicted in \fullref{fig:HMTO3}.
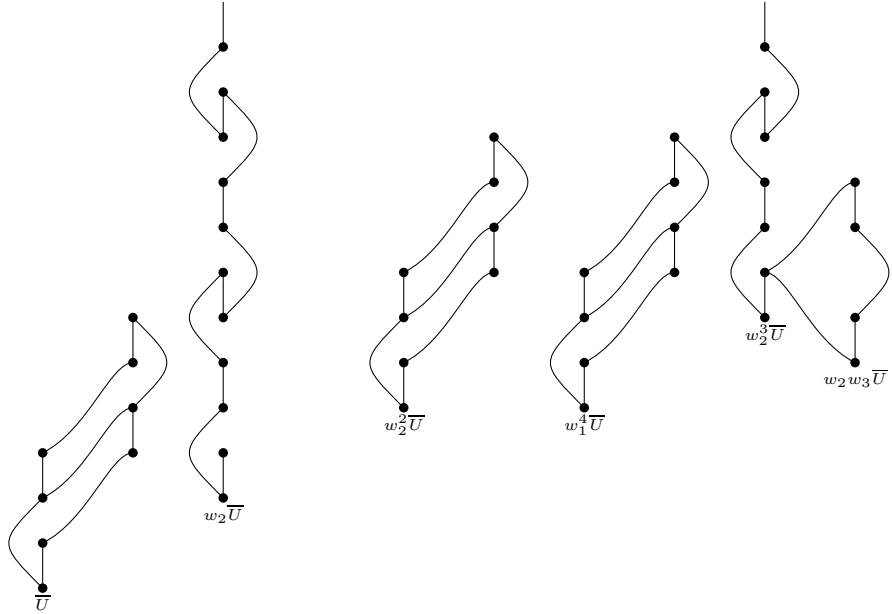
\begin{figure}[ht]
\center
 {\tiny \begin{tikzpicture}[scale = .6]
  
 %%U%%%%
 \Aone(0,0);
    \node[anchor=north] at (0,0) {$\U$} ;

         \fill (4,2) circle (3pt) ;
             \sqone(4,2,black); 
             \sqtwoL(4,2,black);
               \fill (4,3) circle (3pt)  ;
                 \fill (4,4) circle (3pt)  ;
                   \sqone(4,4,black); 
                       \fill (4,5) circle (3pt) ;
                             \sqtwoL(4,5,black);
                            \fill (4,6) circle (3pt)  ;
                               \sqone(4,6,black); 
                               \sqtwoR(4,6,black);
                                   \fill (4,7) circle (3pt)   ;
                                    \fill (4,8) circle (3pt) ;
                                     \sqone(4,8,black); 
                                    \fill (4,9) circle (3pt) ;
                                    \sqtwoR(4,9,black); 
                                      \fill (4,10) circle (3pt) ;
                                      \sqone(4,10,black); 
                                        \fill (4,11) circle (3pt) ;
                                            \sqtwoL(4,10,black); 
                                            
                                          \fill (4,12) circle (3pt) ;
                                                \sqone(4,12,black); 
    \node[anchor=north] at (4,2) {$w_2\U$} ;

  \Aone(8,4);
    \node[anchor=north] at (8,4) {$w_2^2\U$} ;

  \Aone(12,4);
    \node[anchor=north] at (12,4) {$w_1^4\U$} ;

         \fill (16,6) circle (3pt) ;
             \sqone(16,6,black); 
             \sqtwoL(16,6,black);
               \fill (16,7) circle (3pt)  ;
                 \fill (16,8) circle (3pt)  ;
                   \sqone(16,8,black); 
                       \fill (16,9) circle (3pt) ;
                             \sqtwoL(16,9,black);
                            \fill (16,10) circle (3pt)  ;
                               \sqone(16,10,black); 
                               \sqtwoR(16,10,black);
                                   \fill (16,11) circle (3pt)   ;
                                    \fill (16,12) circle (3pt) ;
                                     \sqone(16,12,black);

   \fill (18,5) circle (3pt) ;
   \sqtwoCL(18,5, black);
   \sqone(18,5,black);
      \fill (18,6) circle (3pt) ;
         \sqtwoR(18,6, black);
           \fill (18,8) circle (3pt) ;
              \sqone(18,8,black);
                \fill (18,9) circle (3pt) ;
                
                   \sqtwoCR(16,7, black);
                
                    \node[anchor=north] at (18,5) {$w_2w_3\U$} ;
                       \node[anchor=north] at (16,6) {$w_2^3\U$} ;
  
    \end{tikzpicture}}
\caption{The $\A_1$-submodule of $H^*(MTO_3)$ generated by the classes $\U$, $w_2\U$, $w_2^2\U$, $w_1^4\U$, $w_2^3 \U $ and $w_2w_3\U$. The class $\U \in H^{-3}(MTO_3)$. This submodule contains all cohomology classes in $H^*(\Sigma^{3} MTO_3)$ of degree $*\leq 5$.}
\label{fig:HMTO3}
\end{figure}
\end{ex}

\begin{exc}
Use the formulas of \fullref{ex:MO3} and \fullref{rem:twoforms} to compute that the $\A_1$-submodule of $H^*(MO_3)$ generated by $U$, $w_1^2U$ and $w_2^2U$ has the structure depicted in \fullref{fig:HMO3}. Do the same thing for \fullref{fig:HMTO3} using the results of \fullref{ex:MTO3}.
\end{exc}

\begin{ex}
In this example, we compute the structure of $H^*(MSO_3)$ as a module over $\A_1$. Let $\iota \colon MSO_3 \to MO_3$ be the map of Thom spectra induced by the inclusion of $SO_3$ into $O_3$. The induced map $ \iota^* \colon H^*(MO_3)  \to H^*(MSO_3)$ is given by moding out $w_1$. The Thom class of $\gamma_3$ maps to that of the universal bundle on $BSO_3$. We get an isomorphism
\[H^*(MSO_3) \cong  \Z/2 [w_2,w_3]\{U\} .\]
Further, the Steenrod operations are natural with maps of spaces or spectra, so $Sq^k\iota^* =\iota^* Sq^k$.

We use that $x \equiv w_3 \mod (w_1)$ for $x$ as in \eqref{eq:namex}. We get the following formulas from \fullref{ex:MO3}. First, in the cohomology of $BSO_3$, we have 
\begin{align*}
Sq^1(w_2)&= w_3  & Sq^1(w_3) &= 0  & Sq^2(w_2)&= w_2^2 & Sq^2(w_3) &=w_2w_3   .
\end{align*}
So, in the cohomology of $MSO_3$, we have
\begin{align*}
 Sq^1(U) &=0 & Sq^1(w_2U)&=w_3 U  & Sq^1(w_3U)&=0   \\
 Sq^2( U)&=w_2U  & Sq^2(w_2U) &= 0  & Sq^2(w_3U) &= 0 .
\end{align*}

\begin{figure}[ht]
\center
{\tiny
 \begin{tikzpicture}[scale = .6]
    \fill (0,0) circle (3pt) node[anchor=west] {$U$} ;
    \sqtwoL(0,0,black);
       \fill (0,2) circle (3pt) node[anchor=west] {$w_2 U$} ;
         \sqone(0,2,black);
       \fill (0,3) circle (3pt) node[anchor=west] {$w_3 U$} ;
       
            \fill (2,4) circle (3pt) node[anchor=west] {$w_2^2 U$} ;
                      \sqtwoCR(2,4,black);
         \fill (4,6) circle (3pt) node[anchor=west] {$(w_3^2+w_2^3) U$} ;
                   \sqone(4,6,black);
                \fill (4,7) circle (3pt) node[anchor=west] {$w_2^2w_3 U$} ;
                    \sqtwoR(4,7,black);
                    \fill (4,9) circle (3pt) node[anchor=west] {$w_3^3 U$} ;
             
              \fill (2,5) circle (3pt) node[anchor=east] {$w_2w_3 U$} ;
                  \sqone(2,5,black);
                  \sqtwoCR(2,5,black);
                \fill (2,6) circle (3pt) node[anchor=east] {$w_3^2 U$} ;
                  \sqtwoCR(2,6,black);
                   \fill (4,8) circle (3pt) node[anchor=east] {$w_2w_3^2 U$} ;
                       \sqone(4,8,black);
                       
                           \fill (8,8) circle (3pt) node[anchor=west] {$w_2^4U$} ;
    \sqtwoL(8,8,black);
       \fill (8,10) circle (3pt) node[anchor=west] {$w_2^5 U$} ;
         \sqone(8,10,black);
       \fill (8,11) circle (3pt) node[anchor=west] {$w_2^4w_3 U$} ;
             
\end{tikzpicture}}
\caption{The $\A_1$-submodule of $H^*(MSO_3)$ generated by the classes $U$, $w_2^2U$, $w_2w_3U$ and $w_2^4U$. The class $U \in H^3(MSO_3)$. This submodule contains all cohomology classes in $H^*(\Sigma^{-3}MSO_3)$ of degree $*\leq 5$.}
\label{fig:HMSO3}
\end{figure}
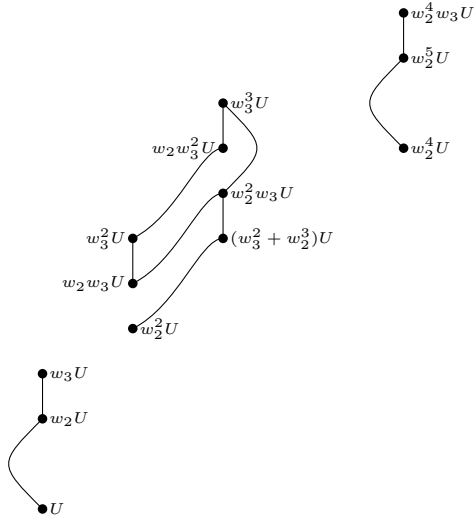

\end{ex}

% !TEX root = cbms-master.tex

\section{The Adams spectral sequence}\label{sec:ass}
One of the most effective methods for computing stable homotopy groups is the Adams spectral sequence. The idea is roughly as follows. Take a space or a spectrum $X$ and resolve it into pieces whose homotopy we understand. The Eilenberg--MacLane spectra are good candidates --- they are constructed to have homotopy in a single degree. Then, reconstruct the stable homotopy groups of $X$ from algebraic data associated to this resolution.

We will make this more precise and give a sketch of the construction of the Adams spectral sequence. In the cases of interest, it has the form
\begin{equation}\label{eq:ASSdisplay}
E_2^{s,t} = \Ext_{\A}^{s,t}(H^*(X), \Z/2 ) \Longrightarrow (\pi_{t-s}X)_2^{\wedge}\end{equation}
We will explain the terms in \eqref{eq:ASSdisplay} throughout this section. We begin by defining $\Ext_{\A}$ and giving tools to compute it.

\subsection{Computing $\Ext$ over the Steenrod algebra}\label{sec:compext}
Let $\cB$ be a graded ring. For any $\cB$-module $M$ and $r \in \Z$, let $\Sigma^r M = M[r]$ be the graded $\cB$-module given in degree $t$ by
\[(\Sigma^rM)^t  = (M[r])^t =M^{t-r}. \]
Let $\Hom_{\cB}^*(M, N)$ be the graded abelian group given in degree $t$ by
\[\Hom_{\cB}^t(M, N)  = \Hom_{\cB}(M, \Sigma^t N).\]
The contravariant functor
\[\Hom_{\cB}^*(-, N) \colon \cB\text{-}\Mod \to \Ab  \] 
is left exact and has right derived functors $\Ext^s_{\cB}(-,N)$. We let
\[\Ext^{s,t}_{\cB}(-,N) = (\Ext^s_{\cB}(-,N))^t \]
and treat $\Ext_{\cB}^{*,*}(-,N)$ as a functor with values in bi-graded abelian groups.  As always, the value of these functors on a $\cB$-module $M$ can be computed by choosing a resolution $P_{\bullet}$ of $M$ by projective $\cB$-modules and forming the cochain complex $\Hom_{\cB}^*(P_\bullet, N)$. Then 
\[ \Ext^{s,t}_{\cB}(M,N) = H^s(\Hom_{\cB}^t(P_\bullet, N) ).\]

A useful tool is the interpretation of elements in $\Ext^{s,t}_{\cB}(M,N)$ as equivalence classes of extensions when $s\geq 1$. That is, an element of $\Ext^{s,t}_{\cB}(M,N)$ is an exact complex, or \emph{extension},
\[ \xymatrix{ 0 \ar[r] & \Sigma^t N \ar[r] & P_1 \ar[r] & \ldots  \ar[r] & P_s \ar[r] & M \ar[r] & 0  }  \]
where two extensions are equivalent if there exists a commutative diagram
\[ \xymatrix{ 0 \ar[r] &\Sigma^{t}  N \ar[d]^-{\id_N}  \ar[r] & P_1 \ar[d]^-{\cong}   \ar[r] & \ldots  \ar[r] & P_s \ar[d]^-{\cong}  \ar[r] &  M \ar[r] \ar[d]^-{\id_M}  & 0 \\
0 \ar[r] & \Sigma^{t}   N \ar[r] & P_1'  \ar[r] & \ldots  \ar[r] & P_s' \ar[r] &  M \ar[r] & 0 .  }  \]

\begin{ex}\label{ex:h0h1def}
The class in $\Ext_{\A}^{1,1}(\Z/2 ,\Z/2 )$ represented by the extension
\[ 0 \to \Sigma \Z/2  \to \Sigma^{-1} H^*(\R P^2) \to \Z/2  \to 0,\]
which is depicted in \fullref{fig:h0}, is called $h_0$.  The class in $\Ext_{\A}^{1,2}(\Z/2 ,\Z/2 )$ represented by the extension
\[ 0 \to \Sigma^{2}\Z/2  \to \Sigma^{-2} H^*(\C P^2) \to \Z/2  \to 0,\]
which is depicted in \fullref{fig:h1}, is called $h_1$.

\begin{figure}[ht]
{\tiny
\center
 \begin{tikzpicture}[scale = .75]
   \fill (0,1) circle (3pt) ;   
   \fill (4,0) circle (3pt) ;
            \sqone(4,0,black);
      \fill (4,1) circle (3pt) ;
            \fill (8,0) circle (3pt) ;
 \draw[blue, ->] (0,1) -- (4,1);
  \draw[blue, ->] (4,0) -- (8,0);
               \node[anchor=north] at (0,1) {$\Sigma \Z/2 $} ;
               \node[anchor=north] at (4,0) {$ \Sigma^{-1} H^*(\R P^2)$} ;
                  \node[anchor=north] at (8,0) {$\Z/2 $} ;
 \end{tikzpicture}} 
\caption{The extension representing $h_0$ in $\Ext_{\A}^{1,1}(\Z/2 ,\Z/2 )$.}
\label{fig:h0}
\center
{\tiny
  \begin{tikzpicture}[scale = .75]
   \fill (0,2) circle (3pt) ;   
   \fill (4,0) circle (3pt) ;
            \sqtwoR(4,0,black);
      \fill (4,2) circle (3pt) ;
            \fill (8,0) circle (3pt) ;
 \draw[blue, ->] (0,2) -- (4,2);
  \draw[blue, ->] (4,0) -- (8,0);
               \node[anchor=north] at (0,2) {$\Sigma^2 \Z/2 $} ;
               \node[anchor=north] at (4,0) {$ \Sigma^{-2} H^*(\C P^2)$} ;
                  \node[anchor=north] at (8,0) {$\Z/2 $} ;
 \end{tikzpicture}} 
\caption{The extension representing $h_1$ in $\Ext_{\A}^{1,2}(\Z/2 ,\Z/2 )$.}
\label{fig:h1}
\end{figure}
\end{ex}

\subsection{Module Structure on $\Ext$}\label{sec:multiplicative}
Let $\cB$ be a sub-Hopf algebra of the Steenrod algebra  $\A$. Then for any $\cB$-module $M$, there is a map
\[  \Ext_{\cB}^{s,t}(M,\Z/2 ) \otimes_{\Z/2} \Ext_{\cB}^{s',t'}(\Z/2 ,\Z/2 )  \to \Ext_{\cB}^{s+s',t+t'}(M,\Z/2 ) .\]
This is called the \emph{Yoneda product}. It is straightforward to describe the product in terms of extensions. Suppose that $s,s' \geq 1$. Given two extensions 
\begin{equation}
\label{eq:elementM}
 \xymatrix{ 0 \ar[r] & \Sigma^t \Z/2  \ar[r] & P_1 \ar[r] & \ldots  \ar[r] & P_s \ar[r] & M \ar[r] & 0  }  
  \end{equation}
and
\begin{equation}
\label{eq:element}\xymatrix{ 0 \ar[r] & \Sigma^{t'} \Z/2  \ar[r] & Q_1 \ar[r] & \ldots  \ar[r]^-{\varphi_{s'}} & Q_{s'} \ar[r] & \Z/2  \ar[r] & 0  ,}  \end{equation}
where \eqref{eq:elementM} represents an element of  $\Ext_{\cB}^{s,t}(M,\Z/2 ) $ and \eqref{eq:element} an element of $\Ext_{\cB}^{s',t'}(\Z/2 ,\Z/2 ) $, 
we can splice the complexes to obtain an extension of length $s+s'$:
\[ \xymatrix@-1pc{ 0 \ar[r] & \Sigma^{t'+t} \Z/2  \ar[r] & \Sigma^t Q_1 \ar[r] & \ldots  \ar[r] & \Sigma^t  Q_{s'} \ar[rr] \ar[dr]  & & P_1 \ar[r] & \ldots  \ar[r] & P_s \ar[r] & M \ar[r] & 0  \\
 & & & &   & \Sigma^t \Z/2 \ar[ur] &  &     }  \]
 which represents the product in $ \Ext_{\cB}^{s+s',t+t'}(M,\Z/2 )$.
This defines the module structure for elements of degree $s\geq 1$ in $\Ext_{\cB}^{s,t}(M,\Z/2 )$. If $s=0$, then given a homomorphism $ M \to \Sigma^t \Z/2 $ in 
\[\Ext_{\cB}^{0,t}(M, \Z/2 ) \cong \Hom_{\cB}(M, \Sigma^t \Z/2)\] 
and an element of $\Ext_{\cB}^{s',t'}(\Z/2 ,\Z/2 )$ represented by \eqref{eq:element}, we obtain an element in $\Ext_{\cB}^{s',t+t'}(M, \Z/2 )$ represented by
\[\xymatrix@-0.5pc{0 \ar[r] & \Sigma^{t'} \Z/2  \ar[r] & Q_1 \ar[r] & \ldots  \ar[r]  & \Sigma^t Q_{s'-1} \ar[r] &  \Sigma^tQ_{s'} \times_{ \Sigma^{t}\Z/2 } M  \ar[r] & M \ar[r] & 0  } \]
where $\Sigma^t Q_{s'} \times_{ \Sigma^{t}\Z/2 } M  $ is the pull-back of $\A_1$-modules. There is a commutative diagram of exact sequences
\[\xymatrix{ 0 \ar[r] & \ker( \Sigma^t\varphi_{s'}) \ar[r] \ar[d] &  \Sigma^t Q_{s'} \times_{ \Sigma^{t}\Z/2 } M  \ar[d] \ar[r]  & M \ar[d] \ar[r] & 0 \\ 
 0 \ar[r] & \ker( \Sigma^t\varphi_{s'}) \ar[r] &   \Sigma^t Q_{s'} \ar[r] &\Sigma^t \Z/2  \ar[r]  & 0 }\]
 so that we really do get an exact complex. An example when $\cB =\A_1$ is given in \fullref{fig:extmultforhom}.
 
 \begin{figure}[ht]
 \begin{tikzpicture}[scale = .5]
            \fill (-6,0) circle (3pt) ;
             \fill (-6,2) circle (3pt) ;
                      \sqtwoR(-6,0,black);
                      \fill (-4,-0) circle  (3pt) ;
   \draw[blue, ->] (-6,0) -- (-4,0);
   \node[anchor=north] at (-5,-0.5) {$\A_1 /\!\! /\E_1 \to \Z/2 $} ;
      \node[anchor=north] at (-5,-2) {$\varphi$} ;
 
   \fill (0,1) circle (3pt) ;   
   \fill (4,0) circle (3pt) ;
            \sqone(4,0,black);
            \sqtwoR(4,0,black);
      \fill (4,1) circle (3pt) ;
         \fill (4,2) circle (3pt) ;
            \fill (8,0) circle (3pt) ;
             \fill (8,2) circle (3pt) ;
                      \sqtwoR(8,0,black);
 \draw[blue, ->] (0,1) -- (4,1);
  \draw[blue, ->] (4,0) -- (8,0);
  
     \node[anchor=north] at (4,-0.5) { $\Sigma \Z/2  \to  \A_1 /\!\! /\E_1 \times_{\Z/2 }\Sigma \Z/2  \to \A_1 /\!\! /\E_1 $} ;
        \node[anchor=north] at (4,-2) {$\varphi h_0$} ;
  
 \end{tikzpicture}
 \caption{A representative extension for the element $ \varphi h_0$ in $\Ext_{\A_1}^{1,1}( \A_1 /\!\! /\E_1, \Z/2 )$ where the element $\varphi \colon \A_1 /\!\! /\E_1 \to \Z/2 $ of $\Ext_{\A_1}^{0,0}( \A_1 /\!\! /\E_1, \Z/2 )$ is the map which sends the element of degree two to zero.}
 \label{fig:extmultforhom}
 \end{figure}

\subsection{Adams Charts}\label{sec:achart}
For $\cB$ a sub-Hopf algebra of $\A$, we depict the information contained in $ \Ext_{\cB}^{s,t}(M, \Z/2 )$ in a picture which we call an \emph{Adams chart}. See \fullref{fig:achart}. 
An Adams chart is an illustration of $ \Ext_{\cB}^{s,t}(M, \Z/2 )$ in the $(t-s,s)$-plane. A generator for a copy of $\Z/2 $ in $ \Ext_{\cB}^{s,t}(M, \Z/2 )$ is denoted by a $\bullet$. Multiplication by $h_0$ is recorded by drawing a vertical line between two classes and multiplication by $h_1$ by a line of slope $(1,1)$. An infinite string of classes connected by multiplications by $h_0$ is called an \emph{$h_0$-tower}. Note that the Adams chart for $\Sigma^rM$ is the same as that of $M$, but horizontally shifted to the right by $r$.
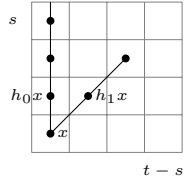
\begin{figure}[ht]
\center
 {\tiny
  \begin{tikzpicture}[scale=.5]
    \draw[step=1cm,gray,very thin] (0,0) grid (4,4);
    \fill (0.5,0.5) circle (3pt);
        \draw (0.5,0.5) -- (0.5,1.5);
           \fill (0.5,1.5) circle (3pt);
                  \draw (0.5,1.5) -- (0.5,2.5);
           \fill (0.5,2.5) circle (3pt);
                    \draw (0.5,2.5) -- (0.5,3.5);
           \fill (0.5,3.5) circle (3pt);
                \draw (0.5,3.5) -- (0.5,4);
                
                  \fill (1.5,1.5) circle (3pt);
                     \draw (0.5,0.5) -- (1.5,1.5);
                         \fill (2.5,2.5) circle (3pt);
                     \draw (1.5,1.5) -- (2.5,2.5);

  \node[anchor=west] at (0.5,0.5) {$x$} ;
    \node[anchor=east] at (0.5,1.5) {$h_0x$} ;
        \node[anchor=west] at (1.5,1.5) {$h_1x$} ;
                  
  \node at (-0.5,3.5) {$s$} ;
  \node at (3.5,-0.5) {$t-s$} ;
 \end{tikzpicture}} 
\caption{An example of an Adams chart for $\Ext_{\cB}^{s,t}(M,\Z/2 )$.}
\label{fig:achart}
\end{figure}

\subsection{Minimal Resolutions}\label{sec:MinRes}
Let $\cB$ be a sub-Hopf algebra of $\A$. Recall that $\A$ is an augmented algebra with $\A_0 = \Z/2 $. So this holds for any of its subalgebras. We let $I(\cB)$ be the kernel of the augmentation of $\cB$. Note that for any $\cB$-module $P$ and $\Z/2 $ the trivial $\cB$-module, the map
\[   \Hom_{\cB}^*(P, \Z/2  ) \to \Hom_{\cB}^*(I(\cB)P, \Z/2 )  \]
induced by the inclusion $I(\cB)P \hookrightarrow P$ is zero.
So, if $P_{\bullet}$,
\[\ldots \to P_s \xra{f_s} P_{s-1} \to \ldots \to P_0 \to M\] 
is a projective resolution of $M$ which satisfies
\[f_s(P_s) \subseteq I(\cB)P_{s-1},\]
then the maps in the cochain complex $\Hom_{\cB}^*(P_{\bullet},\Z/2 )$ are trivial and it follows that
\[ \Ext_{\cB}^{s,t}(M, \Z/2 ) \cong \Hom_{\cB}^t(P_s, \Z/2 ).\]
Such a resolution is called a \emph{minimal resolution}. 

If $M$ is a $\cB$-module which is bounded below, then $M$ has a minimal resolution by free $\cB$-modules. In such a resolution $P_{\bullet} \to M$, the $P_s$ are direct sums of suspensions of $\cB$ and $\Ext_{\cB}^{s,t}(M, \Z/2)$ is a product of $\Z/2$s indexed over the summands $\Sigma^t \cB \subseteq P_s$. If there are finitely many of these, the product is isomorphic to a direct sum and each summand corresponds to a generator in $\Ext_{\cB}^{s,t}(M, \Z/2)$.

If $\cB$ and $M$ are small, these are straightforward to construct and we do a few examples here in the case when $\cB=\A_1$.

\begin{rem}\label{rem:range}
Let $\cB$ a sub-Hopf algebra of $\A$ and $M$ a graded $\cB$-module of finite type which is zero in degrees $t<n$.
Using a free minimal resolution of $M$ to compute $\Ext_{\cB}^{s,t}(M, \Z/2 )$, one deduces that $\Ext_{\cB}^{s,t}(M, \Z/2 ) =0$ for $t-s<n$. 
\end{rem}

\begin{ex}\label{ex:M1}
We begin by constructing a resolution of the $\A_1$-module $M_0= \A_1  /\!\!/  \E_0$, where $\E_0$ is the subalgebra generated by $Sq^1$. It is depicted below. This example is also treated by a different method in \fullref{ex:corM1}. The module $M_0$ has a periodic minimal resolution of the form
\begin{equation}\label{eq:resM} \xymatrix{ M_0  & \A_1 \ar[l] & \Sigma \A_1 \ar[l] &  \Sigma^2 \A_1  \ar[l] & \ldots \ar[l] } \end{equation}
See \fullref{fig:minM0}. The horizontal (blue) arrows indicate the maps in \eqref{eq:resM}. The circled  (in red) classes are in the kernel. We have redrawn the kernels to the right (in red) to make the next map easier to visualize. The duals of the boxed classes (in blue) will form a basis of $\Ext_{\A_1}^{*,*}(M_0, \Z/2 )$. Recall that $h_0$ was defined in \fullref{ex:h0h1def}. See also \fullref{fig:h0}. The class in $\Ext^{1,1}_{\A_1}(M_0, \Z/2 )$ is the $h_0$ multiple of the class in $\Ext^{0,0}_{\A_1}(M_0, \Z/2 )$. This is read off of the part of \fullref{fig:minM0} that has been framed (in gray).

The $\A_1$-module $M_0$ is not the restriction of any $\A$-module, but it has such a nice projective resolution that it is often used as a tool to compute resolutions for other modules. This will be explained below.
There are larger versions of the module $M_0$ that we will denote by $M_n$ obtained by stringing together copies of $M_0$, including the case $n = \infty$.
For example, $M_1$ is drawn in \fullref{fig:M1}. These all have periodic minimal resolutions. For example,
\begin{equation*} \xymatrix{ M_1  & \A_1 \oplus \Sigma^4\A_1 \ar[l] & \Sigma(  \A_1 \oplus \Sigma^4\A_1 )  \ar[l] &  \Sigma^2 ( \A_1 \oplus \Sigma^4\A_1 )  \ar[l] & \ldots \ar[l] } \end{equation*}
The Adams chart of $M_n$ has $h_0$-towers starting in $(4k,0)$ for $0 \leq k \leq n$. For example, the Adams chart for $M_1$ is depicted in \fullref{fig:M1}.

\begin{figure}[ht]
  \begin{minipage}{.4\textwidth}
    \centering \captionsetup{width=0.9\linewidth}
 \begin{tikzpicture}[scale = .6]
   \fill (0,0) circle (3pt) ;
   \sqtwoR(0,0,black);
      \fill (0,2) circle (3pt) ;
         \sqone(0,2,black);
      \fill (0,3) circle (3pt) ;
   \sqtwoR(0,3,black);
         \fill (0,5) circle (3pt) ;
  \end{tikzpicture}
 \caption{The $\A_1$ module $M_0$.}
 \label{fig:M0}
 \end{minipage}
  \begin{minipage}{.55\textwidth}
    \centering
    \centering \captionsetup{width=0.9\linewidth}
  \begin{tikzpicture}[scale=.5]
    \draw[step=1cm,gray,very thin] (0,0) grid (8,7);
    \fill (0.5,0.5) circle (3pt);
        \draw (0.5,0.5) -- (0.5,1.5);
           \fill (0.5,1.5) circle (3pt);
                  \draw (0.5,1.5) -- (0.5,2.5);
           \fill (0.5,2.5) circle (3pt);
                    \draw (0.5,2.5) -- (0.5,3.5);
           \fill (0.5,3.5) circle (3pt);
                        \draw (0.5,3.5) -- (0.5,4.5);
           \fill (0.5,4.5) circle (3pt);
                            \draw (0.5,4.5) -- (0.5,5.5);
           \fill (0.5,5.5) circle (3pt);
                           \draw (0.5,5.5) -- (0.5,6.5);
           \fill (0.5,6.5) circle (3pt);
                \draw (0.5,6.5) -- (0.5,7);
 \end{tikzpicture}
 \caption{The Adams chart for $\Ext_{\A_1}^{s,t}(M_0, \Z/2 )$.}
 \label{fig:M0achart}
 \end{minipage}
\end{figure}

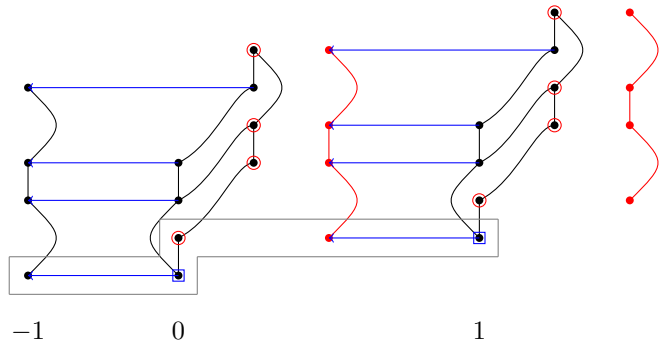
\begin{figure}[ht]
\center
\begin{tikzpicture}[scale=.5] 
   \fill (0,0) circle (3pt) ;
   \sqtwoR(0,0,black);
      \fill (0,2) circle (3pt) ;
         \sqone(0,2,black);
      \fill (0,3) circle (3pt) ;
   \sqtwoR(0,3,black);
         \fill (0,5) circle (3pt) ;
         
           \Aone (4, 0);
             \draw[red] (4, 1) circle (5pt);
  \draw[red] (6, 3) circle (5pt);
  \draw[red] (6, 4) circle (5pt);
  \draw[red] (6, 6) circle (5pt);
    \draw[blue, ->] (4,0) -- (0, 0);
    \draw[blue, ->] (4,2) -- (0,2);
        \draw[blue, ->] (4,3) -- (0,3);
              \draw[blue, ->] (6,5) -- (0,5);
           
        \fill[red] (8,1) circle (3pt) ;
   \sqtwoR(8,1,red);
      \fill[red] (8,3) circle (3pt) ;
         \sqone(8,3,red);
      \fill[red] (8,4) circle (3pt) ;
   \sqtwoR(8,4,red);
         \fill[red] (8,6) circle (3pt) ;      
         
\Aone (12, 1);
\draw[red] (12, 2) circle (5pt);
\draw[red] (14, 4) circle (5pt);
\draw[red] (14, 5) circle (5pt);
\draw[red] (14, 7) circle (5pt);
\draw[blue, ->] (12,1) -- (8, 1);
\draw[blue, ->] (12,3) -- (8,3);
\draw[blue, ->] (12,4) -- (8,4);
\draw[blue, ->] (14,6) -- (8,6);
              
\fill[red] (16,2) circle (3pt) ;
\sqtwoR(16,2,red);
\fill[red] (16,4) circle (3pt) ;
\sqone(16,4,red);
\fill[red] (16,5) circle (3pt) ;
\sqtwoR(16,5,red);
\fill[red] (16,7) circle (3pt) ;     
         
\rectangle(4,0,blue);
\rectangle(12,1,blue);
\node[anchor=north] at (0,-1) {$-1$} ;
\node[anchor=north] at (4,-1) {$0$} ;
\node[anchor=north] at (12,-1) {$1$} ;

\draw[gray] (-0.5,-0.5) -- (4.5,-0.5);
\draw[gray] (-0.5,-0.5) -- (-0.5,0.5);
\draw[gray] (-0.5,0.5) -- (3.5, 0.5);
\draw[gray] (3.5,0.5) -- (3.5, 1.5);
\draw[gray] (3.5, 1.5) -- (12.5,1.5);
\draw[gray] (12.5,1.5) -- (12.5,0.5);
\draw[gray] (12.5,0.5) --  (4.5,0.5) ;
\draw[gray] (4.5,0.5) --  (4.5,-0.5) ;

  \end{tikzpicture}
\caption{A minimal projective resolution for $M_0 = \A_1  /\!\!/  \E_0$. The horizontal (blue) arrows indicate the maps in the resolution. The circled  (in red) classes are in the kernel. The kernels are redrawn to the right (in red). The duals of the boxed classes (in blue) form a basis of $\Ext_{\A_1}^{*,*}(M_0, \Z/2 )$.}
\label{fig:minM0}
\end{figure}

\begin{figure}[ht]
  \begin{minipage}{.4\textwidth}
    \centering \captionsetup{width=0.9\linewidth}
      \begin{tikzpicture}[scale = .6]
   \fill (0,0) circle (3pt) ;
   \sqtwoR(0,0,black);
      \fill (0,2) circle (3pt) ;
         \sqone(0,2,black);
      \fill (0,3) circle (3pt) ;
   \sqtwoR(0,3,black);
         \fill (0,4) circle (3pt) ;
               \sqone(0,4,black);
                 \sqtwoL(0,4,black);
         \fill (0,5) circle (3pt) ;
               \fill (0,6) circle (3pt) ;
               \sqone(0,6,black) ;
                 \fill (0,7) circle (3pt) ;
                  \sqtwoL(0,7,black);
                          \fill (0,9) circle (3pt) ;
 \end{tikzpicture}
 \caption{The $\A_1$-module $M_1$.}
 \label{fig:M1}
 \end{minipage}
  \begin{minipage}{.55\textwidth}
    \centering
    \centering \captionsetup{width=0.9\linewidth}
  \begin{tikzpicture}[scale=.5]
    \draw[step=1cm,gray,very thin] (0,0) grid (8,8);
    \fill (0.5,0.5) circle (3pt);
        \draw (0.5,0.5) -- (0.5,1.5);
           \fill (0.5,1.5) circle (3pt);
                  \draw (0.5,1.5) -- (0.5,2.5);
           \fill (0.5,2.5) circle (3pt);
                    \draw (0.5,2.5) -- (0.5,3.5);
           \fill (0.5,3.5) circle (3pt);
                        \draw (0.5,3.5) -- (0.5,4.5);
           \fill (0.5,4.5) circle (3pt);
                            \draw (0.5,4.5) -- (0.5,5.5);
           \fill (0.5,5.5) circle (3pt);
                           \draw (0.5,5.5) -- (0.5,6.5);
           \fill (0.5,6.5) circle (3pt);
                                \draw (0.5,6.5) -- (0.5,7.5);
           \fill (0.5,7.5) circle (3pt);
                \draw (0.5,7.5) -- (0.5,8);

                    \fill (4.5,0.5) circle (3pt);
        \draw (4.5,0.5) -- (4.5,1.5);
           \fill (4.5,1.5) circle (3pt);
                  \draw (4.5,1.5) -- (4.5,2.5);
           \fill (4.5,2.5) circle (3pt);
                    \draw (4.5,2.5) -- (4.5,3.5);
           \fill (4.5,3.5) circle (3pt);
                        \draw (4.5,3.5) -- (4.5,4.5);
           \fill (4.5,4.5) circle (3pt);
                            \draw (4.5,4.5) -- (4.5,5.5);
           \fill (4.5,5.5) circle (3pt);
                           \draw (4.5,5.5) -- (4.5,6.5);
           \fill (4.5,6.5) circle (3pt);
                                \draw (4.5,6.5) -- (4.5,7.5);
           \fill (4.5,7.5) circle (3pt);
                \draw (4.5,7.5) -- (4.5,8);
 \end{tikzpicture}
 \caption{The Adams chart for $\Ext_{\A_1}^{s,t}(M_1, \Z/2 )$.}
 \label{fig:M1achart}
 \end{minipage}
\end{figure}

\end{ex}

\begin{ex}
The module $\Z/2 $ has a rather complicated minimal resolution. It is an excellently annoying exercise to work it out. We have illustrated the first two terms of such a resolution in \fullref{fig:minF2}. We will give a different approach in \fullref{ex:periodicF2} to computing the Adams chart for $\Ext_{\A_1}(\Z/2 ,\Z/2 )$ but we include it here in \fullref{fig:F2}.
\begin{figure}[ht]
\center
\begin{tikzpicture}[scale=.5] 
    \fill (0,0) circle (3pt) ;
         \Aone (2, 0);
             \draw[blue,->] (2, 0)  -- (0, 0);
\draw[red] (2,1) circle (5pt);
\draw[red] (2,2) circle (5pt);
\draw[red] (2,3) circle (5pt);
\draw[red] (4,3) circle (5pt);
\draw[red] (4,4) circle (5pt);
\draw[red] (4,5) circle (5pt);
\draw[red] (4,6) circle (5pt);

\msopart(5,1,red);
\Aone(9,2);
\draw[blue,->] (9,2)  .. controls (7, 1.5) .. (5, 2);
\draw[blue,->] (9,3)   .. controls (7, 2.5) ..  (5, 3);
\draw[blue,->] (9,4)   .. controls (8, 3.75) .. (7,4);
\draw[blue,->] (11,5)  .. controls (9, 4.5) ..   (7,5);
\draw[blue,->] (11,6)  .. controls (9, 5.5) ..   (7,6);

\Aone(11,1);
\draw[green,->] (11,1)  .. controls (8, 0.5) ..   (5,1);
\draw[green,->] (11,3)  .. controls (9, 2.5) ..   (7,3);
\draw[green,->] (11,4)  .. controls (9, 4.5) ..   (7,4);
\draw[green,->] (13,6)  .. controls (9.5, 6.5) ..   (7,6);
\draw[red] (9,4) circle (5pt);
\draw[red] (9,5) circle (5pt);
\draw[red] (11,2) circle (5pt);
\draw[red] (11,4) circle (5pt);
\draw[red] (10.8,4)  -- (9.2,4);
\draw[red] (11,6) circle (5pt);
\draw[red] (11,7) circle (5pt);
\draw[red] (11,8) circle (5pt);
\draw[red] (13,4) circle (5pt);
\draw[red] (13,5) circle (5pt);
\draw[red] (13,6) circle (5pt);
\draw[red] (13,7) circle (5pt);
\draw[red] (12.8,6)  -- (11.2,6);

\amme(17,2,red);
       \fill[red] (17,6) circle (3pt);
        \sqtwoCL(17,6,red);
             \sqone(17,6,red);
           \fill[red] (15,8) circle (3pt);
              \fill[red] (15,7) circle (3pt);
                 \fill[red] (15,5) circle (3pt);
                   \sqtwoCR(15,4,red);
                      \sqtwoL(15,5,red);
                    \fill[red] (15,4) circle (3pt);
                    \sqone(15,4,red);
                     \sqone(15,7,red);

\rectangle(2,0,blue);
\rectangle(9,2,blue);
\rectangle(11,1,blue);

\node[anchor=north] at (0,-1) {$-1$} ;
\node[anchor=north] at (2,-1) {$0$} ;
\node[anchor=north] at (10,-1) {$1$} ;
  \end{tikzpicture}
\caption{The beginning of a minimal resolution for the trivial $\A_1$-module $\Z/2$. Circled (red) classes joined by a horizontal line indicate that the sum of the classes are in the kernel.}
\label{fig:minF2}

\bigskip

  \begin{tikzpicture}[scale=.5]
    \draw[step=1cm,gray,very thin] (0,0) grid (12,8);
    \fill (0.5,0.5) circle (3pt);
        \draw (0.5,0.5) -- (0.5,1.5);
           \fill (0.5,1.5) circle (3pt);
                  \draw (0.5,1.5) -- (0.5,2.5);
           \fill (0.5,2.5) circle (3pt);
                    \draw (0.5,2.5) -- (0.5,3.5);
           \fill (0.5,3.5) circle (3pt);
                        \draw (0.5,3.5) -- (0.5,4.5);
           \fill (0.5,4.5) circle (3pt);
                            \draw (0.5,4.5) -- (0.5,5.5);
           \fill (0.5,5.5) circle (3pt);
                           \draw (0.5,5.5) -- (0.5,6.5);
           \fill (0.5,6.5) circle (3pt);
                                \draw (0.5,6.5) -- (0.5,7.5);
           \fill (0.5,7.5) circle (3pt);
                \draw (0.5,7.5) -- (0.5,8);

           \fill (4.5,3.5) circle (3pt);
                        \draw (4.5,3.5) -- (4.5,4.5);
           \fill (4.5,4.5) circle (3pt);
                            \draw (4.5,4.5) -- (4.5,5.5);
           \fill (4.5,5.5) circle (3pt);
                           \draw (4.5,5.5) -- (4.5,6.5);
           \fill (4.5,6.5) circle (3pt);
                                \draw (4.5,6.5) -- (4.5,7.5);
           \fill (4.5,7.5) circle (3pt);
                \draw (4.5,7.5) -- (4.5,8);

           \fill (8.5,4.5) circle (3pt);
                            \draw (8.5,4.5) -- (8.5,5.5);
           \fill (8.5,5.5) circle (3pt);
                           \draw (8.5,5.5) -- (8.5,6.5);
           \fill (8.5,6.5) circle (3pt);
                                \draw (8.5,6.5) -- (8.5,7.5);
           \fill (8.5,7.5) circle (3pt);
                \draw (8.5,7.5) -- (8.5,8);
                
                   \fill (1.5,1.5) circle (3pt);
                        \draw (0.5,0.5) -- (1.5,1.5);
                                  \fill (2.5,2.5) circle (3pt);
                        \draw (1.5,1.5) -- (2.5,2.5);
                
                          \fill (9.5,5.5) circle (3pt);
                        \draw (8.5,4.5) -- (9.5,5.5);
                                  \fill (10.5,6.5) circle (3pt);
                        \draw (9.5,5.5) -- (10.5,6.5);

 \end{tikzpicture}
 \caption{The Adams chart for $\Ext_{\A_1}^{s,t}(\Z/2 , \Z/2 )$.}
\label{fig:F2}
\end{figure}
\end{ex}

\subsection{Change-of-Rings}\label{sec:COR}
Let $\mathcal{B}$ be a subalgebra of $\A$. We defined $\A /\!\!/ \mathcal{B}$ in \fullref{defn:AmmB}. 
\begin{lem}[Shearing Isomorphism]\label{lem:shearing}
Let $\mathcal{B}$ be a sub Hopf-algebra of $\A$. Let $M$ be an $\A$-module. Then there is an isomorphism of $\A$-modules
\[ \A \otimes_{\mathcal{B}} M \cong \mathcal{A}  /\!\!/ \mathcal{B} \otimes_{\Z/2 } M  \]
where the action of $\A$ on $\A \otimes_{\mathcal{B}} M$ is via the left action of $\A$ on itself and the action of $\A$ on $ \mathcal{A}  /\!\!/ \mathcal{B} \otimes_{\Z/2 } M$ is the one described in \fullref{rem:modA}.
\end{lem}
\begin{rem}
If $\mathcal{B}=\Z/2 $, the isomorphism of \fullref{lem:shearing} is induced by the composite
\[\xymatrix{ \A \ar[r]^-{\psi\otimes M} & \A \otimes \A \otimes M \ar[r]^-{\A \otimes f} & \A\otimes M} \]
where $f \colon \A \otimes M \to M$ is the structure map of the $\A$-module $M$. The maps $\psi$ and $\chi$ below are as in \fullref{rem:hopfalgebra}. The inverse is induced by the composite
\[\xymatrix{ \A \ar[r]^-{\psi\otimes M} &\A \otimes \A \otimes M  \ar[rr]^-{\A \otimes \chi \otimes M} & &  \A \otimes \A \otimes M \ar[r]^-{\A \otimes f} & \A\otimes M.}  \]
One verifies that these maps descend to the quotients for more general $\mathcal{B}$.
\end{rem}

From the shearing isomorphism and, from the adjunction
\[\Hom_{\cB}(M, N) \cong \Hom_{\A}(\A\otimes_{\cB} M, N) \]
one can prove that
\[ \Ext_{\A}^{*,*}( \mathcal{A}  /\!\!/ \mathcal{B} \otimes_{\Z/2 } M, N   ) \cong \Ext_{ \mathcal{B} }^{*,*}(  M, N )  \]
for any $\A$-modules $M$ and $N$. Therefore, in the case of extended modules, computations over $\A$ can be reduced to potentially easier computations over smaller sub- Hopf algebras $\cB$. Some common examples are described below. 

\begin{ex}
Many of the modules relevant in the computations of \cite{FH} are of the form $\A   /\!\!/ \A_1  \otimes_{\Z/2 } M_0$ in some range. By the adjunction
\[\Ext_{\A}^{*,*}(\A   /\!\!/ \A_1  \otimes_{\Z/2 } M_0, \Z/2 ) \cong \Ext_{\A_1}^{*,*}(M_0, \Z/2 ), \]
we only need to keep track of the $\A_1$-module structure.
\end{ex}
\begin{rem}
Let $R$ be a graded exterior algebra on $n$ generators over $\Z/2 $ 
\[R = E(x_1, \ldots, x_n) =  \Z/2 [x_1, \ldots, x_n]/(x_1^2, \ldots, x_n^2).\]
where $x_i$ is in degree $t_i$. Then $\Ext_{R}^{*,*}(\Z/2 ,\Z/2 )$ is a polynomial algebra on $n$ generators
\[ \Ext_{R}^{*,*}(\Z/2 ,\Z/2 ) \cong \Z/2 [y_1, \ldots, y_n]\]
for $y_i \in \Ext^{1,t_i}(\Z/2 ,\Z/2 )$. This is an example of a phenomenon called Koszul duality. 
\end{rem}

\begin{ex}\label{ex:corM1}
The module $M_0$ of \fullref{ex:M1} is isomorphic to $\A_1  /\!\!/  \E_0 $, where $\E_0$ is the algebra generated by $Sq^1$. The algebra $\E_0$ is an exterior algebra on one generator in degree $1$, so that
\[\Ext_{\A_1}^{*,*}(M_0,\Z/2 ) \cong \Ext_{\E_0}^{*,*}(\Z/2 ,\Z/2 ) \cong \Z/2 [h_0] \]
where $h_0 \in \Ext^{1,1}_{\E_0}(\Z/2 , \Z/2 )$. The Adams chart for $M_0$ contains one $h_0$-tower starting in degree $(0,0)$.
\end{ex}

\begin{ex}\label{ex:coneeta}
The $\A_1$-module $\A_1  /\!\!/  \E_1$ is the cohomology of $\Sigma^{-2}H^*(\C P^2)$, illustrated in \fullref{fig:E1}.  By the change-of-rings isomorphism,
\[\Ext_{\A_1}^{*,*}(\A_1  /\!\!/  \E_1, \Z/2 ) \cong \Ext_{\E_1}^{*,*}(\Z/2 , \Z/2 ).  \]
Since $\E_1 = E(Q_0, Q_1)$, it follows that $\Ext_{\A_1}^{*,*}(\A_1  /\!\!/  \E_1, \Z/2 )$ is a polynomial algebra on two generators. It is common to call the generator corresponding to $Q_0 = Sq^1$ by $h_0 \in \Ext_{\A_1}^{1,1}(\A_1  /\!\!/  \E_1, \Z/2 )$. The generator corresponding to $Q_1$ is often called $v_1 \in \Ext_{\A_1}^{1,3}(\A_1  /\!\!/  \E_1, \Z/2 )$, so that 
\[ \Ext_{\A_1}^{*,*}(\A_1  /\!\!/  \E_1, \Z/2 ) \cong \Z/2 [h_0, v_1].\]
The Adams chart is depicted in \fullref{fig:coneeta}.
\begin{figure}[ht]
  \centering
  \begin{minipage}{.4\textwidth}
    \centering \captionsetup{width=0.9\linewidth}
      \begin{tikzpicture}[scale=.5]
     \fill (0,0) circle (3pt);
     \sqtwoR(0,0,black);
       \fill (0,2) circle (3pt);
      \end{tikzpicture}
      \caption{$\mathcal{A}_1 /\!\!/ \E_1$}
      \label{fig:E1}
  \end{minipage}
  \begin{minipage}{.55\textwidth}
\centering
\centering \captionsetup{width=0.9\linewidth}
{\tiny
  \begin{tikzpicture}[scale=.5]
    \draw[step=1cm,gray,very thin] (0,0) grid (8,8);
    \fill (0.5,0.5) circle (3pt);
        \draw (0.5,0.5) -- (0.5,1.5);
           \fill (0.5,1.5) circle (3pt);
                  \draw (0.5,1.5) -- (0.5,2.5);
           \fill (0.5,2.5) circle (3pt);
                    \draw (0.5,2.5) -- (0.5,3.5);
           \fill (0.5,3.5) circle (3pt);
                        \draw (0.5,3.5) -- (0.5,4.5);
           \fill (0.5,4.5) circle (3pt);
                            \draw (0.5,4.5) -- (0.5,5.5);
           \fill (0.5,5.5) circle (3pt);
                           \draw (0.5,5.5) -- (0.5,6.5);
           \fill (0.5,6.5) circle (3pt);
                                \draw (0.5,6.5) -- (0.5,7.5);
           \fill (0.5,7.5) circle (3pt);
                \draw (0.5,7.5) -- (0.5,8);

           \fill (2.5,1.5) circle (3pt);
                  \draw (2.5,1.5) -- (2.5,2.5);
           \fill (2.5,2.5) circle (3pt);
                    \draw (2.5,2.5) -- (2.5,3.5);
           \fill (2.5,3.5) circle (3pt);
                        \draw (2.5,3.5) -- (2.5,4.5);
           \fill (2.5,4.5) circle (3pt);
                            \draw (2.5,4.5) -- (2.5,5.5);
           \fill (2.5,5.5) circle (3pt);
                           \draw (2.5,5.5) -- (2.5,6.5);
           \fill (2.5,6.5) circle (3pt);
                                \draw (2.5,6.5) -- (2.5,7.5);
           \fill (2.5,7.5) circle (3pt);
                \draw (2.5,7.5) -- (2.5,8);

           \fill (4.5,2.5) circle (3pt);
                    \draw (4.5,2.5) -- (4.5,3.5);
           \fill (4.5,3.5) circle (3pt);
                        \draw (4.5,3.5) -- (4.5,4.5);
           \fill (4.5,4.5) circle (3pt);
                            \draw (4.5,4.5) -- (4.5,5.5);
           \fill (4.5,5.5) circle (3pt);
                           \draw (4.5,5.5) -- (4.5,6.5);
           \fill (4.5,6.5) circle (3pt);
                                \draw (4.5,6.5) -- (4.5,7.5);
           \fill (4.5,7.5) circle (3pt);
                \draw (4.5,7.5) -- (4.5,8);

           \fill (6.5,3.5) circle (3pt);
                        \draw (6.5,3.5) -- (6.5,4.5);
           \fill (6.5,4.5) circle (3pt);
                            \draw (6.5,4.5) -- (6.5,5.5);
           \fill (6.5,5.5) circle (3pt);
                           \draw (6.5,5.5) -- (6.5,6.5);
           \fill (6.5,6.5) circle (3pt);
                                \draw (6.5,6.5) -- (6.5,7.5);
           \fill (6.5,7.5) circle (3pt);
                \draw (6.5,7.5) -- (6.5,8);
                
                  \node[anchor=west] at (0.5,0.5) {$1$} ;
                          \node[anchor=west] at (0.5,1.5) {$h_0$} ;
                  \node[anchor=west] at (2.5,1.5) {$v_1$} ;
                                   \node[anchor=west] at (4.5,2.5) {$v_1^2$} ;
                                     \node[anchor=west] at (6.5,3.5) {$v_1^3$} ;
 \end{tikzpicture}}
 \caption{The Adams chart for $\Ext_{\A_1}^{s,t}(\A_1  /\!\!/ \E_1 , \Z/2 )$.}
 \label{fig:coneeta}
 \end{minipage}
\end{figure}

\end{ex}

\subsection{Long Exact Sequences}\label{sec:LES}

For some of the computations below we will need to use the long exact sequence induced on $\Ext$ from a short exact sequence of modules.

\begin{prop}\label{prop:LES}
  Let $0 \to M \to N \to P \to 0$ be an exact sequence of $\mathcal{B}$-modules. Then there is a long exact sequence
  \[
  \xymatrix{
\ldots \ar[r] &  \Ext^{s, t}_{\mathcal{B}} (P, \Z/2 ) \ar[r] & \Ext^{s, t}_{\mathcal{B}} (N, \Z/2 ) \ar[r] & \Ext^{s, t}_{\mathcal{B}} (M, \Z/2 ) \ar[dll]_{\delta}\\
  & \Ext^{s+1, t}_{\mathcal{B}} (P, \Z/2 ) \ar[r] & \Ext^{s+1,t}_{\mathcal{B}} (N, \Z/2 ) \ar[r] & \ldots
  }
  \]
\end{prop}

The map $\delta$ can be identified using the description of $\Ext$ in terms of extensions given in \fullref{sec:compext}. Given an extension
\begin{equation}\label{eq:orig} \xymatrix{ 0 \ar[r] & \Sigma^t \Z/2  \ar[r] & P_1 \ar[r] & \ldots  \ar[r] & P_s \ar[r] & M \ar[r] & 0  }  \end{equation}
we let $P_{s+1} = N$ and get an extension of length $s+1$
\[ \xymatrix@-0.5pc{ 0 \ar[r] & \Sigma^t \Z/2  \ar[r] & P_1 \ar[r] & \ldots  \ar[r] & P_s \ar[rr]  \ar[dr] & &  N=P_{s+1}  \ar[r]  & P \ar[r] & 0   \\
 & &   & & &     M  \ar[ur] & }  \]
which corresponds to the boundary of the element of $ \Ext^{s, t}_{\mathcal{B}} (M, \Z/2 ) $ represented by \eqref{eq:orig} in $  \Ext^{s+1, t}_{\mathcal{B}} (P, \Z/2 ) $.
See, e.g. \cite[9.6]{mccleary} for more details.

Computations using \fullref{prop:LES} can be done with the help of an Adams chart. The trick is to draw both $  \Ext^{s, t}_{\mathcal{B}} (P, \Z/2 )$ and $  \Ext^{s, t}_{\mathcal{B}} (M, \Z/2 )$ in the same chart and to treat the boundary map $\delta$ as a differential of slope $(-1,1)$. We illustrate this by an example.

\begin{ex}\label{ex:R0}
We compute $\Ext_{\A_1}^{s,t}(R_0, \Z/2 )$ for $R_0$ as depicted in \fullref{fig:R0}. The module $R_0$ sits in an exact sequence 
\begin{equation}\label{eq:extR0}
0 \to \Sigma \Z/2  \to R_0 \to M_{\infty} \to 0, \end{equation}
so we use the long exact sequence of \fullref{prop:LES} to compute $\Ext_{\A_1}^{s,t}(R_0, \Z/2 )$:
  \[
  \xymatrix{
 \ldots \ar[r] & \Ext^{s, t}_{\mathcal{A}_1} (M_{\infty}, \Z/2 ) \ar[r] & \Ext^{s, t}_{\mathcal{A}_1} (R_0, \Z/2 ) \ar[r] & \Ext^{s, t}_{\mathcal{A}_1} (\Sigma \Z/2 , \Z/2 ) \ar[dll]_-{\delta}^-{h_0}\\
 &  \Ext^{s+1, t}_{\mathcal{A}_1} (M_{\infty}, \Z/2 ) \ar[r] & \Ext^{s+1,t}_{\mathcal{A}_1} (R_0, \Z/2 ) \ar[r] & \ldots
  }
  \]
The boundary is given by multiplication by $h_0$ since \eqref{eq:extR0} is a representative extension for the element $h_0 \cdot 1 \in \Ext_{\A_1}^{1,1}(M_{\infty}, \Z/2)$.
  
In \fullref{fig:R0achart}, the classes of $\Ext_{\A_1}^{s,t}(\Sigma \Z/2 , \Z/2 )$ (blue), which is illustrated in \fullref{fig:F2}, support boundaries (red) to the classes of $\Ext_{\A_1}^{s+1,t}(M_{\infty}, \Z/2 )$ (green).
The circled classes are the elements of $\Ext_{\A_1}^{s,t}(R_0, \Z/2 )$ in this range. The dashed line indicates a multiplication by $h_1$ between a class coming from $\Ext_{\A_1}^{s,t}(\Sigma \Z/2 , \Z/2 )$ and a class coming from $\Ext_{\A_1}^{s+1,t}(M_{\infty}, \Z/2 )$, which we have not justified. One way to do this is to compute a minimal resolution for $R_0$ and use the fact that multiplication by $h_1$ corresponds to the extension depicted in \fullref{fig:h1}.

\begin{figure}[ht]
  \centering
\begin{minipage}{.4\textwidth}
\centering \captionsetup{width=0.9\linewidth}
 \begin{tikzpicture}[scale = .5]
 \fill (-4,1)  circle (3pt)  ;
   \draw[blue, ->] (-4,1) -- (0,1);
      \draw[blue, ->] (0,0) -- (4,0);
 
    \fill (0,0) circle (3pt)  ;
        \sqone(0,0,black); 
         \sqtwoL(0,0,black); 
            \fill (0,1) circle (3pt) ;
                   \fill (0,2) circle (3pt)  ;
                      \sqone(0,2,black); 
               \fill (0,3) circle (3pt) ;
               \sqtwoL(0,3, black);
               \fill (0,4) circle (3pt)   ;
               \sqone(0,4,black); 
                    \sqtwoR(0,4, black);
               \fill (0,5) circle (3pt) ;
                 \fill (0,6) circle (3pt) ;
                   \sqone(0,6,black); 
    \fill (4,0) circle (3pt)  ;
         \sqtwoL(4,0,black); 
                   \fill (4,2) circle (3pt)  ;
                      \sqone(4,2,black); 
               \fill (4,3) circle (3pt) ;
               \sqtwoL(4,3, black);
               \fill (4,4) circle (3pt)   ;
               \sqone(4,4,black); 
                    \sqtwoR(4,4, black);
               \fill (4,5) circle (3pt) ;
                 \fill (4,6) circle (3pt) ;
                   \sqone(4,6,black); 
                       \node[anchor=north] at (-4, 0.5) {$\Sigma \Z/2 $}; 
                      \node[anchor=north] at (0, -0.5) {$R_0$}; 
                         \node[anchor=north] at (4, -0.5) {$M_{\infty}$}; 
                                                         \end{tikzpicture}
                                    \caption{An extension exhibiting an $\A_1$-module we call $R_0$.}
                                    \label{fig:R0}
                             
\end{minipage}                                    
\begin{minipage}{.55\textwidth}
\centering
\centering \captionsetup{width=0.9\linewidth}
  \begin{tikzpicture}[scale=.5]
    \draw[step=1cm,gray,very thin] (0,0) grid (8,8);
    %the black parts, i.e. the homotopy groups
    \fill[green] (0.5,0.5) circle (3pt);
    \fill[green] (4.5,0.5) circle (3pt);
    \fill[green] (4.5,1.5) circle (3pt);
    \draw[green] (4.5,0.5) -- (4.5,1.5);
    \fill[blue] (2.5,1.5) circle (3pt);
    \fill[blue] (3.5,2.5) circle (3pt);
    \draw[blue] (2.5,1.5) -- (3.5,2.5);
 \draw[blue] (1.5,0.5) -- (2.5,1.5);
    %all of the colored in parts
    \foreach \y in {1,2,3,4,5,6,7}
             {\fill[green] (0.5,\y+.5) circle (3pt);}
    \foreach \y in {1,2,3,4,5,6,7}
             {\draw[green] (0.5,\y-.5)--(.5,\y+.5);}
     \foreach \y in {0,1,2,3,4,5,6,7}
             {\fill[blue] (1.5,\y+.5) circle (3pt);}
     \foreach \y in {1,2,3,4,5,6,7}
              {\draw[blue] (1.5,\y-.5)--(1.5,\y+.5);}
     \foreach \y in {1, 2, 3, 4, 5, 6}
              {\draw[green] (4.5,\y+.5)--(4.5,\y+1.5);}
     \foreach \y in {2,3,4,5,6,7}
              {\fill[green] (4.5,\y+.5) circle (3pt);}
     \foreach \y in {3,4,5,6,7}
              {\fill[blue] (5.5,\y+.5) circle (3pt);}
     \foreach \y in {3,4,5,6,6}
              {\draw[blue] (5.5,\y+.5) -- (5.5,\y+1.5);}
     \foreach \y in {0,1,2,3,4,5,6}
              {\draw[red] (1.5,\y+.5) -- (0.5,\y+1.5);}
     \foreach \y in {3,4,5,6}
              {\draw[red] (5.5,\y+.5) -- (4.5,\y+1.5);} 
              \draw[green] (0.5,7.5) -- (0.5,8); 
                 \draw[blue] (1.5,7.5) -- (1.5,8);
                    \draw[green] (4.5,7.5) -- (4.5,8); 
                       \draw[blue] (5.5,7.5) -- (5.5,8); 
                        \draw[red] (1.5,7.5) -- (1,8);
                          \draw[red] (5.5,7.5) -- (5,8);
                          
\draw[black] (0.5,0.5) circle (5pt);     
\draw[black] (4.5,0.5) circle (5pt);
\draw[black] (4.5,1.5) circle (5pt);    
\draw[black] (4.5,2.5) circle (5pt);    
\draw[black] (4.5,3.5) circle (5pt);    
\draw[black] (2.5,1.5) circle (5pt);        
\draw[black] (3.5,2.5) circle (5pt);    

\draw[dashed] (3.5,2.5) -- (4.5,3.5);          
  \end{tikzpicture}
  \caption{The computation of the Adams chart for $\Ext_{\A_1}^{s,t}(R_0, \Z/2 )$ using the exact sequence of \fullref{fig:R0}. }
  \label{fig:R0achart}
  \end{minipage}
\end{figure}
\end{ex}

\begin{ex}
Consider the $\A_1$-module depicted in \fullref{fig:R1}. Using \fullref{prop:LES}, we get the Adams chart depicted in \fullref{fig:R1achart}.
\begin{figure}[ht]
  \centering
  \begin{minipage}{.4\textwidth}
    \centering
\centering \captionsetup{width=0.9\linewidth}
 \begin{tikzpicture}[scale = .4]
 \fill (0,0)  circle (3pt)  ;
     \fill (0,2)  circle (3pt)  ;
      \sqtwoL(0,0,black);

\fill (0,1)  circle (3pt)  ;
\sqtwoR(0,1,black);
\sqone(0,1,black);
\fill (0,3)  circle (3pt)  ;
\sqone(0,3,black);
\fill (0,4) circle (3pt) ;
\fill (0,5) circle (3pt) ;
\sqone(0,5,black);
\fill (0,6) circle (3pt);
\sqtwoR(0,4,black);
\fill (0,7) circle (3pt);
\sqtwoL(0,5,black);
\sqone(0,7,black);

\fill (4,1)  circle (3pt)  ;
\sqtwoR(4,1,black);
\fill (4,3)  circle (3pt)  ;
\sqone(4,3,black);
\fill (4,4) circle (3pt) ;
\fill (4,5) circle (3pt) ;
\sqone(4,5,black);
\fill (4,6) circle (3pt);
\sqtwoR(4,4,black);
\fill (4,7) circle (3pt);
\sqtwoL(4,5,black);
\sqone(4,7,black);

 \fill (-4,0)  circle (3pt)  ;
  \sqtwoL(-4,0,black);
   \fill (-4,2)  circle (3pt)  ;
   \draw[blue, ->] (-4,0) -- (0,0);
      \draw[blue, ->] (-4,2) -- (0,2);

                       \node[anchor=north] at (-4, -0.5) {$\A_1  /\!\!/ \E_1$}; 
                      \node[anchor=north] at (0, -0.5) {$R_1$}; 
                         \node[anchor=north] at (4, 0.5) {$\Sigma M_{\infty}$}; 
                                                         \end{tikzpicture}
                                    \caption{An exact sequence of $\A_1$-modules depicting $R_1$.}
                                    \label{fig:R1}

\end{minipage}
  \begin{minipage}{.55\textwidth}
\centering
\centering \captionsetup{width=0.9\linewidth}
  \begin{tikzpicture}[scale=.5]
\draw[step=1cm,gray,very thin] (0,0) grid (7,8);
\foreach \y in {0,1,2,3,4,5,6,7}
{\fill[blue] (0.5,\y+.5) circle (3pt);}
\foreach \y in {1,2,3,4,5,6,7}
{\draw[blue] (0.5,\y-.5)--(.5,\y+.5);}
\draw[blue] (0.5,7.5) -- (0.5,8); 

\foreach \y in {0,1,2,3,4,5,6,7}
{\fill[green] (1.5,\y+.5) circle (3pt);}
\foreach \y in {1,2,3,4,5,6,7}
{\draw[green] (1.5,\y-.5)--(1.5,\y+.5);}
\draw[green] (1.5,7.5) -- (1.5,8); 

\foreach \y in {1,2,3,4,5,6,7}
{\fill[blue] (2.5,\y+.5) circle (3pt);}
\foreach \y in {2,3,4,5,6,7}
{\draw[blue] (2.5,\y-.5)--(2.5,\y+.5);}
\draw[blue] (2.5,7.5) -- (2.5,8);

\foreach \y in {2,3,4,5,6,7}
{\fill[blue] (4.5,\y+.5) circle (3pt);}
\foreach \y in {3,4,5,6,7}
{\draw[blue] (4.5,\y-.5)--(4.5,\y+.5);}
\draw[blue] (4.5,7.5) -- (4.5,8);

\foreach \y in {0,1,2,3,4,5,6,7}
{\fill[green] (5.5,\y+.5) circle (3pt);}
\foreach \y in {1,2,3,4,5,6,7}
{\draw[green] (5.5,\y-.5)--(5.5,\y+.5);}
\draw[green] (5.5,7.5) -- (5.5,8); 

\foreach \y in {3,4,5,6,7}
{\fill[blue] (6.5,\y+.5) circle (3pt);}
\foreach \y in {4,5,6,7}
{\draw[blue] (6.5,\y-.5)--(6.5,\y+.5);}
\draw[blue] (6.5,7.5) -- (6.5,8);

\foreach \y in {1,2,3,4,5,6}
{\draw[red] (2.5,\y+.5) -- (1.5,\y+1.5);}

\foreach \y in {3,4,5,6}
{\draw[red] (6.5,\y+.5) -- (5.5,\y+1.5);}

\draw[red] (2.5,7.5) -- (2,8);
\draw[red] (6.5,7.5) -- (6,8);
              
\draw[black] (0.5,0.5) circle (5pt);  
\draw[black] (0.5,1.5) circle (5pt);  
\draw[black] (0.5,2.5) circle (5pt);  
\draw[black] (0.5,3.5) circle (5pt);  
\draw[black] (0.5,4.5) circle (5pt);  
\draw[black] (0.5,5.5) circle (5pt);
\draw[black] (0.5,6.5) circle (5pt);  
\draw[black] (0.5,7.5) circle (5pt);

\draw[black] (1.5,0.5) circle (5pt);  
\draw[black] (1.5,1.5) circle (5pt);

\draw[black] (5.5,0.5) circle (5pt);  
\draw[black] (5.5,1.5) circle (5pt);  
\draw[black] (5.5,2.5) circle (5pt);  
\draw[black] (5.5,3.5) circle (5pt);       

\draw[black] (4.5,2.5) circle (5pt);  
\draw[black] (4.5,3.5) circle (5pt);  
\draw[black] (4.5,4.5) circle (5pt);  
\draw[black] (4.5,5.5) circle (5pt);
\draw[black] (4.5,6.5) circle (5pt);  
\draw[black] (4.5,7.5) circle (5pt);

  \end{tikzpicture}
\caption{The computation of the Adams chart of $\Ext_{\A_1}^{s,t}(R_1, \Z/2 )$ using the exact sequence of \fullref{fig:R1}.}
\label{fig:R1achart}
\end{minipage}
  \end{figure}
\end{ex}

\begin{rem}\label{rem:sstrick}
We present one last trick which is a variation on \fullref{prop:LES}. It uses the fact that, although the module $\A_1 /\!\!/ \E_0$ is not projective, it has a nice periodic resolution as an $\A_1$-module. Given a module $M$, suppose that  there is an exact complex
\begin{equation}
\xymatrix{0  & M \ar[l] & P_0 \ar[l]_-{f_0}  & P_1 \ar[l]_-{f_1}  & P_2 \ar[l]_-{f_2}  &  \ldots \ar[l] }
\end{equation}
where the $P_s$ are direct sums of suspensions of copies of $\A_1$ and $\A_1 /\!\!/ \E_0$ and with the property that
\[f_s(P_s) \subseteq I(\cB)P_{s-1},\]
so that $P_{\bullet} \to M$ is a ``minimal resolution'', but not by projective modules. We call this a ``modified'' minimal resolution. For each summand $\Sigma^{t} \A_1$ in $P_s$, there will be a generator of $\Z/2  \in \Ext^{s,t}_{\A_1}(M,\Z/2 )$ and for each summand
 $\Sigma^{t} \A_1 /\!\!/ \E_0$ in $P_s$, there will be an $h_0$-tower whose generator is in $\Ext^{s,t}_{\A_1}(M,\Z/2 )$. 

The proof of this fact uses the collapsing of the spectral sequence of a double complex built from minimal resolutions.
 \end{rem}

 \begin{ex}\label{ex:periodicF2}
 We give a modified minimal resolution for $\A_1$ which is periodic in \fullref{fig:periodicF2}. More precisely, the figure depicts the top row of \eqref{eq:per}. The periodic resolution is obtained by splicing copies of this complex together and is the bottom row of \eqref{eq:per}.
 \begin{equation}\label{eq:per}
 \xymatrix@-1.2pc{  0 & \Z/2  \ar[l]   \ar@{=}[d]  & \A_1 \ar@{=}[d]  \ar[l] & \Sigma^2 \A_1 \oplus  \Sigma \A_1 /\!\!/ \E_0  \ar@{=}[d] \ar[l] & \Sigma^{4} \A_1 \ar@{=}[d] \ar[l] & \Sigma^7 \A_1 /\!\!/ \E_0 \ar@{=}[d] \ar[l] & \Sigma^{12} \Z/2  \ar[l] & 0 \ar[l] \\
 0 & \Z/2  \ar[l]  & P_0  \ar[l]  & P_1 \ar[l] & P_2 \ar[l] &P_3  \ar[l] & \Sigma^{12} P_0 \ar[l] \ar[u] &  \Sigma^{12} P_1 \ar[l] &  \ldots \ar[l]   }
 \end{equation}

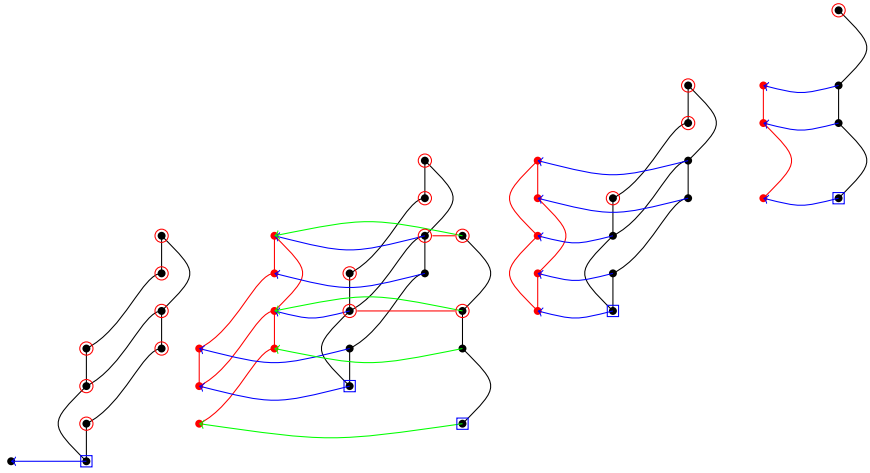
\begin{figure}[ht]
\center
{\tiny 
\begin{tikzpicture}[scale=.5] 
    \fill (0,0) circle (3pt) ;
         \Aone (2, 0);
             \draw[blue,->] (2, 0)  -- (0, 0);
\draw[red] (2,1) circle (5pt);
\draw[red] (2,2) circle (5pt);
\draw[red] (2,3) circle (5pt);
\draw[red] (4,3) circle (5pt);
\draw[red] (4,4) circle (5pt);
\draw[red] (4,5) circle (5pt);
\draw[red] (4,6) circle (5pt);

\msopart(5,1,red);
\Aone(9,2);
\draw[blue,->] (9,2)  .. controls (7, 1.5) .. (5, 2);
\draw[blue,->] (9,3)   .. controls (7, 2.5) ..  (5, 3);
\draw[blue,->] (9,4)   .. controls (8, 3.75) .. (7,4);
\draw[blue,->] (11,5)  .. controls (9, 4.5) ..   (7,5);
\draw[blue,->] (11,6)  .. controls (9, 5.5) ..   (7,6);

\amme(12,1,black)
\draw[green,->] (12,1)  .. controls (8.5, 0.5) ..   (5,1);
\draw[green,->] (12,3)  .. controls (9.5, 2.5) ..   (7,3);
\draw[green,->] (12,4)  .. controls (9.5, 4.5) ..   (7,4);
\draw[green,->] (12,6)  .. controls (9.5, 6.5) ..   (7,6);

\draw[red] (9,4) circle (5pt);
\draw[red] (9,5) circle (5pt);
\draw[red] (12,4) circle (5pt);
\draw[red] (11.8,4)  -- (9.2,4);

\draw[red] (11,6) circle (5pt);
\draw[red] (11,7) circle (5pt);
\draw[red] (11,8) circle (5pt);
\draw[red] (12,6) circle (5pt);
\draw[red] (11.8,6)  -- (11.2,6);

\jokercolor(14,4, red);
\Aone(16,4);
\draw[blue,->] (16,4)  .. controls (15, 3.75) .. (14, 4);
\draw[blue,->] (16,5)  .. controls (15, 4.75) .. (14, 5);
\draw[blue,->] (16,6)  .. controls (15, 5.75) .. (14, 6);
\draw[blue,->] (18,7)  .. controls (16, 6.5) .. (14, 7);
\draw[blue,->] (18,8)  .. controls (16, 7.5) .. (14, 8);

\draw[red] (16,7) circle (5pt);
\draw[red] (18,9) circle (5pt);
\draw[red] (18,10) circle (5pt);

\questionupsidedon(20,7,red);
\amme(22,7,black);
\draw[blue,->] (22,7)  .. controls (21, 6.75) .. (20, 7);
\draw[blue,->] (22,9)  .. controls (21, 8.75) .. (20, 9);
\draw[blue,->] (22,10)  .. controls (21, 9.75) .. (20, 10);
\draw[red] (22,12) circle (5pt);

\rectangle(2,0,blue);
\rectangle(9,2,blue);
\rectangle(12,1,blue);
\rectangle(16,4,blue);
\rectangle(22,7,blue);

  \end{tikzpicture}} 
\caption{A modified minimal resolution for $\Z/2 $. The dual of the boxed classes correspond to a copy of $\Z/2 $ if they generate an $\A_1$, or an $h_0$-tower if they generate an $\A_1  /\!\!/  \E_0$.}
\label{fig:periodicF2}
\end{figure}
\end{ex}

\begin{ex}
Consider the module $R_2$ depicted in \fullref{fig:R2JQ}. Using the resolution constructed in \fullref{ex:periodicF2}, we have a modified minimal resolution
\[ \xymatrix{ 0  & R_2 \ar[l] & \Sigma^{-1}P_1 \ar[l] &  \Sigma^{-1}P_2 \ar[l]  & \ldots \ar[l]. }\]
So the Adams chart for $\Ext^{s,t}_{\A_1}(R_2, \Z/2 )$ is a truncated version of that for $\Ext^{s,t}_{\A_1}(\Z/2 , \Z/2 )$ and is given in \fullref{fig:R2JQachart}. Similarly, the $\A_1$-modules $J$, called the \emph{joker}, and $Q$ called the \emph{``upside down'' question mark complex} also have Adams charts which are truncated versions of that for  $\A_1$. These are depicted in \fullref{fig:R2JQachart}.

\begin{figure}[ht]
\center
{\tiny  \begin{tikzpicture}[scale=.5]
\msopart(-10,0,black)
  \node[anchor=north] at (-10,-0.5) {$R_2$};  
  
\joker(-4,0);
\node[anchor=north] at (-4,-0.5) {$J$};

    \fill (0, 0) circle (3pt);
    \fill (0, 3) circle (3pt);
    \fill (0, 2) circle (3pt);
    \sqone (0, 2, black);
    \sqtwoL (0, 0, black); 
  \node[anchor=north] at (0,-0.5) {$Q$};  
  
  \end{tikzpicture}}
  \caption{An $\A_1$-module we call $R_2$ (left), the joker $J$ (center) and the ``upside down'' question mark complex $Q$ (right).}
  \label{fig:R2JQ}
  \end{figure}
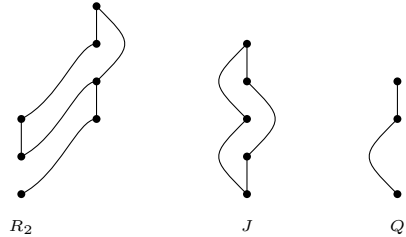

 \begin{figure}[ht]
\center
{\tiny
  \begin{tikzpicture}[scale=.5]
    \draw[step=1cm,gray,very thin] (0,1) grid (12,8);
           \fill (0.5,1.5) circle (3pt);
                  \draw (0.5,1.5) -- (0.5,2.5);
           \fill (0.5,2.5) circle (3pt);
                    \draw (0.5,2.5) -- (0.5,3.5);
           \fill (0.5,3.5) circle (3pt);
                        \draw (0.5,3.5) -- (0.5,4.5);
           \fill (0.5,4.5) circle (3pt);
                            \draw (0.5,4.5) -- (0.5,5.5);
           \fill (0.5,5.5) circle (3pt);
                           \draw (0.5,5.5) -- (0.5,6.5);
           \fill (0.5,6.5) circle (3pt);
                                \draw (0.5,6.5) -- (0.5,7.5);
           \fill (0.5,7.5) circle (3pt);
                \draw (0.5,7.5) -- (0.5,8);

           \fill (4.5,3.5) circle (3pt);
                        \draw (4.5,3.5) -- (4.5,4.5);
           \fill (4.5,4.5) circle (3pt);
                            \draw (4.5,4.5) -- (4.5,5.5);
           \fill (4.5,5.5) circle (3pt);
                           \draw (4.5,5.5) -- (4.5,6.5);
           \fill (4.5,6.5) circle (3pt);
                                \draw (4.5,6.5) -- (4.5,7.5);
           \fill (4.5,7.5) circle (3pt);
                \draw (4.5,7.5) -- (4.5,8);

           \fill (8.5,4.5) circle (3pt);
                            \draw (8.5,4.5) -- (8.5,5.5);
           \fill (8.5,5.5) circle (3pt);
                           \draw (8.5,5.5) -- (8.5,6.5);
           \fill (8.5,6.5) circle (3pt);
                                \draw (8.5,6.5) -- (8.5,7.5);
           \fill (8.5,7.5) circle (3pt);
                \draw (8.5,7.5) -- (8.5,8);
                
                   \fill (1.5,1.5) circle (3pt);
                    
                                  \fill (2.5,2.5) circle (3pt);
                        \draw (1.5,1.5) -- (2.5,2.5);
                
                          \fill (9.5,5.5) circle (3pt);
                        \draw (8.5,4.5) -- (9.5,5.5);
                                  \fill (10.5,6.5) circle (3pt);
                        \draw (9.5,5.5) -- (10.5,6.5);

 \end{tikzpicture}
 
 \ \\

  \begin{tikzpicture}[scale=.5]
    \draw[step=1cm,gray,very thin] (2,2) grid (14,8);

           \fill (4.5,3.5) circle (3pt);
                        \draw (4.5,3.5) -- (4.5,4.5);
           \fill (4.5,4.5) circle (3pt);
                            \draw (4.5,4.5) -- (4.5,5.5);
           \fill (4.5,5.5) circle (3pt);
                           \draw (4.5,5.5) -- (4.5,6.5);
           \fill (4.5,6.5) circle (3pt);
                                \draw (4.5,6.5) -- (4.5,7.5);
           \fill (4.5,7.5) circle (3pt);
                \draw (4.5,7.5) -- (4.5,8);

           \fill (8.5,4.5) circle (3pt);
                            \draw (8.5,4.5) -- (8.5,5.5);
           \fill (8.5,5.5) circle (3pt);
                           \draw (8.5,5.5) -- (8.5,6.5);
           \fill (8.5,6.5) circle (3pt);
                                \draw (8.5,6.5) -- (8.5,7.5);
           \fill (8.5,7.5) circle (3pt);
                \draw (8.5,7.5) -- (8.5,8);
                
                                  \fill (2.5,2.5) circle (3pt);

                          \fill (9.5,5.5) circle (3pt);
                        \draw (8.5,4.5) -- (9.5,5.5);
                                  \fill (10.5,6.5) circle (3pt);
                        \draw (9.5,5.5) -- (10.5,6.5);
                        
                             \fill (12.5,7.5) circle (3pt);
    \draw (12.5,7.5) -- (12.5,8);

 \end{tikzpicture}

\ \\
  \begin{tikzpicture}[scale=.5]
    \draw[step=1cm,gray,very thin] (4,3) grid (16,8);

           \fill (4.5,3.5) circle (3pt);
                        \draw (4.5,3.5) -- (4.5,4.5);
           \fill (4.5,4.5) circle (3pt);
                            \draw (4.5,4.5) -- (4.5,5.5);
           \fill (4.5,5.5) circle (3pt);
                           \draw (4.5,5.5) -- (4.5,6.5);
           \fill (4.5,6.5) circle (3pt);
                                \draw (4.5,6.5) -- (4.5,7.5);
           \fill (4.5,7.5) circle (3pt);
                \draw (4.5,7.5) -- (4.5,8);

           \fill (8.5,4.5) circle (3pt);
                            \draw (8.5,4.5) -- (8.5,5.5);
           \fill (8.5,5.5) circle (3pt);
                           \draw (8.5,5.5) -- (8.5,6.5);
           \fill (8.5,6.5) circle (3pt);
                                \draw (8.5,6.5) -- (8.5,7.5);
           \fill (8.5,7.5) circle (3pt);
                \draw (8.5,7.5) -- (8.5,8);

                          \fill (9.5,5.5) circle (3pt);
                        \draw (8.5,4.5) -- (9.5,5.5);
                                  \fill (10.5,6.5) circle (3pt);
                        \draw (9.5,5.5) -- (10.5,6.5);

     \fill (12.5,7.5) circle (3pt);
    \draw (12.5,7.5) -- (12.5,8);

 \end{tikzpicture}}
 \caption{The Adams chart for $\Ext_{\A_1}^{s,t}(R_2, \Z/2 )$ (top), $\Ext_{\A_1}^{s,t}(J, \Z/2 )$ (center) and the Adams chart for $\Ext_{\A_1}^{s,t}(Q, \Z/2 )$ (bottom).}
\label{fig:R2JQachart}
\end{figure}
\end{ex}

\subsection{The Adams spectral sequence}\label{sec:assconstruction}
We turn to the construction of the spectral sequence. In this section, we make the following assumption:

\begin{assumption}\label{ass:assX}
Let $X$ be the suspension spectrum of a CW-complex that has finitely many cells in each dimension. 
\end{assumption}

For example, the Thom spectra we are considering have this property since Grassmanians have cell structures with finitely many $n$-cells for each $n$. Some of this can be done in more generality, but all of our examples will have models of this form so we limit ourselves to this case. A friendly reference to spectral sequences is Hatcher's online notes \cite{hatcherss}. Other great references are McCleary \cite{mccleary}, Boardman \cite{boardman} and Miller \cite{millerrelations}.

\begin{defn}
The \emph{Hurewicz homomorphism}
\begin{align*} 
h \colon \pi_tX =  [S^t, X]  \to \Hom_{\A} (H^*(X),  H^*(S^t))\cong \Hom_{\A} (H^*(X),  \Sigma^t \Z/2 )
\end{align*}
is defined by sending a map $f \colon S^t \to X$ to the induced map on cohomology, $f^* \colon H^*(X) \to H^*(S^t)$.
\end{defn}

If $h$ were an isomorphism, computing the homotopy groups of $X$ would be as easy as understanding its cohomology. In certain cases, this does happen. 

\begin{defn}\label{defn:genEM}
A spectrum $Z$ is a generalized Eilenberg--MacLane spectrum of finite type if 
\[Z  \simeq HV \simeq  \bigvee_{i \in I} \Sigma^iH\Z/2   \] 
where $V$ is a graded $\Z/2 $ vector space which is finite in each degree.
\end{defn}

The finiteness assumption in \fullref{defn:genEM} gives an isomorphism
\[\bigvee_{i \in I} \Sigma^iH\Z/2 \simeq \prod_{i \in I} \Sigma^iH\Z/2 .\]
See \eqref{eq:prodcoprod}.

\begin{ex}\label{rem:freeup}
Let $X$ be a spectrum that satisfies \fullref{ass:assX}. There is an isomorphism
\[ H^*(\HF \smsh X ) \cong \A \ox_{\Z/2 } H^*(X)  \]
and a class $1 \ox x \in H^{|x|}(\HF \smsh X )$ corresponds to a map
$\HF \smsh X \ra \Sigma^{|x|}\HF$. By \fullref{ass:assX}, the cohomology of $X$ is finite in each degree, so
\[ \prod_{x \in H^*(X) } \Sigma^{|x|} \HF \simeq \bigvee_{x \in H^*(X) } \Sigma^{|x|} \HF \]
and the product of these maps is a weak equivalence:
\[ \xymatrix{\HF \smsh X  \ar[r]  &  \bigvee_{x \in H^*(X) } \Sigma^{|x|} \HF }\]
So any spectrum of the form $\HF \smsh X$ for $X$ satisfying \fullref{ass:assX} is a generalized Eilenberg--MacLane spectrum of finite type.
\end{ex}

If $Z$ is a generalized Eilenberg--MacLane spectrum of finite type, then the Hurewicz homomorphism is an isomorphism. So, the idea is to resolve $X$ by generalized Eilenberg--MacLane spectra. 
\begin{defn}
  Let $X$ be a spectrum that satisfies \fullref{ass:assX}. An \textbf{Adams resolution} is a sequence of spectra
  \begin{equation}\label{eq:adamsres}
  \xymatrix{
  X = X_0 \ar[d]^-{j_0} &  X_1 \ar[d]^-{j_1}  \ar[l]_-{i_0} &  X_2 \ar[d]^-{j_2} \ar[l]_-{i_1}  &    X_3 \ar[l]_-{i_2}  \ar[d]^-{j_3} & \ldots \ar[l] &  \\ 
K_0  \ar@{.>}[ur]_-{\delta_0} &   K_1  \ar@{.>}[ur]_-{\delta_1}  &  K_2  \ar@{.>}[ur]_-{\delta_2} &    K_3 \ar@{.>}[ur]_-{\delta_3}    &     }   \end{equation}
where 
\[   \xymatrix{ X_{s+1} \ar[r]^-{i_s} & X_s \ar[r]^-{j_s} & K_s \ar[r]^-{\delta_s}  & \Sigma X_{s+1}}  \]
are cofiber sequences (i.e. exact triangles) and such that
    \begin{enumerate}[(a)]
  \item $K_i \simeq \bigvee \Sigma^j \HF$, and 
  \item $H^\ast (K_i) \to H^\ast (X_i)$ is surjective.  
  \end{enumerate}
\end{defn}

\begin{rem}
From an Adams resolution, we obtain a sequence
\[ \xymatrix{X = X_0 \ar[r] &  K_0  \ar[r] &  \Sigma K_1  \ar[r] &  \Sigma^2 K_2 \ar[r] & \ldots} \]
where $ \Sigma^s K_s \to  \Sigma^{s+1} K_{s+1}  $ is the composite $j_{s+1} \circ \delta_s$.  Further, the resolution is constructed so that $H^*(\Sigma^{\bullet} K_{\bullet}) \to H^*(X)$ is a projective resolution of $H^*(X)$ as an $\A$-module.
\end{rem}

\begin{rem}
Let $\bHF $ be defined by the fiber sequence $\bHF  \to S \to \HF$. From \fullref{rem:freeup}, it follows that
\[  \xymatrix{
  X  \ar[d]^-{j_0} &   \bHF \smsh X  \ar[d]^-{j_1}  \ar[l]_-{i_0} &  \bHF^{\smsh 2} \smsh X  \ar[d]^-{j_2} \ar[l]_-{i_1}   & \ldots \ar[l]_-{i_2}  &  \\ 
 \HF \smsh X  \ar@{.>}[ur]_-{\delta_0} &   \HF \smsh \bHF \smsh X  \ar@{.>}[ur]_-{\delta_1}  & \HF \smsh  \bHF^{\smsh 2} \smsh X  \ar@{.>}[ur]_-{\delta_2}     &     }  \]
is an Adams resolution. So Adams resolutions always exist.
\end{rem}

\begin{defn}
Let 
\[F^s = \text{im}(\pi_*X_s \ra \pi_* X ).\]
Then $\alpha \in \pi_*X$ has \emph{Adams filtration} $s$ if $\alpha \in F^s\backslash F^{s+1}$. \end{defn}

The Adams filtration of an element is independent of the choice of Adams resolution. 

\begin{lem}
An element $f\in \pi_t X$ has Adams filtration $\geq s$ if and only if $f$ factors as 
\[f : S^t =U_s\ra  U_{s-1} \ra U_{s-2} \ra \ldots \ra U_1 \ra U_0=X \]
where the maps $U_{i} \ra U_{i-1}$ induce the zero maps on mod-$2$ cohomology.
\end{lem}

\begin{ex}
An element of $\pi_*X$ has Adams filtration $0$ if and only if its image under the Hurewicz homomorphism is non-zero.
The image of $\pi_tX_1 \ra \pi_t X$ is the kernel of the map $j_0$. But $j_0^*$ is surjective on cohomology, so 
\[ \Hom_{\A}(H^*(X), \Sigma^t \Z/2) \to \Hom_{\A}(H^*(K_0), \Sigma^t \Z/2) \]
is injective. In particular, $i_0^*$ must be zero and the image of $i_0$ consists of elements of filtration $s\geq 1$.

Examples of elements of Adams filtration one are the Hopf maps
\begin{align*}
\eta \colon S^3 &\ra S^2,  &   \nu \colon  S^7 &\ra S^4,  & \sigma \colon  S^{15} &\ra S^8. 
\end{align*}
\end{ex}

We now turn to the construction of the Adams spectral sequence. Fix an Adams resolution of $X$ as in \eqref{eq:adamsres}.
Applying $\pi_*(-)$, we get an unravelled exact couple
\[\xymatrix{ \pi_*X \ar@{=}[r] &\pi_*X_0 \ar[d]^{j_0} & \pi_*X_1 \ar[l]_{i_0} \ar[d]^{j_1} & \pi_*X_2  \ar[l]_{i_1}  \ar[d]^{j_2} & \ar[l]_{i_2} \pi_*X_3 \ar[d]^{j_3} & \ar[l] \ldots  \\
& \pi_* K_0  \ar@{.>}[ur]_{\delta_0} & \pi_*K_1  \ar@{.>}[ur]_{\delta_1} & \pi_*K_2  \ar@{.>}[ur]_{\delta_2} & \pi_*K_3  \ar@{.>}[ur]  &  }  \]
from which we obtain a spectral sequence. More precisely, we let
\begin{enumerate}[(a)]
\item $E_1^{s, t} = \pi_{t-s}K_s \cong \pi_t \Sigma^s K_s$, and
\item $d_1 \colon E_1^{s, t}  \to E_1^{s+1, t} $ be given by $d_1 = \Sigma j_{s+1} \circ \delta_s $.
\end{enumerate}
In general, $E_r^{*,*} = \ker(d_{r-1})/\im(d_{r-1})$ and $d_r \colon E_r^{s, t} \to E_r^{s+r, t +r-1}$ is given by $d_r(x) = j(y)$ for any $y$ such that $i^{r-1}(y) = \delta(x)$. Here, $i^{n} =i \circ \ldots \circ i$ iterated $n$-times and we have left out indices and suspensions.

\begin{prop}
Let $X$ satisfy \fullref{ass:assX}. There is an isomorphism
\[E_2^{s,t} \cong \Ext_{\A}^{s,t}(H^*(X), \Z/2 ).\]
\end{prop}
\begin{proof}
In degree $t$, the $E_2$ term is the cohomology of 
\begin{equation}\label{eq:respiK} \xymatrix{0 \ar[r] &  \pi_tK_0 \ar[r] &  \pi_t\Sigma K_1 \ar[r] &  \pi_t\Sigma^2 K_2 \ar[r] &  \ldots }  \end{equation}
However, the $K_s$ are generalized Eilenberg MacLane spectra, so
\[ \pi_{t}\Sigma^s K_s \cong \Hom_{\A}(H^*(\Sigma^s K_s), \Z/2 ).\]
The Adams resolutions are built so that $H^*(\Sigma^{\bullet} K_{\bullet}) \to H^*(X)$ is a projective resolution as $\A$-modules, so the homology of \eqref{eq:respiK} is $ \Ext_{\A}^{s,t}(H^*(X), \Z/2 )$.
\end{proof}

In general, the Adams spectral sequence does not exactly compute the homotopy groups of the spectrum $X$. However, under \fullref{ass:assX}, it does compute their \emph{$2$-completion}, a construction we review here.

\begin{defn}\label{defn:abcomp}
Let $G$ be an abelian group. For each $s\in \N$, let 
\[p_{s+1} \colon G/2^{s+1} \to G/2^{s}\] 
be the map induced by reduction modulo $2^{s}$. The \emph{$2$-completion} of $G$, denoted by $ G_2^{\wedge}$, is the inverse limit of $G/2^{s}$ along the maps $p_{s}$. That is,
\[  G_2^{\wedge} = \varprojlim_s G/2^s . \]
\end{defn}

\begin{rem}
Note that in the category of abelian groups, $\varprojlim_s G/2^s$ is isomorphic to the kernel of the map
\[  p \colon \prod_{s} G/2^s \to \prod_{s} G/2^s \]
where $p$ is the difference of the identity and the map to the product induced by the composites $ \prod_{s } G/2^s \to G/2^{k+1} \xra{p_{k+1}} G/2^k $. 
\end{rem}

\begin{ex}\label{ex:2adic}
If $G=\Z$, then 
\[\Z_2 := (\Z)^{\wedge}_2 = \varprojlim_s \Z/2^s  \]
are the $2$-adic integers. In general, if $G$ is a finitely generated abelian group, 
\[G_2^{\wedge}\cong G\otimes \Z_2.\]
In particular, if $G$ is a finite abelian $2$-group, then $G_2^{\wedge} \cong G$.
\end{ex}

\begin{thm}\label{thm:assconv}
Let $X$ satisfy \fullref{ass:assX}. Then the Adams spectral sequence for $X$ computes the $2$-completion of the homotopy groups of $X$. That is, the spectral sequence converges to $(\pi_{*}X)^{\wedge}_{2}$:
\begin{align*}
\Ext_{\A}^{s,t}(H^*(X) , \Z/2  ) \Longrightarrow (\pi_{t-s}X)^{\wedge}_{2}
\end{align*}
\end{thm}

\begin{rem}\label{rem:compsepctra}
In fact, the Adams spectral sequence for $X$ satisfying \fullref{ass:assX} computes the homotopy groups of a spectrum $X_2^{\wedge}$ that can be obtained using a construction analogous to completion for abelian group, and which has the property that $\pi_*(X_2^{\wedge})  \cong  (\pi_*X)^{\wedge}_2$.
In broad strokes, we define $X/2^{s} \in \mathrm{hSp} $ via the exact triangle:
\[\xymatrix{ X \ar[r]^-{2^{s}}  & X \ar[r] &  X/2^{s}  \ar[r] &  \Sigma X.}\]
There are induced maps $p_{s+1}  \colon X/2^{s+1} \to X/2^s$ and we define $\varprojlim_s X/2^s \in \mathrm{hSp}$ by the exact triangle
\[ \xymatrix{ \varprojlim_s X/2^s \ar[r] &  \prod_{s} X/2^s \ar[r]^-{p} & \prod_{s} X/2^s  \ar[r] & \Sigma \varprojlim_s X/2^s} \]
where $p$ is the difference of the identity and the map to the product induced by the composites $ \prod_{s } X/2^s \to X/2^{k+1} \xra{p_{k+1}} X/2^k $. This is called the \emph{homotopy inverse limit}. For a spectrum $X$ that satisfies \fullref{ass:assX}, then
\[ X_2^{\wedge} = \varprojlim_s X/2^s. \]
We refer the reader to Bousfield \cite[Section 2]{bousfield_lochom} and Ravenel \cite[II.2.1]{ravgreen} for more details on this and related topics. 
\end{rem}

\subsection{Using the Adams spectral sequence}\label{sec:usingASS}
In this section, we continue to assume that $X$ satisfies \fullref{ass:assX}, so that the Adams spectral sequence for $X$ computes the $2$-completion of the homotopy groups of $X$.

Computing Adams differentials is ``an art not a science''. There is no algorithm for determining them in general and it usually is a theorem when one computes a new differential in a spectral sequence of interest. However, there are rules to the game and the goal of this section is to share some of the tricks of the trade.

First, the Adams spectral sequence is depicted in an Adams chart as in \fullref{fig:ASSexample}. 
In this grading, a $d_r$ differential increases $s$ by $r$ and decreases $t-s$ by $1$. If $d_r(x)=y$, we say that $x$ hits, or \emph{kills} $y$. The class $x$ is the \emph{source} and $y$ the \emph{target} of the differential. A class which is in the the kernel of $d_r$ for every $r$ is called a \emph{permanent cycle}. A class which is hit by a differential is called a \emph{boundary}. We say that $x$ \emph{survives} if it is a permanent cycle, but not a boundary. 

The Adams spectral sequence for $X$ is a module over the Adams spectral sequence for $S^0$. From this, it follows in particular that the differentials are $h_0$ and $h_1$-linear. That is, $d_r(h_ix)=h_i d_r(x)$ for $i=0,1$.

We draw each page $E_r^{*,*}$ of the spectral sequence in subsequent Adams charts, erasing pairs of classes $x$ and $y$ that are connected by a differential $d_r(x)=y$ as we ``turn the pages''. Letting the process go to infinity, or stopping when there are no possible differentials left, we get the last page, called $E_{\infty}^{*,*}$. The last page of the spectral sequence contains the information for $(\pi_*X)^{\wedge}_{2}$ in the form of an \emph{associated graded}. That is, there is a filtration
\[ (\pi_{t}X)^{\wedge}_{2} = F_{\infty}^{0,t} \supseteq F_{\infty}^{1,t+1} \supseteq F_{\infty}^{2,t+2} \supseteq \ldots    \] 
related to $E_{\infty}^{*,*}$ by exact sequences
\begin{equation}\label{eq:extASS}
0 \to  F_{\infty}^{s+1,t+s+1} \to   F_{\infty}^{s,t+s} \to E_{\infty}^{s,t+s} \to 0 . \end{equation}
So, each box in the $t-s$ column of the Adams chart at $E_{\infty}^{*,*}$ is a subquotient of the answer $ (\pi_{t}X)^{\wedge}_{2}$ and the last problem is to reassemble them together.  This is called \emph{solving the extensions}, where the word ``extension'' refers to \eqref{eq:extASS}. We solve for $F_{\infty}^{0,t}/F^{s,t+s}$ inductively, starting with
\[0 \to  E_{\infty}^{1,t+1} \to   F_{\infty}^{0,t}/F_{\infty}^{2,t+2} \to E_{\infty}^{0,t} \to 0  \]
and continuing on to 
\[ 0 \to  E_{\infty}^{s,t+s} \to F_{\infty}^{0,t}/F_{\infty}^{s+1,t+s+1} \to F_{\infty}^{0,t}/F_{\infty}^{s,t+s} \to 0.    \]
We take the inverse limit once all of the terms $F_{\infty}^{0,t}/F^{s,t+s}$ have been determined. 

An element $a \in E_{\infty}^{s,t+s}$ only represents a class $\alpha \in  (\pi_{t}X)^{\wedge}_{2}$ modulo elements of \emph{higher filtration}. The language used is that $a$ \emph{detects} the element $\alpha$. Note that if $a$ detects $\alpha$, then it detects any class $\alpha + \beta$ where $\beta \in F_{\infty}^{s+1,t+s+1}$, so a class $a$ may detect multiple elements.

For the Adams spectral sequence, if $a \in E_{\infty}^{s,t+s}$ and $b \in E_{\infty}^{s+1, t+s+1}$ are such that $h_0a=b$, then $a$ detects an element $\alpha$ and $b$ detects an element $\beta$ such that $2\alpha = \beta$. So multiplication by $h_0$ records multiplication by $2$ and corresponds to a non-trivial, but easy to detect, extension as it comes from the module structure of the $E_2$-page. However, there can be non-trivial extensions coming from multiplications that do not come from the algebraic structure of the $E_2$-page. These are called \emph{exotic extensions}. For example, if $2\zeta = \omega$, for $\zeta \in F_{\infty}^{s,t+s}$ and $\omega \in F_{\infty}^{s+\epsilon,t+s+\epsilon}$ where $\epsilon >1$. Then $\zeta$ will be detected by some $z \in E_{\infty}^{s,t+s}$ and $\omega$ will be detected by some $w \in E_{\infty}^{s+\epsilon,t+s+\epsilon}$ so that these two classes are too far apart to be connected by an $h_0$. These situations are illustrated in \fullref{fig:ASSexample}.

If there are no non-trivial differentials, we say that the spectral sequence \emph{collapses}. We say that it \emph{collapses at $E_r$} if $E_r^{*,*}=E_{\infty}^{*,*}$. Often, there will be no possibilities for non-trivial differentials as the target of any possible differential will be zero. In this case, we say that the spectral sequence \emph{collapses for degree reasons}, or \emph{is too sparse for differentials}. Finally, if there are no possibilities for exotic extensions because no two classes on the $E_{\infty}$-page are aligned in a way that would allow for one to exist, we again say that there are \emph{no exotic extensions for degree reasons} or that the spectral sequence is \emph{too sparse for exotic extensions}. These are the best of all possible scenarios since differentials are hard to compute and exotic extensions are hard to solve. We will be in this situation in all of the examples in \fullref{sec:examples}.

\begin{ex}
A typical example of solving extensions is when a column consists of a single $h_0$-tower, say starting in $E_{\infty}^{0,t}$. Then
\[\xymatrix{ 0 \ar[r] & E_{\infty}^{s,t+s} \ar[r] \ar@{=}[d] & F_{\infty}^{0,t}/F_{\infty}^{s+1,t+s+1} \ar[r]  \ar@{=}[d] &  F_{\infty}^{0,t}/F_{\infty}^{s,t+s} \ar[r]\ar@{=}[d] & 0 \\
 0 \ar[r] & \Z/2 \ar[r] & \Z/2^{s+1} \ar[r] &  \Z/2^{s} \ar[r] & 0   }\]
 and $ F_{\infty}^{0,t}/F_{\infty}^{s,t+s}  \cong \Z/2^s\Z$ for all $s$. So
 \[  (\pi_{t}X)^{\wedge}_{2}  \cong \varprojlim_{s}  F_{\infty}^{0,t}/F_{\infty}^{s,t+s} \cong  \varprojlim_{s} \Z/2^s \cong \Z_2 \]
 where $\Z_2$ are the $2$-adic integers defined in \fullref{ex:2adic}.
 \end{ex}

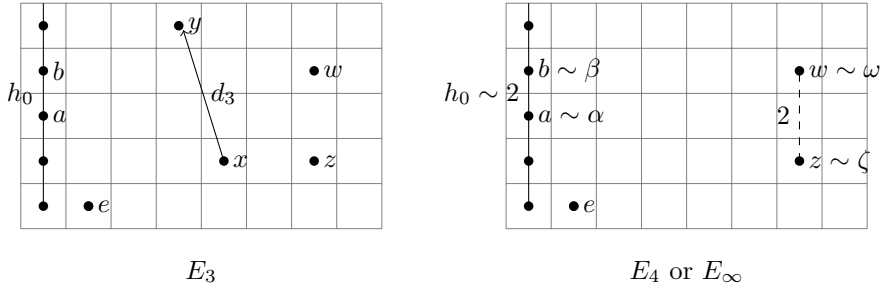
\begin{figure}[ht]
\center

  \begin{tikzpicture}[scale=.6]
    \draw[step=1cm,gray,very thin] (0,0) grid (8,5);
    \fill (0.5,0.5) circle (3pt) ;
        \fill (1.5,0.5) circle (3pt) ;
        \draw (0.5,0.5) -- (0.5,1.5);
           \fill (0.5,1.5) circle (3pt);
                  \draw (0.5,1.5) -- (0.5,2.5);
           \fill (0.5,2.5) circle (3pt) node[anchor=west] {$a$};
                    \draw (0.5,2.5) -- (0.5,3.5);
           \fill (0.5,3.5) circle (3pt) node[anchor=west] {$b$};
                        \draw (0.5,3.5) -- (0.5,4.5);
           \fill (0.5,4.5) circle (3pt);
                            \draw (0.5,4.5) -- (0.5,5);

                 \node[anchor=east] at (0.5,3) {$h_0$};
               
                \fill (4.5,1.5) circle (3pt) node[anchor=west] {$x$} ;
                          \fill (3.5,4.5) circle (3pt) node[anchor=west] {$y$};

                                      \draw[->] (4.5,1.5) -- (3.6,4.4);
          
            \fill (6.5,1.5) circle (3pt) node[anchor=west] {$z$};
                \fill (6.5,3.5) circle (3pt) node[anchor=west] {$w$};
                            \node[anchor=west] at (4,3) {$d_3$};   
                            
                              \node[anchor=north] at (4,-0.5) {$E_3$};   
                      \fill (1.5,0.5) circle (3pt) node[anchor=west] {$e$} ;
        
 \end{tikzpicture} \  \ \ \ \ 
\begin{tikzpicture}[scale=.6]
    \draw[step=1cm,gray,very thin] (0,0) grid (8,5);
    \fill (0.5,0.5) circle (3pt);
      \fill (1.5,0.5) circle (3pt) ;
        \draw (0.5,0.5) -- (0.5,1.5);
           \fill (0.5,1.5) circle (3pt);
                  \draw (0.5,1.5) -- (0.5,2.5) node[anchor=west] {$a\sim \alpha$};
           \fill (0.5,2.5) circle (3pt);
                    \draw (0.5,2.5) -- (0.5,3.5) node[anchor=west] {$b\sim \beta$};
           \fill (0.5,3.5) circle (3pt);
                        \draw (0.5,3.5) -- (0.5,4.5);
           \fill (0.5,4.5) circle (3pt);
                            \draw (0.5,4.5) -- (0.5,5);

            \fill (6.5,1.5) circle (3pt) node[anchor=west] {$z \sim \zeta$};
               \draw[dashed] (6.5,1.5) -- (6.5,3.5);
                \fill (6.5,3.5) circle (3pt) node[anchor=west] {$w\sim \omega$};
           \node[anchor=east] at (0.5,3) {$h_0 \sim 2$};  
              \node[anchor=east] at (6.5,2.5) {$2$};   
                    \fill (1.5,0.5) circle (3pt) node[anchor=west] {$e$} ;
            
                     \node[anchor=north] at (4,-0.5) {$E_4$ or $E_{\infty}$};   
 \end{tikzpicture} 
 
 \caption{Some phenomena in an Adams spectral sequence. The left chart is an example of an $E_3$-page and the right is the corresponding $E_4$-page, which in this case would be the $E_{\infty}$-page as there is no possibilities for further differentials. (The class $e$ cannot support a $d_r$ differential to the $h_0$-tower since this would violate the $h_0$-linearity of the differentials.) }
\label{fig:ASSexample}
\end{figure}

% !TEX root = cbms-master.tex

\section{Examples from the classification problems}\label{sec:examples}

In this section, we work out examples to illustrate the methodology. First, some notation. In \cite{FH}, Freed and Hopkins give a uniform classification of fermionic symmetric groups (\cite[9.2]{FH}) in spacetime dimension $n$. There are two complex symmetry groups, denoted $H^c_n(s)$, and labelled by $s = 0, 1$ and eight real symmetry groups, denoted $H_n(s)$, and labelled by $s = 0, \pm 1, \pm 2, \pm 3, 4$. They also show \cite[2.12]{FH} that in each case there are maps $H_n (s) \hookrightarrow H_{n+1}(s)$ stabilizing the groups, so that it makes sense to speak of $H(s)$ and $H^c(s)$ (this is precisely analogous to how $O(n)$ stabilizes to $O$). The Madsen-Tillman spectra (see \fullref{sec:thomspectra}) $MTH(s)$ are the cobordism theory of manifolds with stable tangential $H(s)$-structure. It is this cobordism theory that features in the Freed-Hopkins classification. This section will be devoted to computing the low dimensional homotopy groups of these cobordism spectra.

In \cite{FH}, Freed and Hopkins produce the tables of \fullref{fig:tablesum}. The explanations in \cite{FH} are brief and some steps are left as exercises. In \cite{campbell}, one of the authors gave a detailed explanation of the computation for $MT\Pin^-$, $MT\Pin^{+}$, $MT\Pin^{\widetilde{c}-}$, $MT\Pin^{\widetilde{c}+}$ and $MTG^{+}$. For this reason, we choose to apply the methods to explain the computations for $MTG^0$, $MTG^{-}$, $MT\Spin^c$ and $MT\Pin^c$, although we start by reproducing the computation for $MTG^{+}$ as a warm-up.

\begin{figure}[ht]
\center
    \begin{tabular}{|c|c|c|ccccc|}
      \hline
      $\ast$ & $s$ & $X(H(s))$ & $\pi_0$ & $\pi_1$ & $\pi_2$ & $\pi_3$ & $\pi_4$ \\\hline \hline
      $M\Spin$& $0$ & $MO_0 \simeq S^0$ &$\Z$   & $\Z/2$   &   $\Z/2$   &    $0$   &   $\Z$  \\
       $MT\Pin^-$ &$+1$ & $\Sigma^{-1}MO_1$ & $\Z/2$ & $\Z/2$   & $\Z/8$   &  $0$  & $0$ \\
    $MT\Pin^+$&$-1$ & $\Sigma MTO_1$ &$\Z/2$& $0$  &  $\Z/2$ &  $\Z/2$ & $\Z/16$    \\
    $MT\Pin^{\tilde{c}-}$ &$+2$ & $\Sigma^{-2}MO_2$ &$\Z/2$   & $0$   & $\Z \times \Z/2$   & $0$  & $\Z/2$ \\
    $MT\Pin^{\tilde{c}+}$&$-2$ & $\Sigma^2MTO_2$  &$\Z/2$   & $0$   & $\Z$  &$\Z/2$   & $(\Z/2)^3$  \\
    $MTG^+$&$+3$ & $\Sigma^{-3}MO_3$  & $\Z/2$  & $0$    & $\Z/2$  & $0$   & $\Z/2 \times \Z/4$ \\
    $MTG^-$&$-3$ & $\Sigma^{3}MTO_3$ &  $\Z/2$    &   0      &   $\Z/2$      &  0     &    $(\Z/2)^3$   \\
    $MTG^0$& $+4$ & $\Sigma^{-3}MSO_3$ &  $\Z$ &  $0$  &  $0$  &  $0$ &     $\Z^2$     \\
    \hline
  \end{tabular}

\bigskip
  
    \begin{tabular}{|c|c|c|ccccc|}
    \hline
      $\ast$ & $s$ & $X(H^c(s))$ & $\pi_0$ & $\pi_1$ & $\pi_2$ & $\pi_3$ & $\pi_4$ \\\hline \hline
     $MT \Spin^c$ & $0$ & $\Sigma^{-2}MU_1$ & $\Z$ & $0$ & $\Z$ & $0$ & $\Z^2$ \\
     $MT \Pin^c$ & $1$ & $\Sigma^{-3}MU_1 \smsh MO_1$ & $\Z/2$ & $0$ & $\Z/4$ & $0$ & $\Z/8 \times \Z/2$ \\
     \hline
   \end{tabular}
   \caption{The various real (top) and complex (bottom) symmetry groups studied in \cite{FH}.}
   \label{fig:tablesum}
  \end{figure}

\subsection{Reducing to computations over $\A_1$}
Computations of $\A$ are in general difficult to perform without computer assistance. However, if one can reduce the computation to one over $\A_1 $, constructing minimal resolutions becomes rather straightforward and computations can be done by hand, at least in some range.

The key to making the shift from computations over $\A$ to computations over $\A_1 $ is the fact that the spectra $MTH$ defined above satisfy
\[MTH \simeq \MSpin  \smsh X(H) \]
where $X(H)$ are the Thom spectra of certain familiar vector bundles \cite[10.7]{FH}. The values of $X(H)$ for the groups $H$ studied in \cite{FH} are given in \fullref{fig:tablesum}.

Since our cohomology is with field coefficients, namely ${\Z}/2$, the K\"unneth formula gives an isomorphism
\[H^*(MTH) \cong H^*(\MSpin) \otimes_{\Z/2} H^*(X(H)).\]
The key steps in the reductions of computations to an $\A_1 $-module problem is the following theorem.
\begin{thm}[Anderson, Brown, Peterson]\label{thm:mspinko}
There is an isomorphism
\[H^*(\MSpin) \cong \A \otimes_{\A_1 }(\Z/2  \oplus M)\]
where $M$ is a graded $\A_1 $-modules which is zero in degrees $t<8$. 
\end{thm}
As a consequence of \fullref{thm:mspinko} and \fullref{rem:range}, we have:
\begin{cor}
There is an isomorphism
\[\Ext_{\A}^{s,t}(H^*(\MSpin \smsh X(H)), \Z/2 ) \cong \Ext_{\A_1 }^{s,t}(H^*(X(H)),\Z/2 ) \]
if $t-s < 8$. 
\end{cor}
So low dimensional computations can be done over $\A_1$. We go through the following steps to compute $\pi_tMTH$ for $0\leq t\leq 4$:
\begin{enumerate}[(1)]
\item Compute $H^*(X(H))$ as modules over the $\A_1$. See \fullref{sec:compA1}.
\item Compute $\Ext_{\A_1}^{s,t}(H^*(X(H)), \Z/2 )$ in the range $t-s \leq 5$. See \fullref{sec:MinRes}, \fullref{sec:COR} and \fullref{sec:LES}.
\item Compute the differentials and extensions. See \fullref{sec:usingASS}. In all of our examples, the spectral sequences are too sparse for differentials and exotic extensions and this step is trivial.
\item Read off $\pi_*MTH$.
\end{enumerate}

We will do this one example at a time.

\subsection{The case $s=3$}
This is the case of $H=G^{+}= \Pin^{+} \times_{\{\pm 1\}} SU_2$ and in this case,
\[MTG^{+}  \simeq \MSpin \smsh \Sigma^{-3} MO_3 .\]

This example was stated in \cite{FH} and explicitly computed in \cite{campbell}. The cohomology of $H^*(\Sigma^{-3}MO_3)$ is illustrated in \fullref{fig:HMO3}. Let $R_3$ be the $\A_1$-module depicted in \fullref{fig:R3}, so that $R_3$ sits in an exact sequence
\[ 0 \to \Sigma Q \to R_3 \to M_{\infty} \to 0.\] 
From \fullref{fig:HMO3}, we have that
\[ H^*(\Sigma^{-3}MO_3) \approx R_3 \oplus \Sigma^{2}\A_1 \oplus \Sigma^{4}\A_1 \oplus \Sigma^5 \A_1 \]
where we will use $\approx$ to denote that there is an isomorphism in the range necessary for computations of homotopy groups in degrees less than or equal to $4$. We include the column $t-s=5$ to preclude the possibility of incoming differentials into the column $t-s=4$.

To compute the $E_2$-page of the spectral sequence
\[\Ext_{\A}^{s,t}(H^*(MTG^{+}), \Z/2 )  \approx \Ext_{\A_1}^{s,t}(H^*(\Sigma^{-3}MO_3), \Z/2 ) \Rightarrow \pi_{t-s}MTG^{+} , \]
we have to compute $\Ext_{\A_1}^{*,*}(R_3, \Z/2 )$. \fullref{fig:R3} and \fullref{fig:R3achart} illustrate this computation.

The Adams spectral sequence computing $\pi_*MTG^{+}$ is depicted in \fullref{fig:ASSMO3}. The spectral sequence is too sparse for differentials and exotic extensions, so the homotopy groups are
\begin{align*}
      \pi_0 MTG^+ &= \Z/2 
      \\  \pi_1 MTG^+ &= 0 \\
       \pi_2 MTG^+ &= \Z/2 \\
      \pi_3 MTG^+ &= 0 \\
        \pi_4 MTG^+ &= \Z/2 \times \Z/4. 
\end{align*}

\begin{figure}[!htb]
\centering
\begin{minipage}{.4\textwidth}
\centering \captionsetup{width=0.9\linewidth}
\begin{tikzpicture}[scale=.5] 
  \fill (-4, 1) circle (3pt);
  \sqtwoL(-4,1,black);
  \fill (-4, 3) circle (3pt);
  \sqone(-4,3,black)
  \fill (-4, 4) circle (3pt);

  \fill (0, 0) circle (3pt);
  \fill (0, 1) circle (3pt);
  \fill (0, 2) circle (3pt);
  \fill (0, 3) circle (3pt);
  \fill (0, 4) circle (3pt);
  \fill (0, 5) circle (3pt);
  \fill (0, 6) circle (3pt);
  \fill (0, 7) circle (3pt);
  \fill (0, 8) circle (3pt);
  \fill (0, 9) circle (3pt); 
  \fill (2, 3) circle (3pt);
  \fill (2, 4) circle (3pt);
  \sqone (0, 0, black);
  \sqtwoL (0, 0, black);
  \sqone (0, 2, black);
  \sqtwoL (0, 3, black);
  \sqone (0, 4, black);
  \sqtwoR (0, 4, black);
  \sqone (0, 6, black);
  \sqtwoR (0, 7, black);
  \sqone (0, 8, black); 
  \sqtwoCR (0, 1, black);
  
    \sqtwoCR (0, 2, black);
  \sqone (2, 3, black);
  
    \foreach \y in {0, 2, 3, 4, 5, 6, 7, 8,9}
           {\fill (4, \y) circle (3pt);}
           \sqtwoL (4, 0,black);
           \sqtwoL (4, 3, black);
           \sqtwoR (4, 4, black);
           \sqtwoR (4, 7, black);
           \sqone (4, 2, black);
           \sqone (4,4, black);
           \sqone (4,6, black); 
             \sqone (4,8, black);

   \draw[blue,->] (-4, 1) .. controls (-2, 0.75) ..  (0, 1);
     \draw[blue,->] (-4, 3)  .. controls (-1, 2.5) ..  (2, 3); 
          \draw[blue,->] (-4, 4) .. controls (-1, 3.5) .. (2, 4);

  \node[anchor=north] at (-4,0.5) {$Q$};
    \node[anchor=north] at (0,-0.5) {$R_3$};
        \node[anchor=north] at (4,-0.5) {$M_{\infty}$};
\end{tikzpicture}
\caption{An $\A_1$-module we call $R_3$.}
\label{fig:R3}
\end{minipage}
\begin{minipage}{.55\textwidth}
\centering
\centering \captionsetup{width=0.9\linewidth}
  \begin{tikzpicture}[scale=.5]
    \draw[step=1cm,gray,very thin] (0,0) grid (6,8);
    %the black parts, i.e. the homotopy groups
    \fill[green] (0.5,0.5) circle (3pt);
    \fill[green] (4.5,0.5) circle (3pt);
    \fill[green] (4.5,1.5) circle (3pt);
    \draw[green] (4.5,0.5) -- (4.5,1.5);
    \foreach \y in {1,2,3,4,5,6,7}
             {\fill[green] (0.5,\y+.5) circle (3pt);}
    \foreach \y in {1,2,3,4,5,6,7}
             {\draw[green] (0.5,\y-.5)--(.5,\y+.5);}
     \foreach \y in {0,1,2,3,4,5,6,7}
             {\fill[blue] (1.5,\y+.5) circle (3pt);}
     \foreach \y in {1,2,3,4,5,6,7}
              {\draw[blue] (1.5,\y-.5)--(1.5,\y+.5);}
     \foreach \y in {1, 2, 3, 4, 5, 6}
              {\draw[green] (4.5,\y+.5)--(4.5,\y+1.5);}
     \foreach \y in {2,3,4,5,6,7}
              {\fill[green] (4.5,\y+.5) circle (3pt);}
     \foreach \y in {1,2,3,4,5,6,7}
              {\fill[blue] (5.5,\y+.5) circle (3pt);}
     \foreach \y in {1,2,3,4,5,6}
              {\draw[blue] (5.5,\y+.5) -- (5.5,\y+1.5);}
     \draw[blue] (5.5,1.5) -- (6,2);
     \foreach \y in {0,1,2,3,4,5,6}
              {\draw[red] (1.5,\y+.5) -- (0.5,\y+1.5);}
     \foreach \y in {1,2,3,4,5,6}
              {\draw[red] (5.5,\y+.5) -- (4.5,\y+1.5);}
              
\draw[red] (1.5,7+.5) -- (1,8);
\draw[red] (5.5,7+.5) -- (5,8);
\draw[green] (0.5,7.5) -- (0.5,8);
\draw[green] (4.5,7.5) -- (4.5,8);
\draw[blue] (1.5,7.5) -- (1.5,8);
\draw[blue] (5.5,7.5) -- (5.5,8);

\draw[black] (0.5,0.5) circle (5pt);
\draw[black] (4.5,0.5) circle (5pt);
\draw[black] (4.5,1.5) circle (5pt);

  \end{tikzpicture}
  \caption{The computation of the Adams chart for $\Ext_{\A_1}^{s,t}(R_3, \Z/2 )$ using the exact sequence of \fullref{fig:R3}. The Adams chart for $Q$ is given in \fullref{fig:R2JQachart} and $M_{\infty}$ is discussed in \fullref{ex:M1}.}
  \label{fig:R3achart}
  \end{minipage}
  
  \bigskip
        \center
  \begin{tikzpicture}[scale=.5]
    \draw[step=1cm,gray,very thin] (0,0) grid (6,3);
    \fill (0.5, 0.5) circle (3pt);
    \fill (2.5, 0.5) circle (3pt);
    \fill (4.25,0.5) circle (3pt);
    \fill (4.75,0.5) circle (3pt);
    \sqone(4.75, 0.5, black);
        \fill (4.75,1.5) circle (3pt);
      \fill (5.5, 0.5) circle (3pt); 
  \end{tikzpicture}
  \caption{The Adams chart for $\Ext_{\A_1}^{s,t}(H^*(\Sigma^{-3}MO_3), \Z/2 )$.}
    \label{fig:ASSMO3}
    \end{figure}

\subsection{The case $s=-3$}
This is the case of $H=G^{-}= \Pin^{-} \times_{\{\pm 1\}} SU_2$ and in this case,
\[MTG^{-}  \simeq \MSpin \smsh \Sigma^{3} MTO_3. \]

In the degrees relevant for us, the $\mathcal{A}_1$-module structure of $H^\ast (\Sigma^3 MTO_3)$ is given in \fullref{fig:HMTO3}. We have that 
\[
H^\ast (\Sigma^3 MTO_3) \approx \mathcal{A}_1 \oplus \Sigma^2 R_0 \oplus \Sigma^4 \mathcal{A}_1 \oplus \Sigma^4 \mathcal{A}_1 \oplus \Sigma^{5}R_5 
\]
where $R_0$ is the module depicted in \fullref{fig:R0} and $R_5$ the module depicted in \fullref{fig:R5}.  The module $R_5$ sits in a short exact sequence of $\mathcal{A}_1$-modules  (pictured in \fullref{fig:R5}): 
\[
0 \to J  \to R_5 \to \Sigma M_\infty  \to 0 .
\]
\fullref{fig:R0achart} gives $\Ext^{\ast, \ast}_{\A_1}(R_0, \Z/2 )$ and \fullref{fig:R5achart} gives $\Ext_{\A_1}^{\ast, \ast}(R_5, \Z/2 )$.
The Adams chart for $\Ext_{\A_1}^{s,t}(H^*(\Sigma^{-3}MTO_3), \Z/2 )$ is depicted in \fullref{fig:ASSMTO3} in the range of interest. The spectral sequence is too sparse for differentials and exotic extensions and the homotopy groups of $MTG^{-}$ are
      \begin{align*}
        \pi_0 MTG^{-} &= \Z/2 \\
        \pi_1 MTG^{-} &= 0 \\
        \pi_2 MTG^{-} &= \Z/2 \\
        \pi_3 MTG^{-} &= 0 \\
        \pi_4 MTG^{-} &= (\Z/2)^3.
      \end{align*}

\begin{figure}[!htb]
\centering
\begin{minipage}{.4\textwidth}
\centering \captionsetup{width=0.9\linewidth}
\begin{tikzpicture}[scale=.5]

\joker(13,5);
         
         \fill (16,6) circle (3pt) ;
             \sqone(16,6,black); 
             \sqtwoL(16,6,black);
               \fill (16,7) circle (3pt)  ;
                 \fill (16,8) circle (3pt)  ;
                   \sqone(16,8,black); 
                       \fill (16,9) circle (3pt) ;
                             \sqtwoL(16,9,black);
                            \fill (16,10) circle (3pt)  ;
                               \sqone(16,10,black); 
                               \sqtwoR(16,10,black);
                                   \fill (16,11) circle (3pt)   ;
                                    \fill (16,12) circle (3pt) ;
                                     \sqone(16,12,black); 

   \fill (18,5) circle (3pt) ;
   \sqtwoCL(18,5, black);
   \sqone(18,5,black);
      \fill (18,6) circle (3pt) ;
         \sqtwoR(18,6, black);
           \fill (18,8) circle (3pt) ;
              \sqone(18,8,black);
                \fill (18,9) circle (3pt) ;
                
                   \sqtwoCR(16,7, black);

    \foreach \y in {0, 2, 3, 4, 5, 6, 7, 8,9}
           {\fill (21, \y+6) circle (3pt);}
           \sqtwoL (21, 6,black);
           \sqtwoL (21, 9, black);
           \sqtwoR (21, 10, black);
           \sqtwoR (21, 13, black);
           \sqone (21, 12, black);
           \sqone (21,10, black);
                   \sqone (21,8, black);
           \sqone (21,12, black); 
             \sqone (21,14, black); 

              \draw[blue,->] (13, 5) .. controls (15.5, 4.5) ..  (18, 5);      
                \draw[blue,->] (13, 6) .. controls (15.5, 5.5) ..  (18, 6);  
                 \draw[blue,->] (13, 8) .. controls (15.5, 7.5) ..  (18, 8);  
                 \draw[blue,->] (13, 9) .. controls (15.5, 8.5) ..  (18, 9);  
                 \draw[blue,->] (13, 7) .. controls (14.5, 6.75) ..  (16, 7);        
                 
                 \node[anchor=north] at (13, 4.5) {$J$};   
              \node[anchor=north] at (16, 4.5) {$R_5$};   
                \node[anchor=north] at (21, 4.5) {$M_{\infty}$};   
\end{tikzpicture}
                  \caption{An $\A_1$-module we call $R_5$.}
                  \label{fig:R5}
\end{minipage}
\begin{minipage}{.55\textwidth}
\centering
\centering \captionsetup{width=0.9\linewidth}

\begin{tikzpicture}[scale=.5] 
  \draw[step=1cm,gray,very thin] (0,0) grid (7,8);
  \foreach \y in {0, 1, 2, 3, 4, 5, 6, 7}
           {\fill[green] (1.5, \y+.5) circle (3pt);
             \fill[green] (5.5, \y+.5) circle (3pt);}
           \foreach \y in {0, 1, 2, 3, 4, 5, 6}
                    {\draw[green] (1.5, \y+.5) -- (1.5, \y+1.5);
                      \draw[green] (5.5, \y+.5) -- (5.5, \y+1.5);}
                    \draw[green] (1.5, 7.5) -- (1.5, 8);
                    \draw[green] (5.5, 7.5) -- (5.5, 8);

                    \fill[blue] (0.5,0.5) circle (3pt);
                    \foreach \y in {1, 2, 3, 4, 5, 6, 7}
                             {\fill[blue] (2.5, \y+.5) circle (3pt); }
                                                 \foreach \y in {1, 2, 3, 4, 5, 6}
                             {\draw[blue] (2.5, \y+.5) -- (2.5, \y+1.5);}
                             \draw[blue](2.5, 7.5)--(2.5, 8);
                          \foreach \y in {2, 3, 4, 5, 6, 7}
                             {\fill[blue] (6.5, \y+.5) circle (3pt); }
                                                 \foreach \y in {2, 3, 4, 5, 6}
                             {\draw[blue] (6.5, \y+.5) -- (6.5, \y+1.5);}
                             \draw[blue](6.5, 7.5)--(6.5, 8);      
                                  \draw[blue](6.5, 2.5)--(7, 3);      
                                             
      \foreach \x in {1,2,3, 4, 5, 6}
                   {\draw[red] (2.5, \x+.5) -- (1.5, \x+1.5);}
                                        \draw[red] (2.5, 7.5) -- (2, 8);    
                                        
                                              \foreach \x in {2,3, 4, 5, 6}
                   {\draw[red] (6.5, \x+.5) -- (5.5, \x+1.5);}
                                        \draw[red] (6.5, 7.5) -- (6, 8);    
\draw (0.5, 0.5) circle (5pt);
\draw (1.5, 0.5) circle (5pt);
\draw (1.5, 1.5) circle (5pt); 
\draw (5.5, 0.5) circle (5pt); 
\draw (5.5, 1.5) circle (5pt); 
\draw (5.5, 2.5) circle (5pt);                                      
\end{tikzpicture}
                  \caption{The computation of the Adams chart for $\Ext_{\A_1}^{s,t}(R_5, \Z/2 )$ using the exact sequence of \fullref{fig:R5}. The Adams chart for $J$ is given in \fullref{fig:R2JQachart} and $M_{\infty}$ is discussed in \fullref{ex:M1}.} 
                  \label{fig:R5achart}
\end{minipage}

\bigskip
      \center
  \begin{tikzpicture}[scale=.5]
    \draw[step=1cm,gray,very thin] (0,0) grid (6,3);
    \fill (0.5, 0.5) circle (3pt);
    \fill (2.5, 0.5) circle (3pt);
    \fill (4.25,0.5) circle (3pt);
    \fill (4.75,0.5) circle (3pt);
    \fill (4.5, 1.5) circle (3pt); 
      \fill (5.5, 2.5) circle (3pt); 
      \draw  (4.5, 1.5) -- (5.5, 2.5);
      
         \fill (5.5, 0.5) circle (3pt); 
  \end{tikzpicture}
  \caption{The Adams chart for $\Ext_{\A_1}^{s,t}(H^*(\Sigma^{3}MTO_3), \Z/2 )$.}
    \label{fig:ASSMTO3}
\end{figure}

\subsection{The case $s=4$}
This is the case of $H=G^{0}= \Spin \times_{\{\pm 1\}} SU_2$ and in this case,
\[MTG^{0}  \simeq \MSpin \smsh \Sigma^{-3} MSO_3.\]
The $\A_1 $-structure of $H^\ast (MSO_3)$ is depicted in \fullref{fig:HMSO3}, and
\[ H^\ast (\Sigma^{-3} MSO_3) \approx Q \oplus \Sigma^4 R_2 . \]
The Adams chart for the modules $Q$ and $R_2$ are depicted in \fullref{fig:R2JQachart}, and the Adams chart for $\Ext_{\A_1}^{s,t}(H^\ast (\Sigma^{-3} MSO_3), \Z/2 )$ is in \fullref{fig:ASSMSO3}. The spectral sequence is too sparse for differentials and exotic extensions and the homotopy groups of $MTG^0$ are
      \begin{align*}
        \pi_0 MTG^0 &= \Z\\
        \pi_1 MTG^0 &= 0\\
        \pi_2 MTG^0 &= 0\\
        \pi_3 MTG^0 &= 0\\
        \pi_4 MTG^0 &= \Z^2 .
      \end{align*}
      
            \begin{figure}[ht]
      \center
      \begin{tikzpicture}[scale=.5] 
        \draw[step=1cm,gray,very thin] (0,0) grid (6,5);
            \foreach \y in {0, 1, 2, 3, 4}
             {\fill (0.5, 0.5+\y) circle (3pt);}
             \foreach \y in {0, 1, 2, 3}
                      {\draw (0.5, 0.5+\y) -- (0.5, 1.5+\y);}
                                  \foreach \y in {0, 1, 2, 3, 4}
             {\fill (4.25, 0.5+\y) circle (3pt);}
             \foreach \y in {0, 1, 2, 3}
                      {\draw (4.25, 0.5+\y) -- (4.25, 1.5+\y);}
                                 \foreach \y in {0, 1, 2, 3}
             {\fill (4.75, 1.5+\y) circle (3pt);}
             \foreach \y in {0, 1, 2}
             {\draw (4.75, 1.5+\y) -- (4.75, 2.5+\y);}
             
               \draw (0.5,4.5) -- (0.5,5);
                  \draw (4.25,4.5) -- (4.25,5);
                   \draw (4.75,4.5) -- (4.75,5);
                   
                   \fill (5.75, 2.5) circle (3pt);
                    \draw (4.75,1.5) -- (5.75,2.5);
                    \draw (5.75,2.5) -- (6, 2.75) ;
                    
                      \fill (5.5, 0.5) circle (3pt);
                       \draw (5.5,0.5) -- (6, 1) ;
      \end{tikzpicture}
      \caption{The Adams chart for $\Ext_{\A_1}^{s,t}(H^\ast (\Sigma^{-3} MSO_3), \Z/2 )$.}
            \label{fig:ASSMSO3}
      \end{figure}
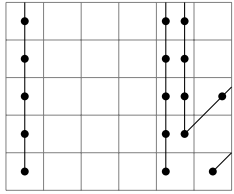

\subsection{The complex case $s=0$} 
This is the case of $H^c=\Spin^c$ and in this case,
\[MTH^c(0)  \simeq \MSpin \smsh \Sigma^{-2}  MU_1 . \]
The structure of $H^\ast (MU_1)$ as an $\mathcal{A}_1$-module is depicted in \fullref{fig:MO1MU1}.
It is given by shifted sums of $\A_1  /\!\!/  \E_1$, so
\[
H^\ast ( \Sigma^{-2} MU_1) \approx \A_1  /\!\!/  \E_1 \oplus \Sigma^4 \A_1  /\!\!/  \E_1
\]
In \fullref{ex:coneeta}, we calculated that
\[
\Ext_{\mathcal{A}_1}^{*,*} (\A_1  /\!\!/  \E_1, \Z/2 ) \cong \Ext_{\E_1}^{*,*} (\Z/2 , \Z/2 )  \cong \Z/2 [h_0, v_1]
\]
for $v_1$ in degree $(s,t)=(1, 3)$. The $E_2$-page of the Adams spectral sequence for $\pi_\ast MTH^c(0)$ is depicted in \fullref{fig:ASSMU1}. The spectral sequence is too sparse for differentials and exotic extensions and the homotopy groups of $MTH^c(0)$ are
\begin{align*}
  \pi_0 MTH^c(0) &= \Z \\
  \pi_1 MTH^c(0) &= 0 \\
  \pi_2 MTH^c (0) &= \Z \\
  \pi_3 MTH^c (0) &= 0\\
  \pi_4 MTH^c (0) &=(\Z)^2 
\end{align*}

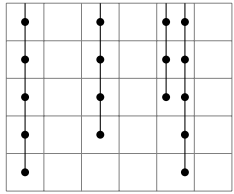
\begin{figure}[ht]
\center
    \begin{tikzpicture}[scale=.5]
          \draw[step=1cm,gray,very thin] (0,0) grid (6,5);
    \foreach \y in {0, 1,2,3,4}
             {\fill (.5,\y+.5) circle (3pt) ; }
                 \foreach \y in {0, 1,2,3}
             {\draw (.5,\y+.5)--(0.5,\y+1.5);}
                \draw (0.5,4.5) -- (0.5,5);
                
    \foreach \y in {1,2,3,4}
             {\fill (2.5,\y+.5) circle (3pt);}
                 \foreach \y in {1,2,3}
             {\draw (2.5,\y+.5)--(2.5,\y+1.5);}
               \draw (2.5,4.5) -- (2.5,5);
               
    \foreach \y in {1,2,3}
             {\fill (4.25,\y+1.5) circle (3pt);}
                 \foreach \y in {1,2}
             { \draw (4.25,\y+1.5)--(4.25,\y+2.5);}
                 \draw (4.25,4.5) -- (4.25,5);

               \foreach \y in {1,2,3,4,5}
             {\fill (4.75,\y-.5) circle (3pt);} 
                   \foreach \y in {1,2,3,4}
              {\draw (4.75,\y-.5)--(4.75,\y+.5);}
               \draw (4.75,4.5) -- (4.75,5);

    \end{tikzpicture}
    \caption{The Adams chart for $\Ext_{\A_1}^{s,t}(H^*(MU_1), \Z/2 )$.}
    \label{fig:ASSMU1}
\end{figure}

\subsection{The complex case $s=1$} 
This is the case of $H^c=\Pin^c$ and in this case,
\[MTH^c(1)  \simeq  \MSpin \smsh  \Sigma^{-3} MU_1\smsh  MO_1 . \]
The structure of $H^*(MU_1 \smsh MO_1)$ is depicted in \fullref{fig:MU1smshMO1}. We have
\[ H^*(\Sigma^{-3}MU_1 \smsh MO_1) \approx R_6 \oplus \Sigma^4 R_6   \]
for the module $R_6$ depicted in \fullref{fig:R6}. In order to compute the $E_2$-page of the Adams spectral sequence for $\pi_\ast MTH^c(1)$ we need to compute $\Ext^{\ast, \ast}_{\mathcal{A}_1} (R_6, \Z/2 )$. The module $R_6$ sits in a short exact sequence of $\mathcal{A}_1$-modules (pictured in \fullref{fig:R6}): 
\[
0 \to \Sigma R_1 \to R_6 \to M_\infty  \to 0
\]
and $\Ext^{\ast, \ast}_{\A_1} (R_6, \Z/2 )$ is computed in \fullref{fig:R6achart}.
The $E_2$-page of the Adams spectral sequence for $\pi_\ast MTH^c(1)$ is depicted in \fullref{fig:ASSMU1MO1}. The spectral sequence is too sparse for differentials and exotic extensions and the homotopy groups of $MTH^c(1)$ are
\begin{align*}
  \pi_0 MTH^c(1) &= \Z/2 \\
  \pi_1 MTH^c(1) &= 0 \\
  \pi_2 MTH^c(1) &= \Z/4 \\
  \pi_3 MTH^c(1) &= 0 \\
  \pi_4 MTH^c(1) &= \Z/2 \times \Z/8 .
  \end{align*}

\begin{figure}[ht]
\centering
\begin{minipage}{.4\textwidth}
\centering \captionsetup{width=0.9\linewidth}
\begin{tikzpicture}[scale=.5]
  \foreach \y in {0, 1, 2, 3, 4, 5, 6, 7,8} 
         {\fill (0, \y+1) circle (3pt);}
         \sqtwoL(0, 1, black);
         \sqtwoR(0, 2, black);
         \sqone (0, 2, black);
         \sqone (0, 4, black);
         \sqtwoR (0, 5, black);
         \sqone (0, 6, black);
         \sqtwoL (0, 6, black);
         \sqone (0, 8, black); 
         \foreach \y in {0,1, 2, 3, 4, 5, 6, 7, 8, 9}
                  {\fill (3, \y) circle (3pt);}
                  \sqtwoL(3, 0, black);
                  \sqone (3, 0, black); 
                  \sqone (3, 2, black);
                  \sqtwoL(3, 3, black);
                  \sqone (3, 4, black);
                  \sqtwoR (3, 4, black);
                  \sqone (3, 6, black);
                  \sqtwoR (3, 7, black);
                  \sqone (3, 8, black); 
                    \foreach \y in {0, 1, 2, 3, 4, 5, 6, 7} 
         {\fill (5, \y+2) circle (3pt);}
         \sqtwoCR (3, 1, black);
         \sqone (5, 2, black);
         \sqtwoR (5, 2, black);
         \sqone (5, 4, black);
         \sqtwoR (5, 5, black);
         \sqone (5, 6, black);
         \sqtwoL (5, 6, black);
         \sqone (5, 8, black);
                  \foreach \y in {0, 2, 3, 4, 5, 6, 7, 8, 9}
                  {\fill (8, \y) circle (3pt);}
                  \sqtwoL(8, 0, black);
                  \sqone (8, 2, black);
                  \sqtwoL(8, 3, black);
                  \sqone (8, 4, black);
                  \sqtwoR (8, 4, black);
                  \sqone (8, 6, black);
                  \sqtwoR (8, 7, black);
                  \sqone (8, 8, black);
                  \draw[blue,->] (0, 1) .. controls (1.5, 0.5) ..  (3, 1);
                  \draw[blue,->] (0, 2) .. controls (2.5, 1.5) .. (5, 2);
                  \draw[blue,->] (0, 3) .. controls (2.5, 2.5) .. (5, 3);
                  \draw[blue,->] (0, 4) .. controls (2.5, 3.5) ..  (5, 4);
                  \draw[blue,->] (0, 5) .. controls (2.5, 4.5) ..  (5, 5);
                  \draw[blue,->] (0, 6) .. controls (2.5, 5.5) ..  (5, 6);
                  \draw[blue,->] (0, 7) .. controls (2.5, 6.5) ..  (5, 7);
                  \draw[blue,->] (0, 8) .. controls (2.5, 7.5) ..  (5, 8);
                  \draw[blue,->] (0, 9) .. controls (2.5, 8.5) ..  (5, 9);         
\end{tikzpicture}
                  \caption{The exact sequence for $R_6$.}
                  \label{fig:R6}
\end{minipage}
\begin{minipage}{.55\textwidth}
\centering
\centering \captionsetup{width=0.9\linewidth}

\begin{tikzpicture}[scale=.5] 
  \draw[step=1cm,gray,very thin] (0,0) grid (6,8);
  \foreach \y in {0, 1, 2, 3, 4, 5, 6, 7}
           {\fill[green] (0.5, \y+.5) circle (3pt);
             \fill[green] (4.5, \y+.5) circle (3pt);}
           \foreach \y in {0, 1, 2, 3, 4, 5, 6}
                    {\draw[green] (0.5, \y+.5) -- (0.5, \y+1.5);
                      \draw[green] (4.5, \y+.5) -- (4.5, \y+1.5);}
                    \draw[green] (0.5, 7.5) -- (0.5, 8);
                    \draw[green] (4.5, 7.5) -- (4.5, 8);
                    \foreach \y in {0, 1, 2, 3, 4, 5, 6, 7}
                             {\fill[blue] (1.5, \y+.5) circle (3pt); }
                             \fill[blue] (2.5, .5) circle (3pt);
                             \fill[blue] (2.5, 1.5) circle (3pt);
                    \foreach \y in {2, 3, 4, 5, 6, 7}
                             {\fill[blue] (5.5, \y+.5) circle (3pt); }
                    \foreach \y in {0, 1, 2, 3, 4, 5, 6}
                             {\draw[blue] (1.5, \y+.5) -- (1.5, \y+1.5);}
                             \draw[blue](1.5, 7.5)--(1.5, 8);
                             \draw[blue] (2.5, .5) -- (2.5, 1.5);
                    \foreach \y in {2,3, 4, 5, 6}
                             {\draw[blue] (5.5, \y+.5) -- (5.5, \y+1.5);}
                             \draw[blue] (5.5,7.5)--(5.5,8);
                     \foreach \x in {0, 1, 2, 3, 4, 5, 6}
                                      {\draw[red] (1.5, \x+.5) -- (.5, \x+1.5);}
                      \foreach \x in {2,3, 4, 5, 6}
                      {\draw[red] (5.5, \x+.5) -- (4.5, \x+1.5);}
                      \draw[red] (1.5, 7.5) -- (1, 8);
                      \draw[red] (5.5, 7.5) -- (5, 8);
                      \draw (.5,.5) circle (5pt);
                      \draw (2.5, .5) circle (5pt);
                      \draw (2.5, 1.5) circle (5pt);
                      \draw (4.5, .5) circle (5pt);
                      \draw (4.5, 1.5) circle (5pt);
                      \draw (4.5, 2.5) circle (5pt);

\end{tikzpicture}
                  \caption{The computation of the Adams chart for $\Ext_{\A_1}^{s,t}(R_6, \Z/2 )$ using the exact sequence of \fullref{fig:R6}. The Adams chart for $R_1$ is given in \fullref{fig:R1achart} and $M_{\infty}$ is discussed in \fullref{ex:M1}.}
                  \label{fig:R6achart}
\end{minipage}
\bigskip

\begin{tikzpicture}[scale=.5] 
  \draw[step=1cm,gray,very thin] (0,0) grid (5,3);
  \fill (0.5, 0.5) circle (3pt);
  \fill (2.5, 0.5) circle (3pt);
  \fill (2.5, 1.5) circle (3pt);
  \sqone (2.5, 0.5, black);
  \fill (4.75, 0.5) circle (3pt);
  \fill (4.75, 1.5) circle (3pt);
  \fill (4.75, 2.5) circle (3pt);
  \sqone (4.75, .5, black);
  \sqone (4.75, 1.5, black);
  \fill (4.25, .5) circle (3pt);  
\end{tikzpicture}
\caption{The Adams chart for $\Ext_{\A_1}^{s,t}(H^*(MU_1\smsh MO_1), \Z/2 )$.}
\label{fig:ASSMU1MO1}

\end{figure}

\bibliographystyle{amsalpha}
%\bibliography{cbms-bib}
\providecommand{\bysame}{\leavevmode\hbox to3em{\hrulefill}\thinspace}
\providecommand{\MR}{\relax\ifhmode\unskip\space\fi MR }
\providecommand{\MRhref}[2]{%
  \href{http://www.ams.org/mathscinet-getitem?mr=#1}{#2}
}
\providecommand{\href}[2]{#2}




\end{document}